%% file: mainDec2.tex
\documentclass[12pt]{amsart}
\usepackage{graphicx}
\usepackage[all]{xy}
\input{macros}

\title[Constructing o-minimal structures with decidable theories] {Constructing o-minimal structures with decidable theories using generic families of functions from quasianalytic classes}

\author[D. J. Miller]{Daniel J. Miller}
\address{Emporia State University, Department of Mathematics, Computer Science and Economics, 1200 Commercial Street, Campus Box 4027, Emporia, KS 66801, U.S.A.}
\email{dmille10@emporia.edu}


\begin{document}
\maketitle

\begin{abstract}
Let $\RR_{\S}$ denote the expansion of the real ordered field by a family of real-valued functions $\S$, where each function in $\S$ is defined on a compact box and is a member of some quasianalytic class which is closed under the operations of function composition, division by variables, and extraction of implicitly defined functions.  It is shown that if the family $\S$ is generic (which is a certain technically defined transcendence condition), then the theory of $\RR_{\S}$ is decidable if and only if $\S$ is computably $C^\infty$ (which means that all the partial derivatives of the functions in $\S$ may be effectively approximated).  It is also shown that, in a certain topological sense, many generic, computably $C^\infty$ families $\S$ exist.
\end{abstract}

\input{intro}

\input{IFT}

\input{compHolom}

\input{parent}

\input{genDec}

\input{genDense}

\bibliographystyle{amsplain}
\bibliography{bibliotex}
\end{document}

%% file: macros.tex
\newtheoremstyle{slthm}
  {9pt}
  {5pt}
  {\slshape}
  {}
  {\bfseries}
  {.}
  {.5em}
  {\thmname{#1} \thmnumber{#2}{\rm \thmnote{ (#3)}}}

\newtheoremstyle{prcl}
  {9pt}
  {5pt}
  {\slshape}
  {}
  {\bfseries}
  {.}
  {.5em}
  {\thmname{#3} \thmnumber{ #2}}

\theoremstyle{slthm}
\newtheorem{theorem}{Theorem}[section]
\newtheorem{lemma}[theorem]{Lemma}

\newtheorem{proposition}[theorem]{Proposition}
\newtheorem{corollary}[theorem]{Corollary}

\newtheorem{introThm}{Theorem}

\theoremstyle{definition}
\newtheorem{definition}[theorem]{Definition}
\newtheorem{definitions}[theorem]{Definitions}

\newtheorem{example}[theorem]{Example}

\theoremstyle{remark}

\newtheorem{notation}[theorem]{Notation}

\newtheorem{remark}[theorem]{Remark}
\newtheorem{remarks}[theorem]{Remarks}

\newenvironment{renumerate}
        {
         \begin{enumerate}}
        {\end{enumerate}}

\numberwithin{equation}{section}

\allowdisplaybreaks[2]

\newcommand{\A}{\mathcal{A}}
\newcommand{\B}{\mathcal{B}}
\newcommand{\C}{\mathcal{C}}

\newcommand{\E}{\mathcal{E}}
\newcommand{\F}{\mathcal{F}}
\newcommand{\G}{\mathcal{G}}

\renewcommand{\L}{\mathcal{L}}   %

\renewcommand{\S}{\mathcal{S}}   %
\newcommand{\T}{\mathcal{T}}

\newcommand{\NN}{\mathbb{N}}

\newcommand{\QQ}{\mathbb{Q}}
\newcommand{\RR}{\mathbb{R}}
\newcommand{\CC}{\mathbb{C}}

\newcommand{\VV}{\mathbb{V}}
\newcommand{\II}{\mathbb{I}}

\renewcommand{\bar}[1]{\overline{#1}} %
\newcommand{\tld}[1]{\widetilde{#1}}
\newcommand{\lb}{\langle} 
\newcommand{\rb}{\rangle} 

\DeclareMathOperator{\td}{td}
\DeclareMathOperator{\rank}{rank}

\DeclareMathOperator{\Zar}{Zar}

\DeclareMathOperator{\dom}{dom}

\DeclareMathOperator{\im}{im}

\newcommand{\Restr}[1]{\big|_{#1}}
\DeclareMathOperator{\an}{an}

\newcommand{\PD}[3]{\frac{\partial^{#1}#2}{\partial {#3}^{#1}}}

\newcommand{\PDn}[3]{\frac{\partial^{|#1|}#2}{\partial {#3}^{#1}}}

\DeclareMathOperator{\Int}{int}
\DeclareMathOperator{\cl}{cl}
\DeclareMathOperator{\bd}{bd}

\DeclareMathOperator{\Th}{Th}

\DeclareMathOperator{\diam}{diam}
\DeclareMathOperator{\dist}{dist}

\DeclareMathOperator{\Ball}{Ball}
\DeclareMathOperator{\Gen}{Gen}
\DeclareMathOperator{\IF}{IF}
\DeclareMathOperator{\Comp}{Comp}

\DeclareMathOperator{\name}{name}

\DeclareMathOperator{\Strat}{Strat}

\DeclareMathOperator{\id}{id}

\DeclareMathOperator{\width}{width}
\DeclareMathOperator{\length}{length}


%% file: intro.tex
\section*{Introduction}\label{s:intro}

Tarski \cite{Tarski} proved that the theory of the real field is decidable.  In the same paper, he asked if the theory of the real exponential field is decidable.  Macintyre and Wilkie \cite{MW} proved that theory of the real exponential field is decidable if Schanuel's conjecture is true.  But since a proof of Schanuel's conjecture --- or a suitable replacement if the conjecture is false --- appears to be no easy feat, the following more basic question was still left open by Macintyre and Wilkie's work:
\begin{enumerate}
\item[($*$)]
Does there exist an o-minimal expansion of the real field with a decidable first-order theory which defines a transcendental function $f:\RR^n\to\RR$ for some $n > 0$?
\end{enumerate}
This paper shows that ($*$) has an affirmative answer.  This is accomplished by constructing expansions of the real field by families of functions from quasianalytic classes which satisfy a certain computability condition and also a genericity condition.  This proof shows that, in a certain topological sense, there are many proper o-minimal expansions of the real field with decidable theories.  But, it does not construct any natural examples of such structures.

The proof is based on a characterization of decidability proven by the author in \cite{DJMcharDec}, and we shall assume that the reader is familiar with \cite[Notation 0.2 and Sections 2-4]{DJMcharDec}.  This includes all of the concepts from computable analysis found in \cite[Sections 2 and 3]{DJMcharDec} and the section on IF-systems \cite[Sections 4]{DJMcharDec}.  (Some of this material on computable analysis will be reviewed, but not all of it.)  The other concepts needed from \cite{DJMcharDec} will be restated in Section \ref{s:parent}.

\subsection*{The Main Results}
We now work towards stating the first of our two main theorems.
Throughout the entire paper, we fix the following objects:
\begin{enumerate}{\setlength{\itemsep}{3pt}
\item
a quasianalytic IF-system $\C = \bigcup_{n\in\NN, r\in\QQ_{+}^{n}} \C_r$;\\
(Recall from \cite{DJMcharDec} that this roughly means that each $\C_r$ is a quasianalytic ring of real-valued functions on $[-r,r] = [-r_1,r_1]\times\cdots\times[-r_n,r_n]$, and $\C$ is closed under the operations of function composition, division by variables, and extraction of implicitly defined functions.  The prototypical example of such a $\C$ is when each $\C_r$ denotes the ring of all real-valued functions on $[-r,r]$ which extend to an analytic function in a neighborhood of $[-r,r]$.)

\item
a computable index set $\Sigma$, and computable maps $\eta:\Sigma\to\NN$ and $\rho:\Sigma\to\bigcup_{n\in\NN}\QQ^{n}_{+}$ such that $\rho(\sigma)\in\NN^{\eta(\sigma)}$ for all $\sigma\in\Sigma$.
}\end{enumerate}
Note that $\rho$ determines both $\Sigma$ and $\eta$.

Consider a family of functions $\S = \{S_\sigma\}_{\sigma\in\Sigma}$, where $S_\sigma\in\C_{\rho(\sigma)}$ for each $\sigma\in\Sigma$.  We call $\Sigma$ the {\bf index set} of $\S$, $\eta$ the {\bf arity map} of $\S$, and $\rho$ the {\bf domain map} of $\S$, since these three objects are used to index $\S$ and to specify the arity and the domains of each of the functions in $\S$.  We also call $\eta$ the {\bf arity map} of $\rho$.

\begin{definition}\label{def:RS}
Let $\RR_{\S}$ denote the expansion of the real ordered field by the family of functions $\{\widehat{S}_\sigma\}_{\sigma\in\Sigma}$, where each $\widehat{S}_\sigma:\RR^{\eta(\sigma)}\to\RR$ is defined from $S_\sigma$ by
\[
\widehat{S}_\sigma(x) = \begin{cases}
S_\sigma(x),   & \text{if $x\in[-\rho(\sigma),\rho(\sigma)]$,} \\
0,      & \text{if $x\in\RR^{\eta(\sigma)}\setminus[-\rho(\sigma),\rho(\sigma)]$,}
\end{cases}
\]
and let $\L_{\S}$ denote that language of the structure $\RR_{\S}$.
\end{definition}

It is shown in \cite{DJMcharDec} that the first-order theory of $\RR_{\S}$ is decidable if and only if two oracles, called the approximation and precision oracles for $\S$, are decidable.  Loosely stated, the approximation oracle for $\S$ allows one to approximate any partial derivative of any function in $\S$ to within any given error, and the precision oracle for $\S$ allows one to decide when a manifold $M\subseteq\RR^n$ is contained in a coordinate hyperplane $\{x\in\RR^n : x_i = 0\}$ when one is given $i\in\{1,\ldots,n\}$ and a system of equations which defines $M$ nonsingularly, where the functions occurring in the equations are rational polynomials of the coordinate variables $x = (x_1,\ldots,x_n)$ and the partial derivatives of the functions in $\S$ (see Definitions \ref{def:approxOracle} and \ref{def:precOracle} for precise definitions).  We say that $\S$ is {\bf computably $C^\infty$} if its approximation oracle is decidable (see Definition \ref{def:compCp}).  We will define in Section \ref{s:genDec} what it means for $\S$ to be ``generic'', which is a certain technically defined transcendence condition (see Definition \ref{def:generic}). Through a proof which was very much inspired by the way in which Macintyre and Wilkie used Schanuel's conjecture in \cite{MW}, we show that if $\S$ is generic and computably $C^\infty$, then the precision oracle for $\S$ is decidable.  This implies our first main theorem.

\begin{introThm}\label{introThm:genDec}
If $\S$ is generic, then the theory of $\RR_{\S}$ is decidable if and only if the approximation oracle for $\S$ is decidable.
\end{introThm}

We now work towards stating our second main theorem.

\begin{definitions}\label{def:compFamSpace}
Let $\Comp_{\C}(\rho)$ denote the set of all computably $C^\infty$ families of functions in $\C$ with domain map $\rho$, and let $\Gen_{\C}(\rho)$ denote that set of all generic families of functions in $\C$ with domain map $\rho$.
Put
\[
\Delta(\rho) = \{(\sigma,\alpha) : \sigma\in\Sigma, \alpha\in\NN^{\eta(\sigma)}\}.
\]
For each $\S\in\Comp_{\C}(\rho)$ and computable map $\epsilon:\Delta(\rho)\to\QQ_+$, define
\[
\Ball_{\C}(\S,\epsilon)
=
\left\{\T\in\Comp_{\C}(\rho) :
\begin{array}{l}
\left|\PDn{\alpha}{T_\sigma}{x}(x) - \PDn{\alpha}{S_\sigma}{x}(x)\right| < \epsilon(\sigma,\alpha), \vspace*{3pt}\\
\text{for all $(\sigma,\alpha)\in\Delta(\rho)$ and $x\in[-\rho(\sigma),\rho(\sigma)]$}
\end{array}
\right\},
\]
where we are writing $\S = \{S_\sigma\}_{\sigma\in\Sigma}$ and $\T = \{T_\sigma\}_{\sigma\in\Sigma}$.  Topologize $\Comp_{\C}(\rho)$ by taking
\[
\{\Ball_{\C}(\S,\epsilon) : \text{$\S\in\Comp_{\C}(\rho)$ and $\epsilon:\Delta(\rho)\to\QQ_+$ is computable}\}
\]
to be a base for its topology.
\end{definitions}

\begin{definition}\label{def:compAnalIFsystem}
For each $n\in\NN$ and $r\in\QQ_{+}^{n}$, let $\A_r$ denote the set of all real-valued functions on $[-r,r]$ which extend to a function in a neighborhood of $[-r,r]$ which is both analytic and computably $C^\infty$.  Then $\A = \bigcup_{n\in\NN,r\in\QQ_{+}^{n}}\A_r$ is a quasianalytic IF-system, called the {\bf IF-system of all computably analytic functions}. (This is the intersection of the IF-systems found in \cite[Examples 4.4.(2,3)]{DJMcharDec}.)
\end{definition}

\begin{introThm}\label{introThm:genDense}
If $\C$ contains the IF-system of all computably analytic functions, then $\Comp_{\C}(\rho)\cap\Gen_{\C}(\rho)$ is dense in $\Comp_{\C}(\rho)$.  In fact, there is an algorithm which acts as follows:
\begin{quote}
Given a $C^\infty$ approximation algorithm for $\S\in\Comp_{\C}(\rho)$ and a computable map $\epsilon:\Delta(\rho)\to\QQ_+$, the algorithm returns a $C^\infty$ approximation algorithm for some $\T\in \Gen(\rho) \cap \Ball_{\C}(\S,\epsilon)$.
\end{quote}
\end{introThm}

The functions in a generic family $\S$ are easily seen to be transcendental (see Remarks \ref{rmk:genTrans}), so Theorems \ref{introThm:genDec} and \ref{introThm:genDense} answer the question ($*$) in the affirmative.

While this paper and its parent paper \cite{DJMcharDec} were being written, Jones and Servi \cite{JonesServi} also answered the question ($*$) in the affirmative by proving the following two theorems:
\begin{enumerate}
\item
If $r\in\RR$ is not $0$-definable in the real exponential field, then the expansion of the real field by the power function $(0,+\infty)\to\RR:x\mapsto x^r$ has a decidable theory if and only if $r$ is a computable real.

\item
There exists a computable real which is not $0$-definable in the real exponential field.
\end{enumerate}
The type of o-minimal structures considered by Jones and Servi are different from what is considered here, but these two theorems do compare closely with Theorems \ref{introThm:genDec} and \ref{introThm:genDense}.

\subsection*{The Method}

We now discuss the main ideas of the proofs of Theorems \ref{introThm:genDec} and \ref{introThm:genDense}.  Nearly all the ``definitions'' given in this discussion are imprecise approximations to the actual definitions of the concepts, and are stated as such to not get bogged down in insignificant technical modifications needed by the actual definitions.  Precise definitions will be given in later sections.  We begin with some basic notation and terminology:
\begin{itemize}
\item
If $f:A\to B$, then $\dom(f) = A$ and $\im(f) = f(A)$.

\item
If $f:A\to B$, $C\subseteq B$, and $g:C\to A$, we call $g$ a {\bf section} of $f$ if $f\circ g(x) = x$ for all $x\in C$.

\item
If $p\in\QQ[x]^k$, where $x=(x_1,\ldots,x_n)$, define $\VV(p) = \{x\in\RR^n : p(x) = 0\}$.

\item
If $A\subseteq\RR^n$, define
\begin{eqnarray*}
\II(A)
    & = &
    \{q(x)\in\QQ[x] : \text{$q = 0$ on $A$}\},
\\
\QQ[A]
    & = &
    \{g:A\to\RR : \text{$g = q$ on $A$ for some $q(x)\in\QQ[x]$}\}
    \cong
    \QQ[x]/\II(A),
\\
\QQ(A)
    & = &
    \text{fraction field of $\QQ[A]$ (when $\II(A)$ is prime).}
\end{eqnarray*}

\item
If $K\subseteq L$ is a field extension, $\td_K L$ is the transcendence degree of $L$ over $K$.
\end{itemize}

An \emph{$\S$-polynomial map} is a function $P = p\circ F:U \to\RR^k$, where $U$ is an open box in $\RR^m$, $F:U\to\RR^{m+n}$ is defined by $F(x) = (x,f(x))$ with $f(x) = (f_1(x),\ldots,f_n(x))$ for some partial derivatives $f_1,\ldots,f_n$ of functions in $\S$, and $p(x,y)\in\QQ[x,y]^k$, with $x = (x_1,\ldots,x_m)$ and $y = (y_1,\ldots,y_n)$.  Throughout this discussion, we will use this notation for our $\S$-polynomial maps, except we will choose $k$ to be some appropriate number in $\{1,\ldots,m\}$.

Saying that $\S$ is \emph{generic} means that if $a\in\RR^m$ is such that $P(a) = 0$ and $\det\PD{}{P}{x}(a)\neq 0$ for some $\S$-polynomial map $P:U\to\RR^m$, and if the functions $f_1,\ldots,f_n$ are not redundantly listed more than is needed to define $a$ (this is made precise in Definition \ref{def:distCond}), then
\[
\td_{\QQ}\QQ(F(a)) = n.
\]

If $\varphi:\Pi(U)\to U$ is a section of a coordinate projection $\Pi:\RR^m\to\RR^d$, we say that an $\S$-polynomial map $P:U\to\RR^{m-d}$ \emph{implicitly defines} $\varphi$ if $\im(\varphi) = \{x\in U : P(x) = 0\}$ and $\rank\PD{}{P}{x} = m-d$ on $\im(\varphi)$.

The \emph{precision oracle for $\S$} acts as follows: given an $\S$-polynomial map $P:U\to\RR^{m-d}$ which implicitly defines a section $\varphi = (\varphi_1,\ldots,\varphi_m):\Pi(U)\to U$ of a coordinate projection $\Pi:\RR^m\to\RR^d$, and given $i\in\{1,\ldots,m\}$,  the oracle stops if and only if $\varphi_i$ is identically zero.  We would like to see how we might decide this oracle, so fix such a $P$, $\Pi$, $\varphi$, and $i$.  Note that since $F\circ\varphi(x) = (\varphi(x),f\circ\varphi(x))$, we have
\begin{equation}\label{eq:precOracleIdeal}
\text{$\varphi_i = 0$ if and only if $x_i \in \II(\im(F\circ\varphi))$.}
\end{equation}
Now, the nonsingularity of the equation $P(x) = 0$ and a simple computation involving the chain rule show that $\VV(p)$ is an $n+d$ dimensional manifold locally about each point of $\im(F\circ\varphi)$.   It follows that there is a unique irreducible component $X\subseteq\RR^{m+n}$ of the $\VV(p)$ such that $\im(F\circ\varphi) \subseteq X$, and that $\dim\II(X) = n+d$.   By computing the isolated primes of the ideal in $\QQ[x,y]$ generated by the components of $p(x,y)$, we can find a set of generators for the ideal $\II(X)$.  Since $\II(X) \subseteq \II(\im(F\circ\varphi))$, this gives
\begin{equation}\label{eq:tdGeneral}
\td_{\QQ}\QQ(\im(F\circ\varphi)) = \dim \II(\im(F\circ\varphi)) \leq \dim \II(X) = n+d.
\end{equation}

Now assume that $\S$ is generic.  We will show in Section \ref{s:genDec} that if the functions $f_1,\ldots,f_n$ are not redundantly listed more than is needed to define $\varphi$, then \begin{equation}\label{eq:tdGeneric}
\td_{\QQ}\QQ(\im(F\circ\varphi)) = n+d.
\end{equation}
(Note: The definition of generic is simply this statement with $d=0$.)  Now, \eqref{eq:tdGeneral} and \eqref{eq:tdGeneric} imply that $\II(\im(F\circ\varphi)) = \II(X)$, which when combined with \eqref{eq:precOracleIdeal}, shows that can decide if $\varphi_i$ is identically zero since we have a set of generators for $\II(X)$.

The geometric picture when $\S$ is generic can be summarized as follows: the set $\im(F\circ\varphi)$ is a manifold of dimension $d$, and \eqref{eq:tdGeneric} means that the $n$ functions $f_1,\ldots,f_n$ are so ``generic'' that the Zariski closure of $\im(F\circ\varphi)$ is an $n+d$ dimensional manifold locally about each point of $\im(F\circ\varphi)$, and hence $X$ must be the Zariski closure of $\im(F\circ\varphi)$.

The nontrivial direction of the proof of Theorem \ref{introThm:genDec} assumes that $\S$ is generic and computably $C^\infty$, and then proves that the precision oracle for $\S$ is decidable.  The only difficulty is to perform a suitable substitution so that the functions $f_1,\ldots,f_n$ are not ``redundantly listed more than is needed to define $\varphi$'', the exact meaning of which will not be explained here (again, see Definition \ref{def:distCond}).  When $\S$ is generic, we will be able to find an appropriate substitution by a search procedure over all possible redundant listing of the $f_1,\ldots,f_n$, and we can thereby decide the precision oracle for $\S$ using the ideas just discussed.

To prove Theorem \ref{introThm:genDense}, we must construct a family of functions which is both computably $C^\infty$ and generic.  To see what this entails, suppose that $\S$ is computably $C^\infty$, and consider a nonsingular zero $a\in\RR^m$ of an $\S$-polynomial map $P:U\to\RR^m$.  From \eqref{eq:tdGeneral} with $d=0$, we have $\td_{\QQ}\QQ(F(a)) \leq n$, so $\td_{\QQ}\QQ(F(a)) = n$ if and only if there exists a coordinate projection $\Pi:\RR^{m+n}\to\RR^n$ such that $q\circ\Pi(a)\neq 0$ for all nonzero $q(y)\in\QQ[y]$.  The proof of Theorem \ref{introThm:genDense} is based on the following two simple observations:
\begin{enumerate}{\setlength{\itemsep}{3pt}
\item
Even though solving a polynomial equation $q(y) = 0$ can be hard, solving a polynomial inequation $q(y)\neq 0$ is easy!

\item
We can computably enumerate names for all $\S$-polynomial maps, we can effectively discover each of their nonsingular zeros, and we can computably enumerate all nonzero polynomials in $n$ variables.
}\end{enumerate}
Theorem \ref{introThm:genDense} is proved by starting with $\S$ and then constructing a sequence $\{\S^{(k)}\}_{k\in\NN}$ of perturbations of $\S$ which converges uniformly (in a certain computable sense) to $\T$, and which successively solve each of the inequations needed to force $\T$ to be generic.  Each perturbation is done through the use of specially constructed interpolation polynomials.  The perturbations can be made as small as we wish, so we can force the limit function $\T$ to be in $\Ball(\S,\epsilon)$ for any given computable map $\epsilon:\Delta(\rho)\to\QQ_+$. \\

We now briefly outline the paper.  In Section \ref{s:IFT} we establish a continuity property of the implicit function theorem.  This is done for two reasons:
\begin{enumerate}{\setlength{\itemsep}{3pt}
\item
When perturbing $\S^{(k)}$ to construct $\S^{(k+1)}$, for each nonsingular zero $a^{(k)}$ of an $\S^{(k)}$-polynomial map which we have already constructed to solve some inequations, we want the perturbed nonsingular zero $a^{(k+1)}$ of the corresponding $\S^{(k+1)}$-polynomial map to also satisfy the same inequations.

\item
We need to know that every nonsingular zero of a $\T$-polynomial map is the limit of nonsingular zeros of $\S^{(k)}$-polynomial maps as $k\to\infty$.
}\end{enumerate}
In order to know that the limit family $\T$ is computably $C^\infty$, in Section \ref{s:compHolom} we establish some basic properties of sequences of families of computably holomorphic functions.  Section \ref{s:parent} discusses concepts from the parent paper \cite{DJMcharDec} that we shall need, but which are not included in \cite[Notation 0.2, Sections 2-4]{DJMcharDec}.  Theorems \ref{introThm:genDec} and \ref{introThm:genDense} are then proven in Sections \ref{s:genDec} and \ref{s:genDense}, respectively, along the lines discussed above.

%% file: IFT.tex
\section{Computable continuity of the implicit function theorem}
\label{s:IFT}

We begin by reviewing the definition of computably $C^p$ given in \cite[Section 2]{DJMcharDec}, which is a concept dealing with mappings between subsets of Euclidean spaces.  In a little bit, we will introduce some similar sounding terminology for mappings between certain function spaces, but these concepts have slightly different meanings from their Euclidean space counterparts.

\begin{definition}\label{def:compCp}
Consider a function $f:U\to\RR^m$, where $U$ is a c.e.\ open subset of some computable domain $D$ in $\RR^n$, and let $p\in\NN\cup\{\infty\}$.  We say that $f$ is {\bf computably $C^p$} if there is an algorithm which acts as follows:
\begin{equation}\label{eq:compCpRep}
\text{\parbox{5.3in}{Given $\alpha\in\NN^n$ such that $|\alpha|\leq p$, a name for an open rational box $I$ in $\RR^m$, and a name for a compact rational box $B$ in $D$, the algorithm stops if and only if $B\subseteq U$ and $\PDn{\alpha}{f}{x}(B)\subseteq I$.
}}
\end{equation}
Any such algorithm \eqref{eq:compCpRep} is called a {\bf $C^p$ approximation algorithm} for $f$. If $p=0$, we also say that $f$ is {\bf computably continuous}.  If $p=n=0$, we call $f(0)$ a {\bf computable point} in $\RR^m$.  If $p=n=0$ and $m=1$, we call $f(0)$ a {\bf computable real}.  We shall usually just say ``approximation algorithm'', rather than ``$C^0$ approximation algorithm'', when we are working with a computably continuous function, a computable point, or a computable real.

More generally, a family of functions $\S = \{S_\sigma\}_{\sigma\in\Sigma}$ is {\bf computably $C^p$} if the index set $\Sigma$ is computable, and if there is an algorithm which acts as a $C^p$ approximation algorithm for each function in $\S$, as indexed by $\Sigma$.  Such an algorithm is called a {\bf $C^p$ approximation algorithm} for the family $\S$.
\end{definition}

In \cite[Section 2]{DJMcharDec} it was shown that the computably continuous functions satisfy the following computable analogs of two standard topological theorems: (1) a function is computably continuous if and only if it is effective true that the inverse image of every c.e.\ open set is c.e.\ open; (2) images of co-c.e.\ compact sets under computably continuous functions are co-c.e.\ compact.  In \cite[Section 3]{DJMcharDec} it was shown that computably $C^p$ functions are closed under the arithmetic operations, function composition, parameterized integrals, division by variables, and extraction of implicitly defined functions.
We now show that computable continuous functions also have an extreme value theorem and are computably uniformly continuous on co-c.e.\ compact sets.

\begin{lemma}\label{lemma:compEVT}
Let $K\subseteq U$, where $U$ is c.e.\ open in $\RR^n$ and $K$ is co-c.e.\ compact, and let $f:U\to\RR$ be computably continuous.  Then $\max\{f(x) : x\in K\}$ and $\min\{f(x) : x\in K\}$ are computable reals.
\end{lemma}

\begin{proof}
Put $M = \max\{f(x) : x\in K\}$.  Let $\A$ be the set of all finite families $\{(A_i,B_i)\}_{i\in I}$ such that $K\subseteq\bigcup_{i\in I}\Int(B_i)$, where for each $i\in I$, $A_i$ is an open rational interval, $B_i$ is a compact rational box contained in $U$, and $f(B_i) \subseteq A_i$.  Note that if $(a,b)$ is an open interval, then $M\in (a,b)$ if and only if there exists $\{(A_i,B_i)\}_{i\in I}\in\A$ such that $A_i \subseteq (-\infty,b)$ for all $i\in I$, and $A_i \subseteq (a,b)$ for some $i\in I$.  Also note that since $K$ is co-c.e. compact and $U$ is c.e.\ open in $\RR^n$, we may construct a computable enumeration $\{\{(A_i,B_i)\}_{i\in I_j}\}_{j\in\NN}$ of $\A$.  Therefore the following is an approximation algorithm for $M$:
\begin{quote}
Given a rational open interval $(a,b)$, use time sharing to search for some $j\in\NN$ such that $A_i \subseteq (-\infty,b)$ for all $i\in I_j$, and $A_i \subseteq (a,b)$ for some $i\in I_j$.  Stop once such a $j$ has been found, and do not stop otherwise.
\end{quote}
The minimum value of $f$ on $K$ is also a computable real since $\min\{f(x) : x\in K\} = - \max\{-f(x) : x\in K\}$.
\end{proof}

\begin{definition}\label{def:compUnifCont}
Let $f:U\to\RR^m$ be a function, where $U$ is any set in $\RR^n$.  We say that $f$ is {\bf computably uniformly continuous on $A\subseteq U$} if there is a computable function $\delta:\QQ_+\to\QQ_+$ such that for all $\epsilon\in\QQ_+$ and $x,y\in A$, if $|x-y| < \delta(\epsilon)$, then $|f(x) - f(y)| < \epsilon$.
\end{definition}

Note that computably uniformly continuous functions need not be computably continuous, even if we assume their domain is c.e.\ open.  For example, if $a$ is a noncomputable real and we define $f:\RR\to\RR$ by $f(x) = a$, then $f$ is computably uniformly continuous but is not computably continuous.

\begin{lemma}\label{lemma:compUnifCont}
Let $K\subseteq U$, where $U$ is c.e.\ open in $\RR^n$ and $K$ is co-c.e.\ compact, and let $f:U\to\RR^m$ be computably continuous.  The $f$ is computably uniformly continuous.
\end{lemma}

\begin{proof}
Let $\epsilon\in\QQ_+$ be given.  Since $f$ is computably continuous, we can computably enumerate all triples $(A,B,C)$ such that $A$ and $B$ are nondegenerate compact rational boxes in $U$, $C$ is an open rational box in $\RR^m$, $f(B) \subseteq C$, $\diam(C) < \epsilon$, and $\{x\in\RR^n : \dist(x,A) \leq \diam(A)\} \subseteq B$.  Since $K$ is co-c.e.\ compact, we can find a finite family of such triples $\{(A_i,B_i,C_i)\}_{i\in I}$ such that $K\subseteq\bigcup_{i\in I}\Int(A_i)$.  Then $\delta = \min\{\diam(A_i) : i\in I\}$ is as desired.
\end{proof}

\begin{definition}\label{def:norm}
Consider $p\in\NN$.  For any $C^p$-function $f = (f_1,\ldots,f_k):U\to\RR^k$ defined on an open set $U\subseteq\RR^n$, write
\[
\partial^p(f) =
\left\{
\PDn{\alpha}{f_i}{x}
\right\}_{(i,\alpha)\in\{1,\ldots,k\}\times\NN^{n}_{\leq p}},
\]
where $\NN^{n}_{\leq p} = \{\alpha\in\NN^n : |\alpha| \leq p\}$.  Note that $\partial^p(f)$ is a family of functions, not a set of functions.  (Families are indexed; sets are not.)  For each co-c.e.\ compact set $A\subseteq\RR^n$, let
\[
C^p(A,\RR^k) = \{\partial^p(f) : \text{$U \supseteq A$ is c.e.\ open in $\RR^n$ and $f:U\to\RR^k$ is computably $C^p$}\}.
\]
Note that by Lemma \ref{lemma:compUnifCont}, $\partial^p(f)$ is a family of functions which are all computably uniformly continuous on $A$.
Also note that $C^p(A,\RR^k)$ is a vector space over the field of computable reals.  We equip $C^p(A,\RR^k)$ with the norm
\[
\left\|\partial^p(f)\right\|
=
\max\left\{\left|\PDn{\alpha}{f_i}{x}(x)\right| : x\in A, i\in\{1,\ldots,k\}, \alpha\in\NN^{n}_{\leq p}\right\}.
\]
Note that by Lemma \ref{lemma:compEVT}, we can compute $\|\partial^p(f)\|$ using a $C^p$ approximation algorithm for $f$.

If $B\subseteq\RR^m$ is also co-c.e.\ compact, define
\[
C^p(A,\RR^k) \times C^p(B,\RR^l) = \{\partial^p(f,g) : \partial^p(f) \in C^p(A,\RR^k), \partial^p(g) \in C^p(B,\RR^l)\},
\]
which is a certain subspace of $C^p(A\times B, \RR^{k+l})$.  For any set $Y\subseteq \RR^k$, let
\[
C^p(A,Y) = \{\partial^p(f)\in C^p(A,\RR^k) : f(A)\subseteq Y\},
\]
which is a metric space, but not a subspace of $C^p(A,\RR^k)$.

A set $\F\subseteq C^p(A,\RR^k)$ is {\bf c.e.\ open} in $C^p(A,\RR^k)$ if there is an algorithm acting as follows:
\begin{equation}\label{eq:compOpen}
\left\{\text{\parbox{5.8in}{
Given a $C^p$ approximation algorithm for some computably $C^p$ function $f:U\to\RR^k$ defined on a c.e.\ open neighborhood of $A$,\\
\hspace*{10pt} 1. the algorithm stops if and only if $\partial^p(f)\in\F$;\\
\hspace*{10pt} 2. if the algorithm stops, it returns a number $\epsilon\in\QQ_+$
                  such that for all \\
\hspace*{25pt}    $\partial^p(\tld{f})\in C^p(A,\RR^k)$, if
                  $\|\partial^p(f) - \partial^p(\tld{f})\| < \epsilon$, then $\partial^p(\tld{f})\in\F$.
}}
\right.
\end{equation}
If $\F\subseteq C^p(A,\RR^k)$ is c.e.\ open in $C^p(A,\RR^k)$, a map $\psi:\F \to C^q(B,\RR^l)$ is {\bf computably continuous} if there is an algorithm acting as follows:
\begin{quote}
Given a $C^p$ approximation algorithm for some $f$ with $\partial^p(f)\in \F$, and given $\epsilon\in\QQ_+$, the algorithm returns a number $\delta\in\QQ_+$ such that for all $\tld{f}\in \F$, if $\|\partial^p(f) - \partial^p(\tld{f})\| < \delta$, then $\|\partial^q(\psi(f)) - \partial^q(\psi(\tld{f}))\| < \epsilon$.
\end{quote}
Note that since $\F$ is assumed to be c.e.\ open in $C^p(A,\RR^k)$, the number $\delta$ can always be chosen so that for all $\partial^p(\tld{f})\in C^p(A,\RR^k)$, if $\|\partial^p(f) - \partial^p(\tld{f})\| < \delta$, then $\partial^p(\tld{f})\in\F$ and $\|\partial^q(\psi(f)) - \partial^q(\psi(\tld{f}))\| < \epsilon$.
\end{definition}

\begin{remarks}
The following facts are easy to verify.
\begin{enumerate}{\setlength{\itemsep}{3pt}
\item
The c.e.\ open subsets of $C^p(A,\RR^k)$ form a computable topology on  $C^p(A,\RR^k)$, in the following sense:
\begin{enumerate}
\item
The sets $\emptyset$ and $C^p(A,\RR^k)$ are c.e.\ open.

\item
If $\{\F_i\}_{i\in I}$ is a computable family of c.e.\ open subsets of $C^p(A,\RR^k)$ (meaning that there is a single algorithm which acts as \eqref{eq:compOpen} for each $i\in I$), then $\bigcup_{i\in I}\F_i$ is c.e.\ open.

\item
The intersection of finitely many c.e.\ open subsets of $C^p(A,\RR^k)$ is c.e.\ open.
\end{enumerate}

\item
If $\F\subseteq C^p(A,\RR^k)$ and $\G\subseteq C^q(B,\RR^l)$ are c.e.\ open in their respective spaces, and if $\psi:\F \to C^q(B,\RR^l)$ is computably continuous, then $\psi^{-1}(\G)$ is c.e.\ open in $C^p(A,\RR^k)$.
}\end{enumerate}
\end{remarks}

\begin{notation}\label{eq:norms}
For any $x = (x_1,\ldots,x_n)\in\RR^n$, write
\begin{eqnarray*}
|x|
    & = &
    \sqrt{x_{1}^{2} + \cdots + x_{n}^{2}}, \\
\|x\|
    & = &
    \max\{|x_1|,\ldots,|x_n|\},
\end{eqnarray*}
and note that
\begin{equation}\label{eq:norms}
\frac{|x|}{\sqrt{n}} \leq \|x\| \leq |x|.
\end{equation}
\end{notation}

\begin{lemma}\label{lemma:contArith}
The maps $+,\cdot : C^0(A,\RR^k)\times C^0(A,\RR^k) \to C^0(A,\RR^k)$ defined by $(f,g)\mapsto f+g$ and $(f,g)\mapsto fg$ are computably continuous.  For each $r \in\QQ_+$, the map from $\{f\in C^0(A,\RR) : \text{$|f(x)| > r$ for all $x\in A$}\}$ into $C(A,\RR)$ given by $f\mapsto 1/f$ is computably continuous.
\end{lemma}

\begin{proof}
This is straightforward.
\end{proof}

\begin{lemma}\label{lemma:contComp}
The map $\circ : C^p(B,\RR^k) \times C^p(A,B) \to C^p(A,\RR^k)$ given by $(\partial^p(f),\partial^p(g)) \mapsto \partial^p(f\circ g)$ is computably continuous.
\end{lemma}

\begin{proof}
First suppose that $p=0$.  We write $x$ and $y$ for coordinates on the ambient Euclidean spaces containing $A$ and $B$, respectively.  Fix $(f,g) \in C^0(B,\RR^k) \times C^0(A,B)$ and $\epsilon\in\QQ_+$.  We can find $\delta'\in\QQ_+$ such that for all $y,\tld{y}\in B$, if $\|y-\tld{y}\| < \delta'$, then $\|f(y) - f(\tld{y})\| < \frac{\epsilon}{2}$.  Put $\delta = \min\{\delta', \frac{\epsilon}{2}\}$.  Then for all $(\tld{f},\tld{g}) \in C^0(B,\RR^k) \times C^0(A,B)$ with $\|(f,g) - (\tld{f},\tld{g})\| < \delta$ and all $x\in A$,
\[
\|f(g(x)) - \tld{f}(\tld{g}(x))\|
\leq
\|f(g(x)) - f(\tld{g}(x))\| + \|f(\tld{g}(x)) - \tld{f}(\tld{g}(x))\|
<
\epsilon,
\]
so $\|f\circ g - \tld{f}\circ\tld{g}\| < \epsilon$ (the inequality is strict by the extreme value theorem).  This proves the lemma when $p=0$.  The general case follows from the special case of $p=0$, the chain rule, and Lemma \ref{lemma:contArith}.
\end{proof}

Consider $r\in(0,+\infty)^m$, $s\in(0,+\infty)^n$, and a $C^1$ function $f=(f_1,\ldots,f_n):U\to\RR^n$ defined on an open set $U\subseteq\RR^{m+n}$ containing $[-r,r]\times[-s,s]$.  Write $x = (x_1,\ldots,x_m)$ and $y = (y_1,\ldots,y_n)$ for coordinates on $\RR^m$ and $\RR^n$, respectively.  In \cite[Section 3]{DJMcharDec}, a statement $\IF(f;r,s)$ was defined with the following four properties:
\begin{enumerate}{\setlength{\itemsep}{5pt}
\item[IF1.]
If $\IF(f;r,s)$ holds, then
\begin{enumerate}{\setlength{\itemsep}{3pt}
\item
the set
\begin{equation}\label{eq:IFgraph}
\{(x,y)\in[-r,r]\times[-s,s] : f(x,y) = 0\}
\end{equation}
is the graph of a $C^1$ function from $[-r,r]$ into $(-s,s)$;

\item
$\det\PD{}{f}{y}\neq 0$ on \eqref{eq:IFgraph}.
}\end{enumerate}

\item[IF2.]
Suppose $\IF(f;r,s)$ holds.  Then there exists an open box $A\subseteq(0,+\infty)^{m+n}$ containing $(r,s)$ such that $\IF(f;u,v)$ holds for all $(u,v)\in A$.  Moreover, for any open box $V\subseteq U$ containing the origin and any open box $B\subseteq(0,+\infty)^{m+n}$ containing $(r,s)$ such that $V+B\subseteq A$,  the statement $\IF(f_{(a,b)}; u,v)$ holds for all $(a,b)\in V$ and all $(u,v)\in B$, where $f_{(a,b)}(x,y) = f(x+a,y+b)$.\vspace*{3pt}

(Note that $V+B\subseteq A$ will hold for all sufficiently small $V$ and $B$.)

\item[IF3.]
If $f(0) = 0$ and $\det\PD{}{f}{y}(0)\neq 0$, then there exists $(r,s)\in\QQ_{+}^{m}\times\QQ_{+}^{n}$ such that $\IF(f;r,s)$ holds.

\item[IF4.]
If $f$ is computably $C^1$, $r\in\QQ_{+}^{m}$, $s\in\QQ_{+}^{n}$, and $\IF(f;r,s)$ holds, then we can effectively verify that $\IF(f;r,s)$ holds.  In other words, there is an algorithm which acts as follows:
\begin{quote}
Given a $C^1$ approximation algorithm for $f$ and $(r,s)\in\QQ_{+}^{m}\times\QQ_{+}^{n}$, the algorithm stops if and only if $[-r,r]\times[-s,s]\subseteq U$ and $\IF(f;r,s)$ holds.
\end{quote}
}
\end{enumerate}

In \cite{DJMcharDec} the definition of $\IF(f;r,s)$ and these four properties were proven simultaneously by induction on $n$.  We will not repeat all of that here, but will review the definition of $\IF(f;r,s)$ and the proof of IF1(a), for both will be used to prove Proposition \ref{prop:contIFT} below.

\begin{definition}\label{def:IF}
Define the statement $\IF(f;r,s)$ inductively as follows:
\begin{enumerate}{\setlength{\itemsep}{5pt}
\item[]\emph{Bases Case}: $n=1$.

Then $\IF(f;r,s)$ means that there exists $\sigma\in\{-1,1\}$ such that
$\sigma\cdot\PD{}{f}{y}(x,y)> 0$ for all $(x,y)\in[-r,r]\times[-s,s]$, and such that $\sigma\cdot f(x,-s) < 0 < \sigma\cdot f(x,s)$ for all $x\in[-r,r]$.

\item[]\emph{Inductive Step}: $n > 1$.

Then $\IF(f;r,s)$ means that there exist  $i,j\in\{1,\ldots,n\}$ such that, if we write
\begin{eqnarray*}
y' & = & (y_1, \ldots, y_{j-1}, y_{j+1}, \ldots, y_n), \\
s' & = & (s_1, \ldots, s_{j-1}, s_{j+1}, \ldots, s_n), \\
f' & = & (f_1, \ldots, f_{i-1}, f_{i+1}, \ldots, f_n),
\end{eqnarray*}
then $\IF(f_i;(r,s'), s_j)$ and $\IF(f'\circ H;r,s')$ both hold, where $H$ is defined as follows:
\begin{quote}
Since $\IF(f_i;(r,s'), s_j)$ holds, property IF2 from the base case of the induction shows that there exist tuples $R > r$ and $S' > s'$ such that $\IF(f_i;(R,S'),s_j)$ holds.  Let $U' = (-R,R)\times(-S',S')$, and let $h:U'\to(-s_j,s_j)$ be the $C^1$ function whose graph is the set
\[
\{(x,y',y_j)\in U'\times[-s_j,s_j] :f_i(x,y) = 0\}.
\]
Define $H:U'\to\RR^m\times\RR^n$ by
\begin{equation}\label{eq:H}
H(x,y') = (x,y_1,\ldots,y_{j-1},h(x,y'),y_{j+1},\ldots,y_n).
\end{equation}
\end{quote}
}\end{enumerate}
\end{definition}

We will also have use for the function $H':U'\to\RR^n$ defined by
\[
H'(x,y') = (y_1,\ldots,y_{j-1},h(x,y'),y_{j+1},\ldots,y_n).
\]

\begin{proof}[Proof of IF1(a)]
Assume $\IF(f;r,s)$ holds.  First suppose that $n = 1$.  The intermediate value and increasing function theorems show that \eqref{eq:IFgraph} is the graph of a function from $[-r,r]$ into $(-s,s)$.  The implicit function theorem shows that this function is $C^1$.

Now suppose that $n > 1$, and assume (IF1) holds for any $k\in\{1,\ldots,n-1\}$ in place of $n$.  We use the notation from the inductive step of Definition \ref{def:IF}.  The base case shows that $h$ is $C^1$, so $f'\circ H$ is $C^1$.  Since $\IF(f'\circ H;r,s')$ holds, the induction hypothesis shows that $\{(x,y') \in [-r,r]\times[-s',s'] : f'\circ H(x,y') = 0\}$ is the graph of a $C^1$ function $g':[-r,r]\to(-s',s')$.  The function $g:[-r,r]\to(-s,s)$ defined by $g(x) = H'\circ g'(x)$ is $C^1$, and $\{(x,y)\in[-r,r]\times[-s,s] : f(x,y) = 0\}$ is the graph of $g$.
\end{proof}

\begin{proposition}\label{prop:contIFT}
Fix an integer $p > 0$, $r\in\QQ_{+}^{m}$, and $s\in\QQ_{+}^{n}$.  Define
\[
\F^{p}_{(r;s)} = \{\partial^p(f) \in C^p([-r,r]\times[-s,s],\RR^n) : \text{$\IF(f;r,s)$ holds}\}.
\]
\begin{enumerate}{\setlength{\itemsep}{3pt}
\item
The set $\F^{p}_{(r;s)}$ is c.e.\ open in $C^p([-r,r]\times[-s,s],\RR^n)$.

\item
Consider the map $\IF: \F^{p}_{(r;s)} \to C^p([-r,r], (-s,s))$ sending each $\partial^p(f) \in \F^{p}_{(r;s)}$ to $\partial^p(f_{\IF})$, where $f_{\IF}:[-r,r]\to(-s,s)$ is implicitly defined by $f(x,f_{\IF}(x)) = 0$ on $[-r,r]$.  This map is computably continuous.
}\end{enumerate}
\end{proposition}

\begin{proof}
The proof is by induction on $n$.  First suppose that $n=1$.   We first show that $\F^{1}_{(r;s)}$ is c.e.\ open in $C^1([-r,r]\times[-s,s],\RR)$ and that the map $\F_{(r;s)}^{1} \to C^0([-r,r],(-s,s)) : \partial^1(f) \mapsto f_{\IF}$ is computably continuous.  Fix $\partial^1(f) \in \F^{1}_{(r;s)}$.  Property IF4 shows that we can effectively verify that $\partial^1(f)\in\F^{1}_{(r;s)}$.   We can effectively find $\sigma\in\{-1,1\}$ and $a\in\QQ_+$ such that $\sigma \PD{}{f}{y} > a$ on $[-r,r]\times[-s,s]$, and such that $\sigma f(x,-s) < -a$ and $\sigma f(x,s) > a$ for all $x\in[-r,r]$.  The neighborhood of $\partial^1(f)$ in $C^0([-r,r]\times[-s,s],\RR)$ of radius $\frac{a}{2}$ is contained in $\F^{1}_{(r;s)}$, so $\F^{1}_{(r;s)}$ is c.e.\ open.  Now, suppose we are given $\epsilon\in\QQ_+$.  Put
\[
\delta = \min\left\{\frac{a}{2}, a\epsilon\right\}.
\]
Consider $\partial^1(\tld{f})\in C^1(U,\RR^n)$ with $\|\partial^1(f) - \partial^1(\tld{f})\| < \delta$.  Let $b>0$ be such that $\sigma \PD{}{f}{y} < b$ on $[-r,r]\times[-s,s]$.  For any $(x,y)\in[-r,r]\times[-s,s]$,
\[
\sigma f(x,y) = \sigma(f(x,y) - f(x,f_{\IF}(x))) = \sigma\PD{}{f}{y}(x,\xi(x))(y-f_{\IF}(x))
\]
for some $\xi(x)$ between $y$ and $f_{\IF}(x)$.  It follows that
\[
a(y-f_{\IF}(x)) < \sigma f(x,y) < b(y-f_{\IF}(x))
\]
on $[-r,r]\times[-s,s]$.  Writing $\sigma\tld{f} = \sigma\tld{f} - \sigma f + \sigma f$ and applying the previous bound shows that
\[
a(y-f_{\IF}(x)) - \|f-\tld{f}\| < \sigma\tld{f}(x,y) < b(y-f_{\IF}(x)) + \|f-\tld{f}\|
\]
on $[-r,r]\times[-s,s]$.  Therefore, if $y - f_{\IF}(x) \leq -\|f-\tld{f}\|/b$ we have $\sigma\tld{f}(x,y)  < 0$, and if $y - f_{\IF}(x) \geq \|f-\tld{f}\|/a$ we have $\sigma \tld{f}(x,y) > 0$.  Since $\tld{f}(x,\tld{f}_{\IF}(x)) = 0$, it follows that
\[
-\epsilon
    \leq  -\frac{\delta}{a}
    < -\frac{\|f-\tld{f}\|}{b}
    < \tld{f}_{\IF}(x) - f_{\IF}(x)
    < \frac{\|f-\tld{f}\|}{a}
    \leq \frac{\delta}{a}
    \leq \epsilon,
\]
so $\|\tld{f}_{\IF}- f_{\IF}\| < \epsilon$.  This shows that the map $\partial^1(f) \mapsto f_{\IF}$ is computably continuous.

Now, for any positive integer $p$, the set $\F^{p}_{(r;s)}$ is c.e.\ open in $C^p([-r,r]\times[-s,s],\RR)$ since it is the inverse image of $\F^{1}_{(r;s)}$ under the computably continuous map $C^p(A,\RR^k)\to C^1(A,\RR^k): \partial^p(f)\mapsto \partial^1(f)$.  To see that $\IF: \F^{p}_{(r;s)} \to C^p([-r,r], (-s,s))$ is computably continuous, implicitly differentiate $f(x,f_{\IF}(x)) = 0$ repeatedly, and apply Lemmas \ref{lemma:contArith} and \ref{lemma:contComp}.  This completes the base case of the induction.

Now suppose that $n > 1$.  We first show that $\F^{p}_{(r;s)}$ is c.e.\ open in $C^p([-r,r]\times[-s,s],\RR^n)$.  By the same reasoning as in the base case, it suffices to show that $\F^{1}_{(r;s)}$ is c.e.\ open in $C^1([-r,r]\times[-s,s],\RR^n)$.  We use the notation $i,j,f_i,f',r,s',s_j$ from the inductive step of Definition \ref{def:IF}.  For each choice of $i,j\in\{1,\ldots,n\}$, let $\F_{(r;s)}^{1}(i,j)$ be the set of all $\partial^1(f)\in C^1([-r,r]\times[-s,s],\RR^n)$ such that $\IF(f_i;(r,s'),j)$ and $\IF(f'\circ H;r,s')$ both hold, where
\[
H(x,y') = (x,y_1,\ldots,y_{j-1},(f_i)_{\IF}(x,y'),y_{j+1},\ldots,y_n).
\]
Thus $\F^{1}_{(r;s)}(i,j)$ is the intersection of the inverse images of the sets $\F^{1}_{(r,s';s_j)}$ and $\F^{1}_{(r;s')}$ under the maps $\partial^1(f)\mapsto \partial^1(f_i)$ and $\partial^1(f) \mapsto \partial^1(f'\circ H)$.  The sets $\F^{1}_{(r,s';s_j)}$ and $\F^{1}_{(r;s')}$ are both c.e.\ open by the induction hypothesis.  The map $\partial^1(f)\mapsto \partial^1(f_i)$ is clearly computably continuous, and the map $\partial^1(f) \mapsto \partial^1(f'\circ H)$ is also  computably continuous since it can be expressed as the composition of the two computably continuous maps $\partial^1(f) \mapsto \partial^1(f', (f_i)_{\IF})$ and $\partial^{1}(f',h)\mapsto \partial^1(f'\circ H)$, with $H$ defined from $h$ as in \eqref{eq:H}.  Therefore $\F^{1}_{(r;s)}(i,j)$ is c.e.\ open, and hence so is $\F^{1}_{(r;s)} = \bigcup_{i,j=1}^{n} \F^{1}_{(r;s)}(i,j)$.

We now show that the map $\IF:\F_{(r;s)}^{p} \to C^p([-r,r],(-s,s))$ is computably continuous.  Fix $i,j\in\{1,\ldots,n\}$.  It suffices to show that the restriction of this map to $\F^{p}_{(r;s)}(i,j)$ is computably continuous.  We now follow closely the inductive step in the proof of property (IF1).  The base case and Lemma \ref{lemma:contComp} imply that the map $\F^{p}_{(r;s)} \to \F^{p}_{(r;s')} : \partial^p(f) \mapsto \partial^p(f'\circ H)$ is computably continuous, and the induction hypothesis shows that the map $\F^{p}_{(r;s')} \to C^p([-r,r],(-s',s')) : \partial^1(G') \mapsto\partial^1(G'_{\IF})$ is computably continuous.  So Lemma \ref{lemma:contComp} shows that $\F^{p}_{(r;s)} \to C^p([-r,r],(-s',s')) : \partial^p(f) \mapsto \partial^p((f'\circ H)_{\IF})$ is computably continuous, and hence so is the map $\F^{p}_{(r;s)} \to C^p([-r,r],(-s,s)) : \partial^p(f) \mapsto \partial^p(H'\circ(f'\circ H)_{\IF})$, where
$
H'(x,y') = (y_1,\ldots,y_{j-1},(f_i)_{\IF}(x,y'),y_{j+1},\ldots,y_n).
$
We are done since $f_{\IF} = H'\circ(f'\circ H)_{\IF}$.
\end{proof}

\begin{definition}\label{def:ratBoxManifold}
A set $D\subseteq\RR^m$ is a called a {\bf rational box manifold} if it is a rational box and a submanifold of $\RR^m$.
\end{definition}

Thus $D\subseteq\RR^m$ is a rational box manifold if and only if there exist $E\subseteq\{1,\ldots,m\}$, an open rational box  $U\subseteq\RR^E$, and a point $u\in\QQ^{E^c}$ such that $D = U\times\{u\}$.  Note that $\dim(D) = |E|$.

Instead of working with $\IF(f;r,s)$ directly, we will work with the following variant which was defined in \cite[Section 2]{DJMcharDec}.

\begin{definition}\label{def:IFsection}
Consider the following given data:
\begin{itemize}{\setlength{\itemsep}{3pt}
\item
a rational box manifold $D\subseteq\RR^m$,

\item
$d\in\{0,\ldots,\dim(D)\}$,

\item
a $C^1$ function $P:D\to\RR^{\dim(D)-d}$,

\item
a bounded rational box manifold $C$ which is open in $D$ with $\cl(C) \subseteq D$,

\item
an injection $\lambda:\{1,\ldots,d\}\to\{1,\ldots,m\}$ such that $\Pi_\lambda(D)$ is open in $\RR^d$.
}\end{itemize}
Define the statement $\IF_\lambda(P;B)$ as follows:
\begin{quote}
Write $D = U\times\{u\}$ for a set $E\subseteq\{1,\ldots,m\}$, open rational box  $U\subseteq\RR^E$, and $u\in\QQ^{E^c}$.  Write $C = (c-R,c+R)\times\{u\}$ for some $c\in\QQ^E$ and $R\in\QQ_{+}^{E}$ such that $[c-R,c+R] \subseteq U$.  Define $T_c:\RR^E \to \RR^m$ by $T_c(x_E) = (x_E+c,u)$.  Note that $\im(\lambda)\subseteq E$.  Extend $\lambda$ to a bijection $\sigma:\{1,\ldots,\dim(M)\}\to E$, and write $\Pi_\sigma(R) = (r,s)$, where $r\in\QQ_{+}^{d}$ and $s\in\QQ_{+}^{\dim(M)-d}$.
The statement $\IF_\lambda(P;C)$ means $\IF(P\circ T_c\circ \Pi_{\sigma}^{-1}; r,s)$.
\end{quote}
\end{definition}

\begin{remarks}\label{rmk:IFsection}
Consider the situation of Definition \ref{def:IFsection}.  Put $V = (c-R,c+R)$, and define $\lambda':\{1,\ldots,\dim(D)-d\}\to\{1,\ldots,m\}$ by $\lambda'(i) = \sigma(i+d)$ for each $i$.  It follows easily from (IF1)-(IF4) that the statement $\IF_\lambda(P;C)$ has the following properties:
\begin{enumerate}{\setlength{\itemsep}{3pt}
\item
There exists a $C^1$ section $\varphi:\Pi_\lambda(C)\to C$ of the projection $\Pi_\lambda:\RR^m\to\RR^d$ such that
\[
\im(\varphi) = \{x\in C : P(x) = 0\},
\]
and $\det\PD{}{P}{x_{\lambda'}}(x)\neq 0$ for all $x\in\im(\varphi)$.

\item
There exists $\epsilon > 0$ such that for all $C' = V'\times\{u\}$, where $V'\subseteq\RR^E$ is a bounded open rational box, if $\bd(V') \subseteq \{x\in\RR^E : \dist(x,\bd(V)) < \epsilon\}$, then $\cl(V') \subseteq U$ and $\IF_\lambda(f;C')$ holds.

\item
If $a\in D$ is such that $f(a) = 0$ and $\det\PD{}{f}{x_{\lambda'}}(a)\neq 0$, then there exists a rational box manifold $A$ such that $a\in A$, $A$ in open in $D$, $\cl(A)\subseteq D$, and $\IF_\lambda(P,A)$ holds.

\item
If $f$ is computably $C^1$ and $\IF_\lambda(f;C)$ holds, then we can effectively verify that $\IF_\lambda(f;C)$ holds.
}\end{enumerate}
The first remark states that when $\IF_\lambda(P;C)$ holds, the equation $P(x) = 0$ defines the set $\im(\varphi)$ nonsingularly as a subset of $D = U\times\{u\}$.  Of course, $\Pi_{E^c}(\im(\varphi)) = \{u\}$, so the system of equations $P(x_E,u) = 0$ and $x_{E^c} - u = 0$ define $\im(\varphi)$ nonsingularly as a subset of $U\times\RR^{E^c}$.  Namely,
\[
\im(\varphi) = \{x\in\RR^n : x_E\in U, P(x_E,u) = 0, x_{E^c} - u = 0\},
\]
and
\[
\det\PD{}{(P(x_E,u), x_{E^c} - u)}{(x_{\lambda'},x_{E^c})}
=
\det\left(
\begin{matrix}
\PD{}{P}{x_{\lambda'}}(x_E,u)   & 0 \\
0                  & \id
\end{matrix}
\right)
\neq 0
\]
on $\im(\varphi)$.

The case when $d=0$ is of particular interest.  In this case, $\lambda$ is the empty map so we write $\IF_\emptyset(P;C)$, and we have $\{x\in C : P(x) = 0\} = \{a\}$ for a single point $a$.  Note that if $\IF_\emptyset(P;C)$ and $\IF_\emptyset(P;B)$ both hold, with $a\in B$ and $a\in C$, then property 3 implies that there exists a rational box manifold $A\subseteq B\cap C$ containing $a$ such that $\IF_\emptyset(P;A)$ holds (this is because the set $D$ in property 3 can be made as small as we wish).  If $P$ is computably $C^1$ and if $B$ and $C$ are rational boxes, property 4 implies that we can effectively find such an $A$ through a search procedure.
\end{remarks}

%% file: compHolom.tex
\section{Computably Holomorphic Functions}
\label{s:compHolom}

We begin with a lemma about computably $C^p$ functions between Euclidean spaces.

\begin{lemma}\label{lemma:compApprox}
Let $p\in\NN\cup\{\infty\}$, and let $U$ be a c.e.\ open subset of some computable domain in $\RR^n$.  A $C^p$ function $f = (f_1,\ldots,f_m):U\to\RR^m$ is computably $C^p$ if and only if there exists an algorithm acting as follows:
\begin{equation}\label{eq:approxAlg}
\text{\parbox{5.2in}{Given $\epsilon\in\QQ_+$, $\alpha\in\NN^{n}_{\leq p}$, and the name for a nondegenerate compact rational box $B\subseteq\RR^n$, the algorithm stops if and only if $B\subseteq U$.  If the algorithm stops, it outputs an approximation algorithm for a computably continuous function $g:B\to\RR^m$ such that $\left|\PDn{\alpha}{f}{x}(x) - g(x)\right| < \epsilon$
for all $x\in B$.}}
\end{equation}
\end{lemma}

\begin{proof}
If $f$ is computably $C^p$, then letting $g(x) = \PDn{\alpha}{f}{x}(x)$ gives an algorithm such as \eqref{eq:approxAlg}.  Conversely, suppose that there is an algorithm such as \eqref{eq:approxAlg}.  We claim that the following is a $C^p$ approximation algorithm for $f$:
\begin{quote}
Suppose we are given $\alpha\in\NN^{n}_{\leq p}$ and names for a nondegenerate compact rational box $B \subseteq U$ and an open rational box $I = \prod_{i=1}^{m}(a_i,b_i) \subseteq\RR^m$.  For each positive integer $k$, let $I_k = \prod_{i=1}^{m}(a_i+1/k, b_i-1/k)$, and let $g_k = (g_{k,1},\ldots,g_{k,m}):B\to\RR^m$ be the computable function given by \eqref{eq:approxAlg} such that
\begin{equation}\label{eq:gk}
\left|\PDn{\alpha}{f}{x}(x) - g_k(x)\right| < \frac{1}{k}
\end{equation}
for all $x\in B$.  Use time sharing to search for a positive integer $k$ such that
\begin{equation}\label{eq:gkB}
g_k(B)\subseteq I_k.
\end{equation}
Stop once an integer $k$ has been found, and do not stop otherwise.
\end{quote}
To verify the claim, it suffices to show that $\PDn{\alpha}{f}{x}(B)\subseteq I$ if and only if \eqref{eq:gkB} holds for some positive integer $k$, since the algorithm will discover such a $k$ if and only if there actually exists such a $k$.

If \eqref{eq:gkB} holds for some $k$, then \eqref{eq:norms}, \eqref{eq:gk}, and \eqref{eq:gkB} show that $\PDn{\alpha}{f}{x}(B)\subseteq I$.  To prove the converse, suppose that $\PDn{\alpha}{f}{x}(B)\subseteq I$.  Since $I$ is open and $\PDn{\alpha}{f}{x}(B)$ is compact, $\delta := \dist\left(\PDn{\alpha}{f}{x}(B),\bd(I)\right)$ is positive.  For all $i\in\{1,\ldots,m\}$, $k > 2/\delta$, and $x\in B$, we have
\[
g_{k,i}(x) < \PDn{\alpha}{f_i}{x}(x) + \frac{1}{k} \leq b_i -\delta + \frac{1}{k} < b_i - \frac{2}{k} + \frac{1}{k} =  b_i -\frac{1}{k}.
\]
A similar calculation shows that
\[
g_{k,i}(x) > a_i + \frac{1}{k},
\]
so $g_k(B)\subseteq I_k$.
\end{proof}

By identifying $\CC^n$ with $\RR^{2n}$, all of the computable concepts of point-set topology from \cite[Section 2]{DJMcharDec} also make sense in $\CC^n$, as does the notion of a computably continuous function $f:U\to\CC^m$, where $U$ is c.e.\ open in $\CC^n$.  One could define what it means for $f:U\to\CC^m$ to be computably $C^1$, analogous to the real case, but one would simply obtain a definition equivalent to the following concept.

\begin{definition}\label{def:compHolom}
Let $U$ be c.e.\ open in $\CC^n$.  A function $f:U\to\CC^m$ is {\bf computably holomorphic} is $f$ is holomorphic and computably continuous.   More generally, a family of functions on $U$ is computably holomorphic if the family is computably continuous and each member of the family is holomorphic.
\end{definition}

For a holomorphic function $f:U\to\RR$ defined on a simply connected open set $U\subseteq \CC^n$, recall Cauchy's integral formula,
\begin{equation}\label{eq:Cauchy}
f(x) = \frac{1}{(2\pi i)^n} \int_{C_1}\!\!\cdots\!\!\int_{C_n} \frac{f(z)}{(z_1-x_1)\cdots(z_n-x_n)} dz_1\cdots dz_n,
\end{equation}
and its more general differentiated form,
\begin{equation}\label{eq:CauchyDeriv}
\PDn{\alpha}{f}{x}(x) = \frac{\alpha!}{(2\pi i)^n} \int_{C_1}\!\!\cdots\!\!\int_{C_n}
\frac{f(z)}{(z_1-x_1)^{\alpha_1+1}\cdots(z_n-x_n)^{\alpha_n+1}}
dz_1\cdots dz_n,
\end{equation}
for each $\alpha\in\NN^n$, where $C_1,\ldots,C_n$ are counterclockwise oriented circular paths in $\CC$ with $C = C_1\times\cdots\times C_n \subseteq U$, and $x = (x_1,\ldots,x_n)$ with each $x_i$ in the region in $\CC$ bounded by $C_i$.  We use a shorthand notation for \eqref{eq:Cauchy} and \eqref{eq:CauchyDeriv}, writing
\[
f(x) = \frac{1}{(2\pi i)^n}\int_C\frac{f(z)}{z-x}dz,
\]
and
\[
\PDn{\alpha}{f}{x}(x) = \frac{\alpha!}{(2\pi i)^n} \int_{C} \frac{f(z)}{(z-x)^{\alpha+\bar{1}}} dz,
\]
for each $\alpha\in\NN^n$, where $\bar{1} = (1,\ldots,1)$.  Observe that if  $f:U\to\CC$ is computably continuous, and if the parameterization of the paths $C_1,\ldots,C_n$ and their first derivatives are all computably continuous, then the real and imaginary parts of the integral \eqref{eq:CauchyDeriv} can be expressed as real parameterized integrals of computably continuous functions.  It was shown in \cite[Lemma 3.5]{DJMcharDec} that computably $C^p$ functions are preserved under parameterized integrals.

\begin{lemma}\label{lemma:compHolom}
Let $U$ be c.e.\ open in $\CC^n$ and $f:U\to\CC^m$ be computably holomorphic. Then $f$ is computably $C^\infty$, in the sense that there exists an algorithm acting as follows:
\begin{quote}
Given $\alpha\in\NN^n$, a compact rational box $B\subseteq\CC^n$, and an open rational box $I\subseteq\CC^m$, the algorithm stops if and only if $B\subseteq U$ and $\PDn{\alpha}{f}{x}(B) \subseteq I$.
\end{quote}
\end{lemma}

\begin{proof}
Consider a compact rational box $B\subseteq U$.  Since $U$ is c.e.\ open, we can effectively find a compact rational box $A\subseteq U$ such that $B\subseteq\Int(A)$.  Write $A = \prod_{i=1}^{n}A_i$ and $B = \prod_{i=1}^{n} B_i$, where the $A_i$ and $B_i$ are compact rational boxes in $\CC$.  For each $i\in\{1,\ldots,n\}$, choose a counterclockwise oriented, rectangular path $C_i \subseteq \Int(A_i)\setminus B_i$ with rational vertices.  Applying \eqref{eq:CauchyDeriv} and the comments made just prior to the lemma shows that $f$ is computably $C^\infty$ on $B$, and hence is computably $C^\infty$ on $U$ by a rather trivial application of Lemma \ref{lemma:compApprox} (which applies to $\CC$ via the identification of $\CC$ with $\RR^2$).
\end{proof}

\begin{notation}\label{notation:complexNeigh}
For each $r,s>0$, let
\[
\CC(r,s) = \{z\in\CC : \dist([-r,r],z) < s\}.
\]
For each $r = (r_1,\ldots,r_n)$ and $s = (s_1,\ldots,s_n)$ in $(0,\infty)^n$, let
\[
\CC(r,s) = \CC(r_1,s_1)\times\cdots\times\CC(r_n,s_n).
\]
\end{notation}

Fix a computable map $R:\Sigma\to\bigcup_{n\in\NN}\QQ_{+}^{n}$ with arity map $\eta$.

\begin{definition}\label{def:San}
Let $\Comp_{\an}(\rho,R)$ denote the set of all computably holomorphic families of functions $\S = \{S_\sigma\}_{\sigma\in\Sigma}$ such that for each $\sigma\in\Sigma$, we have $S_\sigma:\CC(\rho(\sigma),R(\sigma))\to\CC$ and $S_\sigma(x) \in\RR$ for all $x\in (-\rho(\sigma)-R(\sigma),\rho(\sigma)+R(\sigma))$.
\end{definition}

\begin{definition}\label{def:compConv}
Suppose that $\{\S^{(k)}\}_{k\in\NN}$ is a {\bf computable sequence in $\Comp_{\an}(\rho,R)$}: this means that each $\S^{(k)} = \{S^{(k)}_{\sigma}\}_{\sigma\in\Sigma}$ is a member of $\Comp_{\an}(\rho,R)$ and that $\{S^{(k)}_{\sigma}\}_{(k,\sigma)\in\NN\times\Sigma}$ is a computably continuous family.  The sequence $\{\S^{(k)}\}_{k\in\NN}$ {\bf converges computably} to $\S\in\Comp_{\an}(\rho,R)$, written as
\[
\text{$\S = \lim_{k\to\infty} \S^{(k)}$ computably,}
\]
if there exists a computable function $K:\Sigma\times\QQ_+\to\NN$
such that for all $(\sigma,\epsilon)\in\Sigma\times\QQ_+$ and all
$k\geq K(\sigma,\epsilon)$,
\[
\text{$|S_{\sigma}^{(k)}(x) - S_\sigma(x)| < \epsilon$ on $\CC(\rho(\sigma),R(\sigma))$.}
\]
We say that $\{\S^{(k)}\}_{k\in\NN}$ is {\bf computably Cauchy} if there exists a computable function
$K:\Sigma\times\QQ_+\to\NN$ such that for all
$(\sigma,\epsilon)\in\Sigma\times\QQ_+$ and all $k,l\geq
K(\sigma,\epsilon)$,
\[
\text{$|S_{\sigma}^{(k)}(x) - S_{\sigma}^{(l)}(x)| < \epsilon$ on $\CC(\rho(\sigma),R(\sigma))$.}
\]
\end{definition}

\begin{lemma}\label{lemma:compConvDiff}
If $\{\S^{(k)}\}_{k\in\NN}$ is a computable sequence in $\Comp_{\an}(\rho,R)$ which converges computably to some $\S\in\Comp_{\an}(\rho,R)$, then there exists a computable function $K^\infty:\Delta(\Sigma)\times\QQ_+ \to \NN$ such that for all $(\sigma,\alpha ,\epsilon) \in\Delta(\Sigma) \times \QQ_+$, $k\geq K^\infty(\sigma,\alpha,\epsilon)$, and
$x\in[-\rho(\sigma),\rho(\sigma)]$, we have
\[
\left|\PDn{\alpha}{S_{\sigma}^{(k)}}{x}(x) - \PDn{\alpha}{S_\sigma}{x}(x)\right| < \epsilon.
\]
\end{lemma}

\begin{proof}
Fix a computable function $K:\Sigma\times\QQ_+\to\NN$
such that for all $(\sigma,\epsilon)\in\Sigma\times\QQ_+$ and all
$k\geq K(\sigma,\epsilon)$,
\[
\text{$|S_{\sigma}^{(k)}(x) - S_\sigma(x)| < \epsilon$ on $\CC(\rho(\sigma),R(\sigma))$.}
\]
Choose $c\in\QQ\cap(0,1)$.  For each $(\sigma,\alpha,\epsilon)\in\Delta(\Sigma)\times\QQ_+$, define
\[
K^\infty(\sigma,\alpha,\epsilon)
=
K\left(\sigma,
\frac{c^{|\alpha|} R(\sigma)^\alpha \epsilon}{\alpha!}
\right).
\]
For each $\sigma\in\Sigma$, let $C(\sigma) = C_1(\sigma)\times\cdots\times C_{\eta(\sigma)}(\sigma)$, where each $C_i(\sigma)$ is the counterclockwise oriented path consisting of the boundary of $\CC(\rho_i(\sigma), c R_i(\sigma))$, where $\rho_i(\sigma)$ and $R_i(\sigma)$ are the $i$th components of $\rho(\sigma)$ and $R(\sigma)$.  Then for all $(\sigma,\alpha,\epsilon)\in\Delta(\Sigma)\times\QQ_+$ and $x\in[-\rho(\sigma),\rho(\sigma)]$,
\begin{eqnarray*}
\left|
\PDn{\alpha}{S_\sigma}{x}(x) - \PDn{\alpha}{S^{(k)}_{\sigma}}{x}(x)
\right|
    & = &
    \left|
    \frac{\alpha!}{(2\pi i)^{\eta(\sigma)}} \int_{C(\sigma)}
    \frac{S_\sigma(z) - S^{(k)}_{\sigma}(z)}{(z-x)^{\alpha+\bar{1}}} dz
    \right|
    \\
    & < &
    \frac{\alpha!}{(2\pi)^{\eta(\sigma)}}
    \frac{\frac{c^{|\alpha|} R(\sigma)^\alpha \epsilon}{\alpha!}}
    {c^{|\alpha|+n} R(\sigma)^{\alpha+\bar{1}}}
    \left|\int_{C(\sigma)}dz\right|
    \\
    & = &
    \epsilon.
\end{eqnarray*}
\end{proof}

\begin{lemma}\label{lemma:compCauchy}
If $\{\S^{(k)}\}_{k\in\NN}$ is a computably Cauchy sequence in $\Comp_{\an}(\rho,R)$, then $\{\S^{(k)}\}_{k\in\NN}$ converges computably to some $\S\in\Comp_{\an}(\rho,R)$.
\end{lemma}

\begin{proof}
Fix a computable function $K:\Sigma\times\QQ_+\to\NN$ such that for all
$(\sigma,\epsilon)\in\Sigma\times\QQ_+$ and all $k,l\geq K(\sigma,\epsilon)$,
\[
\text{$|S_{\sigma}^{(k)}(x) - S_{\sigma}^{(l)}(x)| < \epsilon$ on $\CC(\rho(\sigma),R(\sigma))$.}
\]
For each $\sigma\in\Sigma$ and $x\in\CC(\rho(\sigma),R(\sigma))$, the sequence $\{S^{(k)}_{\sigma}(x)\}_{k\in\NN}$ is Cauchy, so it converges.  Define $\S = \{S_\sigma\}_{\sigma\in\Sigma}$ by
\[
S_\sigma(x) = \lim_{k\to\infty} S_{\sigma}^{(k)}(x)
\quad\text{for each $\sigma\in\Sigma$ and $x\in\CC(\rho(\sigma),R(\sigma))$.}
\]
Now, for all $(\sigma,\epsilon)\in\Sigma\times\QQ_+$, $k\geq K(\sigma,\frac{\epsilon}{2})$, and $x\in\CC(\rho(\sigma),R(\sigma))$,
\begin{equation}\label{eq:compCauchyConv}
|S_\sigma(x) - S^{(k)}_{\sigma}(x)|
=
\lim_{l\to\infty}|S^{(l)}_{\sigma}(x) - S^{(k)}_{\sigma}(x)|
\leq
\frac{\epsilon}{2}
< \epsilon,
\end{equation}
which shows that $\{\S^{(k)}\}_{k\in\NN}$ converges to $\S$ computably, as long as we can establish that $\S\in\Comp_{\an}(\rho,R)$.  Lemma \ref{lemma:compApprox} and \eqref{eq:compCauchyConv} shows that $\S$ is computably continuous.  To see that each $S_\sigma$ is holomorphic on $\CC(\rho(\sigma),R(\sigma))$, it suffices to fix $a\in(0,1)$ and show that $S_\sigma$ is holomorphic on $\CC(\rho(\sigma),a R(\sigma))$.  Choose $c\in(a,1)$, and define the contour $C(\sigma)$ from $c$ as in the proof of Lemma \ref{lemma:compConvDiff}.  Each function $S^{(k)}_{\sigma}$ is holomorphic, and therefore satisfies Cauchy's integral formula.  So for all $x\in\CC(\rho(\sigma), a R(\sigma))$,
\[
\left|
S_\sigma(x) -
\frac{1}{(2\pi i)^{\eta(\sigma)}} \int_{C(\sigma)}
\frac{S_\sigma(z)}{z-x} dz
\right|
\leq
\left|S_\sigma(x) - S^{(k)}_{\sigma}(x)\right|
+
\left|
\frac{1}{(2\pi i)^{\eta(\sigma)}} \int_{C(\sigma)}
\frac{S_\sigma(z) - S^{(k)}_{\sigma}(z)}{z-x} dz
\right|,
\]
and this upper bound tends to $0$ as $k\to\infty$.  So
\[
S_\sigma(x)
=
\frac{1}{(2\pi i)^{\eta(\sigma)}} \int_{C(\sigma)}
\frac{S_\sigma(z)}{z-x} dz
\]
on $\CC(\rho(\sigma),a R(\sigma))$, which proves that $S_\sigma$ is holomorphic on this set.
\end{proof}

%% file: parent.tex
\section{Concepts from the parent paper}
\label{s:parent}

As stated in the Introduction, it is expected that the reader has read Notation 0.2 and Sections 2-4 from our parent paper \cite{DJMcharDec}.  This section reviews the other definitions from \cite{DJMcharDec} that we shall need and introduces some additional related notation.  Throughout this section, $\S = \{S_\sigma\}_{\sigma\in\Sigma}$ denotes a family of functions in $\C$ with domain map $\rho$.

\begin{definition}
If $\F$ and $\G$ are $\QQ$-algebras of real-valued functions defined on the
sets $X$ and $Y$, respectively, then their {\bf tensor product}, $\F\otimes\G$,
is the $\QQ$-algebra of functions on $X\times Y$ generated by the functions
$h:X\times Y\to\RR$ of the form
\begin{enumerate}
\item
$h(x,y) = f(x)$ for some $f\in\F$, or

\item
$h(x,y) = g(y)$ for some $g\in\G$.
\end{enumerate}
\end{definition}

\begin{definition}\label{def:naturalStratBox}
The {\bf natural stratification} of a nonempty interval $I\subseteq\RR$ is the set denoted by $\Strat(I)$ which consists of $\Int(I)$ (when $I$ is nondegenerate) and any connected components of $\bd(I)$ that are contained in $I$.  For example, $\Strat(\{a\}) = \{\{a\}\}$, $\Strat((a,b]) = \{(a,b),\{b\}\}$, $\Strat([a,b]) = \{\{a\},(a,b),\{b\}\}$, and $\Strat(\RR) = \RR$.   The {\bf natural stratification} of a nonempty box $B = \prod_{i=1}^{n}B_i \subseteq\RR^n$ is defined by
\[
\Strat(B) = \left\{\prod_{i=1}^{n}C_i : \text{$C_i\in\Strat(B_i)$ for all $i\in\{1,\ldots,n\}$}\right\}.
\]
\end{definition}

\begin{definition}\label{def:Salgebra}
An {\bf $\S$-algebra on a natural domain} is a finite tensor product of algebras of the form $\QQ[x_1]$ (defined on $\RR$) or of the form $\QQ[x_1,\ldots,x_{\eta(\sigma)},\PDn{\alpha}{S_\sigma}{x}]$ (defined on $[-\rho(\sigma),\rho(\sigma)]$) for some $\sigma\in\Sigma$ and $\alpha\in\NN^{\eta(\sigma)}$.  If $\F$ is an $\S$-algebra on a natural domain $D\subseteq\RR^n$, and $B$ is a rational box manifold which is open in some member of the natural stratification of $D$, then we call $\F\Restr{B}$ an {\bf $\S$-algebra}.
\end{definition}

\begin{definition}\label{def:Spoly}
A function $P:D\to\RR^k$ is an {\bf $\S$-polynomial map} if $P\in\F^k$ for some  $\S$-algebra $\F$ on $D$.
\end{definition}

\begin{definition}\label{def:SpolyName}
If $P:D\to\RR^k$ is an $\S$-polynomial map, with $D\subseteq\RR^m$, then a {\bf name} for $P$ is a tuple
\begin{equation}\label{eq:name}
(p(x,y),\sigma,\alpha,\xi,\name(D))
\end{equation}
such that
\begin{equation}\label{eq:nameForP}
P(x) =
p\left(x,\PDn{\alpha(1)}{S_{\sigma(1)}}{x}\circ\Pi_{\xi(1)}(x), \ldots, \PDn{\alpha(n)}{S_{\sigma(n)}}{x}\circ\Pi_{\xi(n)}(x)\right),
\end{equation}
where
\begin{enumerate}{\setlength{\itemsep}{3pt}
\item
$n\in\NN$ and $p(x,y)\in\QQ[x,y]^k$, with $x = (x_1,\ldots,x_m)$ and $y=(y_1,\ldots,y_n)$;

\item
$\sigma$, $\alpha$ and $\xi$ are maps with domain
$\{1,\ldots,n\}$ such that
\begin{enumerate}{\setlength{\itemsep}{3pt}
\item
for each $i\in\{1,\ldots,n\}$, $\sigma(i)$ is a member of $\Sigma$,
$\alpha(i)$ is a member of $\NN^{\eta\circ\sigma(i)}$, and
$\xi(i)$ is an increasing map from
$\{1,\ldots,\eta\circ\sigma(i)\}$ into $\{1,\ldots,m\}$;

\item
the images of $\xi(1),\ldots,\xi(n)$ are disjoint (so
$\eta\circ\sigma(1)+\cdots+\eta\circ\sigma(n)\leq m$).
}\end{enumerate}
}\end{enumerate}
\end{definition}

For any $\S$-algebra $\F$ with domain $D$, there exist maps $\sigma$, $\alpha$, and $\xi$ such that
\begin{equation}\label{eq:names}
\left\{(p(x,y),\sigma,\alpha,\xi,\name(D)) : p(x,y)\in\QQ[x,y]^k\right\}
\end{equation}
is a collection of names for the members of $\F^k$.  The map from the set
\eqref{eq:names} to $\F^k$ that sends each name
$(p(x,y),\sigma,\alpha,\xi,\name(D))$ to $P(x)$, as defined by \eqref{eq:nameForP}, is surjective but is not necessarily injective.

Definition \ref{def:SpolyName} assumes we are given the semantic object $P$ and then specifies the form of the possible syntactic objects \eqref{eq:name} which name $P$ via \eqref{eq:nameForP}.  In the following notation, we specify sets of names for certain types of $\S$-polynomials in a purely syntactic way.

\begin{notation}\label{notation:SpolyName}
Consider $m,n\in\NN$ and $d\in\{0,\ldots,m-1\}$.  Define $\Delta_{m,n}^{d}(\rho)$ to be the set of all tuples
\[
(p(x,y),\sigma,\alpha,\xi,\name(D)),
\]
where
\begin{enumerate}{\setlength{\itemsep}{3pt}
\item
$\sigma$, $\alpha$, and $\xi$ are maps which have the properties given clause 2 of Definition \ref{def:SpolyName};

\item
$D$ is a rational box manifold, with $\dim(D) > d$, which is open in some member of the natural stratification of
\[
\left(\prod_{i=1}^{n} [-\rho\circ\sigma(i),\rho\circ\sigma(i)]\right)
\times
\RR^{\im(\xi)^c},
\]
where $\im(\xi) = \bigcup_{i=1}^{n}\im(\xi(i))$, $\im(\xi)^c = \{1,\ldots,m\}\setminus\im(\xi)$, and  $[-\rho\circ\sigma(i),\rho\circ\sigma(i)]$ is considered to be in $\RR^{\im(\xi(i))}$ for each $i\in\{1,\ldots,n\}$;

\item
$p(x,y)\in \QQ[x,y]^{\dim(D)-d}$, where $x=(x_1,\ldots,x_m)$ and $y = (y_1,\ldots,y_n)$.
}\end{enumerate}
(Note: The set $\Delta_{m,n}^{d}(\rho)$ may be empty.  For instance, this will be the case if $n > m$ and $\eta(\sigma) > 0$ for all $\sigma\in\Sigma$.)  Given any name $(p(x,y),\sigma,\alpha,\xi,\name(D)) \in \Delta^{d}_{m,n}(\rho)$, we will write
\begin{eqnarray*}
f_i(x)
    & = &
    \PDn{\alpha(i)}{S_{\sigma(i)}}{x}\circ\Pi_{\xi(i)}(x) \quad\text{for each $i\in\{1,\ldots,n\}$,} \\
f(x)
    & = &
    (f_1(x), \ldots, f_n(x)), \\
F(x)
    & = &
    (x, f(x)), \\
P(x)
    & = &
    p\circ F(x).
\end{eqnarray*}
Thus $P = p\circ F:D\to\RR^{\dim(D)-d}$ is the $\S$-polynomial map named by $(p(x,y),\sigma,\alpha,\xi,\name(D))$.  Write $\Delta_{m,n}^{d}(\S)$ for the set of all $\S$-polynomial maps named by the members of $\Delta_{m,n}^{d}(\rho)$.
\end{notation}

\begin{definition}\label{def:approxOracle}
The {\bf approximation oracle for $\S$} is an oracle which acts as a $C^\infty$ approximation algorithm for the family $\S$.
\end{definition}

Thus, given $\sigma\in\Sigma$, $\alpha\in\NN^{\eta(\sigma)}$, and names for a compact rational box $B\subseteq[-\rho(\sigma),\rho(\sigma)]$ and a rational open interval $I$, the approximation oracle for $\S$ stops if and only if
\begin{equation}\label{eq:approx}
\PDn{\alpha}{S_\sigma}{x}(B) \subseteq I.
\end{equation}

\begin{definition}\label{def:precOracle}
The {\bf precision oracle for $\S$} acts as follows:
\begin{quote}
Given the following data:
\begin{itemize}
\item
a name for an $\S$-polynomial map $P:D\to\RR^{\dim(D)-d}$, where $D\subseteq\RR^m$ and $d\in\{0,\ldots,\dim(D)-1\}$,

\item
a name for a bounded rational box manifold $C$ which is open in $D$ with $\cl(C) \subseteq D$,

\item
an injection $\lambda:\{1,\ldots,d\}\to\{1,\ldots,m\}$ such that $\Pi_\lambda(D)$ is open in $\RR^d$,

\item
$i\in\{1,\ldots,m\}$,
\end{itemize}
if $\IF_\lambda(P;C)$ holds, and we write $\varphi = (\varphi_1,\ldots,\varphi_m):\Pi_\lambda(C)\to C$ for the section of the projection $\Pi_\lambda:\RR^m\to\RR^d$ implicitly defined by $P\circ\varphi(x_\lambda) = 0$ on $\Pi_\lambda(C)$, then the oracle stops if and only if
\begin{equation}\label{eq:prec}
\text{$\varphi_i(x_\lambda) = 0$ for all $x_\lambda\in \Pi_\lambda(C)$.}
\end{equation}
\end{quote}
\end{definition}

The theory of $\RR_{\S}$ clearly decides the approximation and precision oracles for $\S$, since \eqref{eq:approx} and \eqref{eq:prec} are expressible as $\L_{\S}$-sentences which can be effectively constructed from the data given as input to the two oracles.  The main purpose of \cite{DJMcharDec} was to show that the converse is also true, which gives the following.

\begin{theorem}[{\cite[Theorem 0.1]{DJMcharDec}}]\label{thm:charDec}
The theory of $\RR_{\S}$ is decidable if and only if the approximation and precision oracles for $\S$ are decidable.
\end{theorem}

%% file: genDec.tex
\section{A decision procedure for $\Th(\RR_{\S})$ when $\S$ is generic and computably $C^\infty$}
\label{s:genDec}

Throughout this section, $\S = \{S_\sigma\}_{\sigma\in\Sigma}$ denotes a family of functions in $\C$ with domain map $\rho$.  In this section we define what it means for $\S$ to be generic, and we prove Theorem \ref{introThm:genDec}.  We begin by rephrasing the precision oracle for $\S$ in the language of algebraic geometry.

For us, the word {\bf variety} shall mean a real affine variety over $\QQ$.
To any set of polynomials $Q\subseteq\QQ[x]$, with $x=(x_1,\ldots,x_n)$, we associate the variety
\[
\VV(Q) = \{a\in\RR^n : \text{$q(a) = 0$ for all $q(x)\in Q$}\}.
\]
For any polynomial maps $q_1\in\QQ[x]^{k_1},\ldots,q_l\in\QQ[x]^{k_l}$, if $Q$ is the set of all components of the maps $q_1,\ldots,q_l$, then we write $\lb q_1,\ldots,q_l\rb$ for the ideal of $\QQ[x]$ generated by $Q$, and we write $\VV(q_1,\ldots,q_l) = \VV(Q)$.  To any set $A \subseteq \RR^n$ we associate the following:
\begin{renumerate}
\item
the ideal of $A$,
\[
\II(A) = \{q(x)\in\QQ[x] : \text{$q(a) = 0$ for all $a\in A$}\};
\]

\item
the Zariski closure of $A$,
\[
\Zar(A) = \VV(\II(A));
\]

\item
the ring of regular functions on $A$,
\[
\QQ[A] = \{g:A\to\RR : \text{$g = q$ on $A$ for some $q(x)\in\QQ[x]$}\}.
\]
\end{renumerate}
Clearly $\QQ[A]\cong \QQ[x]/\II(A)$, and this ring is an integral domain if and only if the ideal $\II(A)$ is prime, or equivalently, $\Zar(A)$ is irreducible.  In this case we let $\QQ(A)$ denote the fraction field of $\QQ[A]$.

The {\bf membership problem} for an ideal $I \subseteq \QQ[x]$ is the problem of deciding, when given $q(x)\in\QQ[x]$, whether $q(x)\in I$ or $q(x)\notin I$.

\begin{definition}\label{def:idealMemOracle}
The {\bf ideal membership oracle for $\S$} acts as follows:
\begin{quote}
When given the following data:
\begin{itemize}
\item
a name $(p(x,y),\sigma,\alpha,\xi,\name(D)) \in \Delta_{m,n}^{d}(\rho)$ for an $\S$-polynomial map $P = p\circ F:D\to\RR^{\dim(D)-d}$,

\item
a name for a bounded rational box manifold $C$ which is open in $D$ with $\cl(C) \subseteq D$,

\item
an injection $\lambda:\{1,\ldots,d\}\to\{1,\ldots,m\}$ such that $\Pi_\lambda(D)$ is open in $\RR^d$,
\end{itemize}
if $\IF_\lambda(P;C)$ holds, and we write $\varphi:\Pi_\lambda(C)\to C$ for the section of the projection $\Pi_\lambda:\RR^m\to\RR^d$ implicitly defined by $P\circ\varphi(x_\lambda) = 0$ on $\Pi_\lambda(C)$, then the oracle decides the membership problem for the ideal $\II(\im(F\circ\varphi))$, where $F$ is defined as in Notation \ref{notation:SpolyName} and $\im(F\circ\varphi)$ if the image of the map $F\circ\varphi$.
\end{quote}
\end{definition}

\begin{lemma}\label{lemma:idealMemOracle}
The theory of $\RR_{\S}$ is decidable if and only if the approximation and ideal membership oracles for $\S$ are decidable.
\end{lemma}

\begin{proof}
Assume that the approximation and ideal membership oracles for $\S$ are decidable.  Using the notation of Definition \ref{def:precOracle}, suppose that $\IF_\lambda(P;C)$ holds, where $P = p\circ F$, and write $\varphi = (\varphi_1,\ldots,\varphi_m):\Pi_\lambda(C)\to C$ for the section implicitly defined by $P(x) = 0$ on $C$.  Let $i\in\{1,\ldots,m\}$.  Note that $\varphi_i(x_\lambda) = 0$ for all $x_\lambda\in\Pi_\lambda(C)$ if and only if $x_i \in \II(\im(F\circ\varphi))$, which can be decided.  Therefore the precision oracle for $\S$ is decidable, and hence the theory of $\RR_{\S}$ is decidable by Theorem \ref{thm:charDec}.

Conversely, assume that the theory of $\RR_{\S}$ is decidable.  Then the approximation oracle for $\S$ is decidable.  Now, using the notation of Definition \ref{def:idealMemOracle}, suppose that $\IF_\lambda(P;C)$ holds.  Given $q(x)\in\QQ[x]$, we can effectively write down an $\L_{\S}$-sentence stating that ``$q(x) = 0$ for all $x\in\im(F\circ\varphi)$'', which is equivalent to ``$q(x)\in\II(\im(F\circ\varphi))$''.  We can therefore decide the ideal membership oracle for $\S$.
\end{proof}

We now assume that the approximation oracle for $\S$ is decidable and try to see under what conditions we may also decide the ideal membership oracle, or at least special instances of this oracle.
\begin{equation}\label{eq:IFsetup}
\text{\parbox{5in}{
Fix the following  data:
\begin{itemize}
\item
a name $(p(x,y),\sigma,\alpha,\xi,\name(D))\in\Delta^{d}_{m,n}(\rho)$ for an $\S$-polynomial map $P = p\circ F:D\to\RR^{\dim(D)-d}$;
\item
an injection $\lambda:\{1,\ldots,d\}\to\{1,\ldots,m\}$ such that $\Pi_\lambda(D)$ is open in $\RR^d$;
\item
a bounded rational box manifold $C$ such that $C$ is open in $D$, $\cl(C)\subseteq D$, and $\IF_\lambda(P;C)$ holds.
\end{itemize}
Let $\varphi:\Pi_\lambda(C)\to C$ be the section of the projection $\Pi_\lambda$ implicitly defined by $P\circ\varphi(x_\lambda) = 0$ on $\Pi_\lambda(C)$.  Fix $E\subseteq\{1,\ldots,m\}$ such that $D = U\times\{u\}$ for some open rational box $U\subseteq\RR^E$ and $u\in\QQ^{E^c}$, and fix an injection $\lambda':\{1,\ldots,\dim(D)-d\}\to\{1,\ldots,m\}$ such that $\im(\lambda)\cup\im(\lambda') = E$.
}}
\end{equation}
Note that since $\IF_\lambda(P;C)$ holds, we have $\im(\lambda)\subseteq E$ and $\im(\lambda)\cap\im(\lambda') = \emptyset$.

For each $a\in\Pi_\lambda(C)$, we have $\det\PD{}{P}{x_{\lambda'}}\Restr{x=\varphi(a)}\neq 0$, so $\rank \PD{}{P}{x_E}\Restr{x=\varphi(a)} = \dim(D) - d$.  Write $\varphi_E = \Pi_E\circ\varphi$ and $\varphi_{E^c} = \Pi_{E^c}\circ\varphi$.  Note that $\varphi_{E^c}(a) = u$, so $\PD{}{P}{x_E}\Restr{x=\varphi(a)} = \PD{}{(P(x_E,u))}{x_E}\Restr{x_E = \varphi_E(a)}$.  Since $F(x_E,u) = (x_E,u,f(x_E,u))$, the chain rule gives
\[
\left.\PD{}{(P(x_E,u))}{x_E}\right|_{x_E = \varphi_E(a)}
=
\left.\left(\begin{matrix}
\PD{}{p}{x_E} \PD{}{p}{y}
\end{matrix}\right)\right|_{(x,y)=F\circ\varphi(a)}
\left.\left(\begin{matrix}
\id \\
\PD{}{f}{x_E}
\end{matrix}\right)\right|_{x=\varphi(a)}.
\]
Therefore
\begin{equation}\label{eq:rankp}
\left. \rank\PD{}{p}{(x_E,y)}\right|_{(x,y)=F\circ\varphi(a)} = \dim(D) - d,
\end{equation}
and hence
\begin{equation}\label{eq:rankpu}
\left. \rank \PD{}{(p,x_{E^c} - u)}{(x,y)}\right|_{(x,y) = F\circ\varphi(a)} = m-d.
\end{equation}
Thus every point of $\im(F\circ\varphi)$ is a nonsingular point of the variety $\VV(p(x,y), x_{E^c} - u)$.  Because $\im(F\circ\varphi)$ is also connected, $\im(F\circ\varphi)$ is contained in exactly one irreducible component of $\VV(p(x,y), x_{E^c} - u)$, which we shall call $X$, and $\im(F\circ\varphi)$ is disjoint from all other irreducible components of $\VV(p(x,y), x_{E^c} - u)$.

We claim that we can effectively find a set of generators for $\II(X)$.  Indeed, Becker and Weispfennig \cite[Theorem 8.101]{BW} states that we can effectively find sets of generators for all associated primes of the ideal $\lb p(x,y), x_{E^c} - u\rb$, and can thereby find sets of generators for all isolated primes, which we denote by $\mathfrak{p}_1\ldots,\mathfrak{p}_k$.  (We reference \cite{BW} because it is a widely available textbook, and its algorithm is currently implemented in the computer algebra systems AXIOM and REDUCE.  See Cox, Little and O'Shea \cite[pgs. 205-206]{CLO} for a more extensive list of references.)  Choose $a\in \QQ^d\cap\Pi_\lambda(C)$.  If $j\in\{1,\ldots,k\}$ is such that $F\circ\varphi(a)\notin\VV(\mathfrak{p}_j)$, this fact can be effectively verified since $F\circ\varphi(a)$ is a computable point and we have a set of generators for $\mathfrak{p}_j$.  But $\II(X) = \mathfrak{p}_i$ for exactly one $i\in\{1,\ldots,k\}$, and $F\circ\varphi(a)\notin\VV(\mathfrak{p}_j)$ for all $j\in\{1,\ldots,k\}$ not equal to $i$, so a set of generators for $\II(X)$ can be found by process of elimination.

If one knows a set of generators for an ideal, one has solved the membership problem for this ideal since any set of generators can be expanded to a Gr\"{o}bner basis, from which ideal membership can be tested by a suitable division algorithm.  Therefore if we are lucky enough to have $\II(\im(F\circ\varphi)) = \II(X)$, we have solved the membership problem for $\II(\im(F\circ\varphi))$.  This is always the case when $n=0$, for we then have $F(x) = x$ and $P(x) = p(x)$,
and hence $\II(\im(\varphi)) = \II(X)$.

This observation gives a new proof of Tarski's theorem based on Theorem \ref{thm:charDec}, which is proven in \cite{DJMcharDec} by a model completeness construction.  This contrasts with the previously known proofs of Tarski's theorem, which use quantifier elimination.

\begin{corollary}[{Tarski's Theorem \cite{Tarski}}]\label{cor:TarskiTheorem}
The theory of the real field is decidable.
\end{corollary}

\begin{proof}
Here we have $\S = \emptyset$.  The approximation oracle for $\emptyset$ is trivially decidable.  In the discussion above, necessarily $n = 0$ when $\S=\emptyset$, so the ideal membership oracle is also decidable.
\end{proof}

But in general, when $\S\neq \emptyset$ and $n > 0$, we only have $\II(\im(F\circ\varphi))\supseteq\II(X)$ since $\im(F\circ\varphi) \subseteq X$, so how can we tell when the two ideals are actually the same?  For our purposes, it is most convenient to rephrase this question in terms of the transcendence degree of the field $\QQ(\im(F\circ\varphi))$ over $\QQ$.  For any field extension $K\subseteq L$, we shall write $\td_K L$ for the transcendence degree of $L$ over $K$.

Recall that if $A$ is any subset of $\RR^n$ such that $\Zar(A)$ is an irreducible variety, we have the following three ways of characterizing $\dim\II(A)$, the dimension of the ideal $\II(A)$:
\begin{enumerate}{\setlength{\itemsep}{3pt}
\item[(D1)]
The number $\dim\II(A)$ is the greatest $k\in\NN$ such that there exists a strictly increasing chain of prime ideals $\II(A) = P_0 \subset P_1 \subset \cdots P_k \subset \QQ[x]$.

\item[(D2)]
$\dim\II(A) = \td_{\QQ} \QQ(A)$.

\item[(D3)]
The number $\dim\II(A)$ equals the dimension of the tangent space to $\Zar(A)$ at any nonsingular point of $\Zar(A)$.
}\end{enumerate}

\begin{notation}
Suppose $M\subseteq\RR^n$ is a $\C$-analytic manifold and that $g:M\to\RR^m$ is a $\C$-analytic embedding.  Let $a\in M$.  Define $M_a$ to be the germ of the set $M$ at $a$, define $g_a:M_a\to\RR^{m}_{g(a)}$ to be the germ of the map $g$ at $a$, and define $\im(g_a)$ to be the germ of the manifold $\im(g)$ at $g(a)$.
\end{notation}

In order to study $\td_{\QQ}\QQ(\im(F\circ\varphi))$, it is convenient to conceptualize the field $\QQ(\im(F\circ\varphi))$ in a number of different isomorphic ways.

\begin{definitions}\label{def:isomRings}
Consider the setup given in \eqref{eq:IFsetup}.  Choose a point $a\in\QQ^d\cap\Pi_\lambda(C)$.   We define the following four rings:
\begin{itemize}
\item
$\QQ[\im(F\circ\varphi)]$

\item
$\QQ[\im(F\circ\varphi_a)]$

\item
$\QQ[F\circ\varphi(x_\lambda)]$

\item
$\QQ[F\circ\varphi_a(x_\lambda)]$
\end{itemize}
We have already defined $\QQ[\im(F\circ\varphi)]$ to be the ring of regular functions on the set $\im(F\circ\varphi)$, namely,
\[
\QQ[\im(F\circ\varphi)] = \{g:\im(F\circ\varphi)\to\RR : \text{$g = q$ on $\im(F\circ\varphi)$ for some $q\in\QQ[x,y]$}\}.
\]
Define $\QQ[\im(F\circ\varphi_a)]$ to be the set of germs at $F\circ\varphi(a)$ of the functions in $\QQ[\im(F\circ\varphi)]$.  Define $\QQ[F\circ\varphi(x_\lambda)]$ to be the ring of all functions on $\Pi_\lambda(C)$ which are rational polynomials in the function $F\circ\varphi$, namely,
\[
\QQ[F\circ\varphi(x_\lambda)] = \{q\circ F\circ\varphi:\Pi_\lambda(C)\to\RR : q(x,y)\in\QQ[x,y]\}.
\]
Define $\QQ[F\circ\varphi_a(x_\lambda)]$ to be the set of germs at $a$ of the functions in $\QQ[F\circ\varphi(x_\lambda)]$.
\end{definitions}

\begin{remark}\label{rmk:isomRings}
All four rings in Definition \ref{def:isomRings} are isomorphic to $\QQ[x,y]/\II(\im(F\circ\varphi))$, and the ideal $\II(\im(F\circ\varphi))$ is prime.  So these four isomorphic rings are integrals domains, and thus have isomorphic fields of fractions, which we write as $\QQ(\im(F\circ\varphi))$, $\QQ(\im(F\circ\varphi_a))$, $\QQ(F\circ\varphi(x_\lambda))$, $\QQ(F\circ\varphi_a(x_\lambda))$.
\end{remark}

\begin{proof}
The fact that $\QQ[\im(F\circ\varphi)]$ is isomorphic to $\QQ[x,y]/\II(\im(F\circ\varphi))$ is clear.

Define a ring homomorphism from $\QQ[x,y]/\II(\im(F\circ\varphi))$ to $\QQ[F\circ\varphi(x_\lambda)]$ by
\[
q(x,y) + \II(\im(F\circ\varphi)) \mapsto q\circ F\circ\varphi,
\]
for all $q(x,y)\in \QQ[x,y]$.  This map is clearly surjective, and its kernel is $\II(\im(F\circ\varphi))$, which is the zero element of $\QQ[x,y]/\II(\im(F\circ\varphi))$, so the map is an isomorphism.

Now define a ring homomorphism from $\QQ[F\circ\varphi(x_\lambda)]$ to $\QQ[F\circ\varphi_a(x_\lambda)]$ by
\[
q\circ F\circ\varphi \mapsto q\circ F\circ\varphi_a,
\]
for all $q(x,y)\in\QQ[x,y]$.  Again, this map is clearly surjective.  Because $\C$ is quasianalytic and $\Pi_\lambda(C)$ is a connected open set, the $\C$-analytic function $q\circ F\circ \varphi$ vanishes identically on $\Pi_\lambda(C)$ if and only if it vanishes in a neighborhood of $a$, which means that its germ at $a$ is $0$.  Therefore the kernel of our homomorphism is $\{0\}$, so the map is an isomorphism.

Since $\im(F\circ\varphi)$ is a connected $\C$-analytic manifold, a nearly identical proof shows that $\QQ[\im(F\circ\varphi)]$ is isomorphic to $\QQ[\im(F\circ\varphi_a)]$.

Finally, to see that $\II(\im(F\circ\varphi))$ is prime, note that if $q,r\in\QQ[x,y]$ are such that the product of $qr$ vanishes identically on $\im(F\circ\varphi)$, then $q$ or $r$ must vanish identically on $\im(F\circ\varphi)$ since this set is a connected $\C$-analytic manifold and $\C$ is quasianalytic.
\end{proof}

\begin{lemma}\label{lemma:tdBounds}
Consider the setup given in \eqref{eq:IFsetup}, and $X$ be the unique irreducible component of $\VV(p(x,y), x_{E^c} - u)$ containing $\im(F\circ\varphi)$.
\begin{enumerate}{\setlength{\itemsep}{3pt}
\item
We have $\td_\QQ \QQ(\im(F\circ\varphi))\leq n+d$, and equality holds if and only if $\II(\im(F\circ\varphi)) = \II(X)$.

\item
There exists a dense, open subset $V$ of $\Pi_\lambda(C)$ such that for all $a\in \QQ^d\cap V$,
\[
\td_\QQ \QQ(\im(F\circ\varphi)) \geq \td_\QQ \QQ(F\circ\varphi(a)) + d,
\]
where $\QQ(F\circ\varphi(a))$ is the field generated by the components of the point $F\circ\varphi(a)$.
}\end{enumerate}
\end{lemma}

\begin{proof}
We first prove 1.  We have $\im(F\circ\varphi) \subseteq X$, so $\II(\im(F\circ\varphi) \supseteq \II(X)$.  Using (D1) and the fact that the ideals $\II(\im(F\circ\varphi))$ and $\II(X)$ are prime, it follows that $\dim\II(\im(F\circ\varphi))\leq\dim\II(X)$ and that the ideals $\II(\im(F\circ\varphi))$ and $\II(X)$ are equal if and only if they have the same dimension.  Using (D2) we see that   $\dim\II(\im(F\circ\varphi)) =\td_\QQ \QQ(\im(F\circ\varphi))$.  Using (D3), we see that $\dim\II(X) = n + d$ because for every $a\in \im(F\circ\varphi)$, $a$ is a nonsingular point of $X$ and the tangent space to $X$ at $a$ is of dimension $n+d$ by \eqref{eq:rankpu}.  Statement 1 follows.

To prove 2, note that $\QQ(\im(F\circ\varphi)) \cong \QQ(F\circ\varphi(x_\lambda))$, $\QQ(x_\lambda)$ is a subfield of $\QQ(F\circ\varphi(x_\lambda))$, and $\td_\QQ \QQ(x_\lambda) = d$, so statement 2 holds if and only if there exists a dense open subset $V$ of $\Pi_\lambda(C)$ such that for all $a\in \QQ^d\cap V$,
\[
\td_{\QQ(x_\lambda)}\QQ(F\circ\varphi(x_\lambda)) \geq \td_\QQ \QQ(F\circ\varphi(a)).
\]
Let $t = \td_{\QQ(x_\lambda)}\QQ(F\circ\varphi(x_\lambda))$, and let $\Gamma$ denote the set of all increasing maps $\gamma :\{1,\ldots,t+1\}\to\{1,\ldots,n+m\}\setminus\im(\lambda)$.
(Observe that $\Gamma$ is nonempty: we have $t\leq n$ by Statement 1, and also $d < m$, so $t+1\leq m+n-d$, and the set $\{1,\ldots,m+n\}\setminus\im(\lambda)$ has size $m+n-d$.)  For each $\gamma\in\Gamma$ let $z_\gamma := \Pi_{\gamma}(x,y)$, and note that $x_\lambda$ and $z_\gamma$ are disjoint tuples of variables.  For each $\gamma\in\Gamma$ there exists a nonzero polynomial $q_\gamma(x_\lambda,z_\gamma)\in\QQ[x_\lambda,z_\gamma]$ such that $q_\gamma\circ F\circ\varphi(x_\lambda) = 0$ for all $x_\lambda\in \Pi_\lambda(C)$, where we consider $q_\gamma$ to be a function on $\RR^{m+n}$ which only depends on the coordinates $(x_\lambda,z_\gamma)$.  The set of all $a\in\Pi_\lambda(C)$ such that $q_\gamma(a,z_\gamma)$ is the zero polynomial in $z_\gamma$ is a proper subvariety of $\Pi_\lambda(C)$, so
\[
V = \bigcap_{\gamma\in\Gamma} \{a\in\Pi_\lambda(C) : q_\gamma(a,z_\gamma) \not\equiv 0\}
\]
is a dense open subset of $\Pi_\lambda(C)$.  For each $a\in \QQ^d\cap V$, we have $\td_\QQ \QQ(F\circ\varphi(a))\leq t$ because $\Pi_\lambda(F\circ\varphi(a)) = a\in\QQ^d$ and $q_\gamma\circ F\circ\varphi(a) = 0$ for all $\gamma\in\Gamma$.
\end{proof}

Because of Lemma \ref{lemma:tdBounds}, we are interested in being able to determine when
\[
\td_{\QQ}\QQ(\im(F\circ\varphi)) = n+d.
\]
To do this, we will first give a necessary condition for this to be the case, and will then specially construct $\S$ so that this necessary condition is also a sufficient condition.  To help understand the idea behind this necessary condition, consider the following example.

\begin{example}\label{ex:distCond}
Consider a name $(p(x,y),\sigma,\alpha,\xi,\name(D)) \in \Delta^{1}_{3,2}(\rho)$ for an $\S$-polynomial map $P = p\circ F:D\to\RR^2$, where $D$ is open in $\RR^3$.   Write $x = (x_1,x_2,x_3)$, $y = (y_1,y_2)$, $p = (p_1,p_2)$, and $P = (P_1,P_2)$, and suppose that $F$ is of the form
\[
F(x) = (x_1, x_2, x_3, f_1(x_1), f_2(x_2)).
\]
Suppose that $\IF_\lambda(P;C)$ holds for some $C\subseteq D$ and for the projection
\[
\Pi_{\lambda}(x_1,x_2,x_2) = x_3.
\]
Let $\varphi:\Pi_\lambda(C)\to C$ be the implicitly defined section, and write
\[
\varphi(x_3) = (\varphi_1(x_3), \varphi_2(x_3), x_3).
\]
Thus the matrix
\begin{equation}\label{eq:Pjacobian}
\PD{}{P}{(x_1,x_2)}(x)
=
\left(
\begin{matrix}
\PD{}{p_1}{x_1}\circ F(x) + \PD{}{p_1}{y_1}\circ F(x) \PD{}{f_1}{x_1}(x_1)
    &
    \PD{}{p_1}{x_2}\circ F(x) + \PD{}{p_1}{y_2}\circ F(x) \PD{}{f_2}{x_2}(x_2)
\\
\PD{}{p_2}{x_1}\circ F(x) + \PD{}{p_2}{y_1}\circ F(x) \PD{}{f_1}{x_1}(x_1)
    &
    \PD{}{p_2}{x_2}\circ F(x) + \PD{}{p_2}{y_2}\circ F(x) \PD{}{f_2}{x_2}(x_2)
\end{matrix}
\right)
\end{equation}
is nonsingular on $\im(\varphi)$.

\begin{description}{\setlength{\itemsep}{5pt}
\item[Assumption]
Assume that $\sigma(1) = \sigma(2)$, $\alpha(1) = \alpha(2)$, and $\varphi_1 = \varphi_2$.

\item[Question]
Lemma \ref{lemma:tdBounds} shows that $\td_{\QQ}\QQ(\im(F\circ\varphi)) \leq 2+1 = 3$.  Is it possible for equality to hold?

\item[Answer] No.
}\end{description}
The assumption that $\sigma(1) = \sigma(2)$ and $\alpha(1) = \alpha(2)$ means that $f_1$ and $f_2$ are the same derivative of the same function in $\S$, and hence $f_1 = f_2$.  The additional assumption that $\varphi_1 = \varphi_2$ says much more, for it allows us to perform the substitution $(x_2,y_2) = (x_1,y_1)$, from which we can see that $\td_{\QQ}\QQ(\im(F\circ\varphi)) < 3$, as follows.

Define
\begin{eqnarray*}
\tld{p}(x_1,x_3,y_1)
    & = &
    p(x_1,x_1,x_3,y_1,y_1),
\\
\bar{F}(x_1,x_3)
    & = &
    F(x_1,x_1,x_3,f_1(x_1),f_1(x_1)),
\\
\tld{P}(x_1,x_3)
    & = &
    \tld{p}(x_1,x_1,x_3,f_1(x_1),f_1(x_1)),
\\
\bar{\varphi}(x_3)
    & = &
    (\varphi_1(x_3),x_3),
\end{eqnarray*}
and write $\tld{p} = (\tld{p}_1,\tld{p}_2)$ and $\tld{P} = (\tld{P}_1,\tld{P}_2)$.  Note that $\PD{}{\tld{P}}{x_1}$ is the $2\times 1$ matrix obtained by summing the columns of \eqref{eq:Pjacobian}.  Now fix $a\in\QQ^d\cap\Pi_\lambda(C)$.  Since \eqref{eq:Pjacobian} is nonsingular, $\PD{}{\tld{P}}{x_1}\circ\varphi(a)$ must have rank 1, so either $\PD{}{\tld{P}_1}{x_1}\circ\varphi(a)\neq 0$ or $\PD{}{\tld{P}_2}{x_1}\circ\varphi(a)\neq 0$.  Both cases are symmetric, so assume that $\PD{}{\tld{P}_1}{x_1}\circ\varphi(a)\neq 0$.  Define
\[
\bar{p} = \tld{p}_1
\quad\text{and}\quad
\bar{P} = \tld{P}_1;
\]
thus $\bar{P} = \bar{p}_1\circ\bar{F}$.
Then $\bar{\varphi}_a$ is a germ of the section of the projection $(x_1,x_3)\mapsto x_3$ implicitly defined by the nonsingular equation $\bar{P}(x_1,x_3) = 0$, so by Lemma \ref{lemma:tdBounds},
\[
\td_{\QQ}\QQ(\bar{F}\circ\bar{\varphi}_a(x_\lambda)) \leq 1 + 1 = 2.
\]
For any $q(x,y)\in\QQ[x,y]$, define $\bar{q}(x_1,x_3,y_1) = q(x_1,x_1,x_3,y_1,y_1)$, and note that $q\circ F\circ\varphi = 0$ if and only if $\bar{q}\circ\bar{F}\circ\bar{\varphi} = 0$, that is, $q\in\II(\im(F\circ\varphi))$ if and only if $\bar{q}\in\II(\im(\bar{F}\circ\bar{\varphi}))$.  Thus the map $q\mapsto\bar{q}$ is an isomorphism from $\QQ(F\circ\varphi_a(x_\lambda))$ to $\QQ(\bar{F}\circ\bar{\varphi}_a(x_\lambda))$, and hence
\[
\td_{\QQ}\QQ(\im(F\circ\varphi)) = \td_{\QQ}\QQ(\bar{F}\circ\bar{\varphi}_a(x_\lambda)) \leq 2,
\]
as claimed.
\end{example}

This example motivates the following definition.

\begin{definition}\label{def:distCond}
Consider the setup given in \eqref{eq:IFsetup}.  We say that $\varphi$ satisfies the {\bf $(\sigma,\alpha,\xi)$-distinctness condition} if for all $j_1,j_2\in\{1,\ldots,n\}$ such that $\sigma(j_1) = \sigma(j_2)$, $\alpha(j_1) = \alpha(j_2)$, and $j_1\neq j_2$, the functions $\Pi_{\xi(j_1)}\circ\varphi$ and $\Pi_{\xi(j_2)}\circ\varphi$ are not identically equal.
\end{definition}

\begin{proposition}\label{prop:distCond}
Consider the setup given in \eqref{eq:IFsetup}.  If $\td_{\QQ}\QQ(\im(F\circ\varphi)) = n+d$, then $\varphi$ satisfies the $(\sigma,\alpha,\xi)$-distinctness condition.
\end{proposition}

Proposition \ref{prop:distCond} will be proven using the substitution technique demonstrated in Example \ref{ex:distCond}.  In order to write down a general proof, we introduce the following combinatorial terminology.

\begin{definition}\label{def:equivRelations}
Consider the setup given in \eqref{eq:IFsetup}.  We define three equivalence relations on the set $\{1,\ldots,n\}$.  For each $j_1,j_2\in\{1,\ldots,n\}$, define
\begin{eqnarray*}
j_1\approx_{\bar{\sigma}} j_2
    & \text{if and only if}
    &\sigma(j_1) = \sigma(j_2);
\\
j_1\approx_{(\sigma,\alpha)}j_2
    & \text{if and only if}
    & \text{$\sigma(j_1) = \sigma(j_2)$ and $\alpha(j_1) = \alpha(j_2)$};
\\
j_1\approx_{(\sigma,\alpha,\xi,\varphi)}j_2
    & \text{if and only if}
    & \text{$\sigma(j_1) = \sigma(j_2)$, $\alpha(j_1) = \alpha(j_2)$ and $\Pi_{\xi(j_1)}\circ\varphi = \Pi_{\xi(j_2)}\circ\varphi$}.
\end{eqnarray*}
Recall from Notation \ref{notation:SpolyName} that $\im(\xi) = \bigcup_{j=1}^{n}\im(\xi(j))$.  Define $\gamma:\{1,\ldots,m\} \to \{0,\ldots,n\}$ by
\[
\gamma(i)
=
\begin{cases}
j,
    & \text{if $i\in\im(\xi(j))$ for some (necessarily unique) $j\in\{1,\ldots,n\}$,} \\
0,
    & \text{if $i\in\{1,\ldots,m\}\setminus\im(\xi)$.}
\end{cases}
\]
Any equivalence relation $\approx$ on the set $\{1,\ldots,n\}$ which refines $\approx_{\sigma}$ induces an equivalence relation $\sim$ on the set $\{1,\ldots,m\}$ in the following way: if $i_1,i_2\in\im(\xi)$, define
\[
i_1 \sim i_2
\]
if and only if $\gamma(i_1) \approx \gamma(i_2)$ and there exists $k\in
\{1,\ldots,\eta\circ\sigma\circ\gamma(i_1)\}$ such that
\[
i_1 = \xi(\gamma(i_1))(k)\quad\text{and}\quad i_2 = \xi(\gamma(i_2))(k);
\]
if $i_1,i_2\in\{1,\ldots,m\}$ are such that $i_1\not\in\im(\xi)$ or $i_2\not\in\im(\xi)$, define
\[
i_1 \sim i_2 \quad \text{if and only if} \quad  i_1 = i_2.
\]
We write $\sim_{(\sigma,\alpha)}$ and $\sim_{(\sigma,\alpha,\xi,\varphi)}$ for the equivalence relations on $\{1,\ldots,m\}$ induced by $\approx_{\sigma,\alpha)}$ and $\approx_{(\sigma,\alpha,\xi,\varphi)}$, respectively.
\end{definition}

\begin{remarks}\label{rmk:equivRelations}
\hfill
\begin{enumerate}{\setlength{\itemsep}{5pt}
\item
The function $\varphi$ satisfies the $(\sigma,\alpha,\xi)$-distinctness condition if and only if every $\approx_{(\sigma,\alpha,\xi)}$-equivalence class has size $1$.

\item
The equivalence relation $\approx_{(\sigma,\alpha,\xi,\varphi)}$ refines $\approx_{(\sigma,\alpha)}$.  It follows that $\sim_{(\sigma,\alpha,\xi,\varphi)}$ refines $\sim_{(\sigma,\alpha)}$.

\item
Introduce a new variable $y_0$.  The map $\gamma:\{1,\ldots,m\}\to\{0,1,\ldots,n\}$ is defined so that for all $i\in\{1,\ldots,m\}$,
\begin{equation}\label{eq:differentiate}
\PD{}{P}{x_i}(x) = \PD{}{p}{x_i}(x,f(x)) + \PD{}{p}{y_{\gamma(i)}}(x,f(x))\PD{}{f_{\gamma(i)}}{x_i}(x),
\end{equation}
with the understanding that $\PD{}{p}{y_0} = 0$ since $p(x,y)$ does not depend $y_0$, as it is a newly introduced variable.  Thus \eqref{eq:differentiate} is to be interpreted as $\PD{}{P}{x_i}(x) = \PD{}{p}{x_i}(x,f(x))$ when $i\in\{1,\ldots,m\}\backslash\im(\xi)$.

\item
If $i_1\sim_{(\sigma,\alpha,\xi,\varphi)} i_2$, then $f_{\gamma(i_1)} = f_{\gamma(i_2)}$, and $x_{i_1} = x_{i_2}$ for all $x = (x_1,\ldots,x_m)\in\im(\varphi)$.

\item
If $i_1,i_2\in\im(\lambda)$ and $i_1\neq i_2$, then $i_1 \not\sim_{(\sigma,\alpha,\xi,\varphi)} i_2$.

\begin{proof}
Let $i_1,i_2\in\im(\lambda)$, and suppose that $i_1 \sim_{(\sigma,\alpha,\xi,\varphi)} i_2$.  Then $x_{i_1} = x_{i_2}$ for all $x\in\im(\varphi)$.  But $\Pi_\lambda(\im(\varphi)) = \Pi_\lambda(C)$, which is open in $\RR^d$, so necessarily $i_1 = i_2$.
\end{proof}
}\end{enumerate}
\end{remarks}

\begin{lemma}\label{lemma:distCond}
Consider the setup given in \eqref{eq:IFsetup}.
Fix $a\in\QQ^d\cap\Pi_\lambda(C)$.  Let $\bar{m}$ be the number of $\sim_{(\sigma,\alpha,\xi,\varphi)}$-equivalence classes on $\{1,\ldots,m\}$, and let $\bar{n}$ be the number of $\approx_{(\sigma,\alpha,\xi,\varphi)}$-equivalence classes on $\{1,\ldots,n\}$.  Then there exist
\begin{itemize}
\item
coordinate projections $\Pi':\RR^m\to\RR^{\bar{m}}$ and $\Pi'':\RR^n\to\RR^{\bar{n}}$;
\begin{quote}
Write $\bar{x} = \Pi'(x)$, $\bar{y} = \Pi''(y)$, $\bar{C} = \Pi'(C)$, $\bar{\varphi} = \Pi'\circ\varphi:\Pi_\lambda(C) \to \bar{C}$, and $\bar{D} = \Pi'(D)$, and write $\bar{D} = \bar{U}\times\{\bar{u}\}$ for some $\bar{E}\subseteq\{1,\ldots,\bar{m}\}$ and open rational box $\bar{U}\subseteq\RR^{\bar{E}}$, and $\bar{u}\in\QQ^{\bar{E}^c}$, where $\bar{E}^c = \{1,\ldots,\bar{m}\}\setminus\bar{E}$.
\end{quote}

\item
polynomial maps $\Phi':\RR^{\bar{m}}\to\RR^m$ and $\Phi'':\RR^{\bar{n}}\to\RR^n$ which are sections of these projections;

\item
injections $\bar{\lambda}:\{1,\ldots,d\}\to\{1,\ldots,\bar{m}\}$ and $\bar{\lambda}':\{1,\ldots,\bar{m}-d\}\to\{1,\ldots,\bar{m}\}$ such that $\Pi_\lambda = \Pi_{\bar{\lambda}}\circ\Pi'$, $\im(\bar{\lambda}) \cap \im(\bar{\lambda}') = \emptyset$, and $\im(\bar{\lambda}) \cup \im(\bar{\lambda}') = \bar{E}$;

\item
a name $(\bar{p}(\bar{x},\bar{y}), \bar{\sigma},\bar{\alpha},\bar{\xi},\name(\bar{C})) \in \Delta^{d}_{\bar{m},\bar{n}}(\rho)$ for an $\S$-polynomial map $\bar{P} = \bar{p}\circ\bar{F}:\bar{C}\to\RR^{\bar{m}-d}$
\end{itemize}
such that $\bar{P}\circ\bar{\varphi}(x_\lambda) = 0$ on $\Pi_{\lambda}(C)$ and
\begin{equation}\label{eq:detBarPu}
\left.
\det\PD{}{(\bar{P}(\bar{x}_{\bar{E}},\bar{u}), \bar{x}_{\bar{E}^c} - \bar{u})}{(\bar{x}_{\bar{\lambda}'}, \bar{x}_{\bar{E}^c})}
\right|_{\bar{x}=\bar{\varphi}(a)} \neq 0,
\end{equation}
the map $\bar{\varphi}$ satisfies the $(\bar{\sigma},\bar{\alpha},\bar{\xi})$-distinctness condition, and the following is a commutative diagram of $\C$-analytic isomorphisms of germs of $\C$-analytic manifolds:
\begin{equation}\label{eq:distCondDiagram}
\xymatrix@C+25pt{
    & \im(\varphi_a)
        \ar@<0.5ex>[ld]^-{\Pi_{\lambda}}
        \ar@<0.5ex>[r]^-{F}
        \ar@<0.5ex>[dd]^-{\Pi'}
    & \im(F\circ\varphi_a)
        \ar@<0.5ex>[l]^-{\Pi_m}
        \ar@<0.5ex>[dd]^-{\Pi'\times\Pi''}
\\
\RR^{d}_{a}
    \ar@<0.5ex>[ru]^-{\varphi}
    \ar@<0.5ex>[rd]^-{\bar{\varphi}}
\\
    & \im(\bar{\varphi}_a)
        \ar@<0.5ex>[lu]^-{\Pi_{\bar{\lambda}}}
        \ar@<0.5ex>[r]^-{\bar{F}}
        \ar@<0.5ex>[uu]^-{\Phi'}
    & \im(\bar{F}\circ\bar{\varphi}_a)
        \ar@<0.5ex>[l]^-{\Pi_{\bar{m}}}
        \ar@<0.5ex>[uu]^-{\Phi'\times\Phi''} \quad.
}
\end{equation}
\end{lemma}

\begin{proof}
By Remark \ref{rmk:equivRelations}.5 we may fix an increasing map $\mu:\{1,\ldots,\bar{m}\} \to \{1,\ldots,m\}$ such that $\im(\lambda)\subseteq \im(\mu)$ and such that $\im(\mu)$ is a set of representatives for the $\sim_{(\sigma,\alpha,\xi,\varphi)}$-equivalence classes on $\{1,\ldots,m\}$.  Write $\Pi':\RR^m\to\RR^{\bar{m}}$ for the projection $\Pi_\mu:\RR^m\to\RR^{\bar{m}}$.  Write $M_1,\ldots,M_{\bar{m}}$ for the $\sim_{(\sigma,\alpha,\xi,\varphi)}$-equivalence classes on $\{1,\ldots,m\}$, where $\mu(i)\in M_i$ for all $i\in\{1,\ldots,\bar{m}\}$.  Let $\Phi':\RR^{\bar{m}}\to\RR^m$ be the section of the projection $\Pi'$ defined by
\[
\Phi'(x_1,\ldots,x_{\bar{m}})
=
\left((x_1)_{i\in M_1}, \ldots, (x_{\bar{m}})_{i\in M_{\bar{m}}}\right).
\]

Similarly, fix an increasing map $\nu:\{1,\ldots,\bar{n}\} \to \{1,\ldots,n\}$ such that $\im(\nu)$ is a set of representatives for the $\approx_{(\sigma,\alpha,\xi,\varphi)}$-equivalence classes on $\{1,\ldots,n\}$.  Write $\Pi'':\RR^n\to\RR^{\bar{n}}$ for the projection $\Pi_\nu:\RR^n\to\RR^{\bar{n}}$.  Write $N_1,\ldots,N_{\bar{n}}$ for the $\approx_{(\sigma,\alpha,\xi,\varphi)}$-equivalence classes on $\{1,\ldots,n\}$, where $\nu(j)\in N_j$ for all $j\in\{1,\ldots,\bar{n}\}$.  Let $\Phi'':\RR^{\bar{n}}\to\RR^n$ be the section of the projection $\Pi''$ defined by
\[
\Phi''(y_1,\ldots,y_{\bar{n}})
=
\left((y_1)_{j\in N_1}, \ldots, (y_{\bar{n}})_{j\in N_{\bar{n}}}\right).
\]

As in the statement of the lemma, write $\bar{x} = \Pi'(x)$, $\bar{y} = \Pi''(y)$, $\bar{C} = \Pi'(C)$, $\bar{\varphi} = \Pi'\circ\varphi:\Pi_\lambda(C) \to C$, and $\bar{D} = \Pi'(D)$.  Write $\bar{E} = \mu^{-1}(E)$, fix an injection
$\bar{\lambda}':\{1,\ldots,\dim(\bar{D})-d\}\to\{1,\ldots,\bar{m}\}$ such that $\im(\bar{\lambda}) \cap \im(\bar{\lambda}') = \emptyset$ and $\im(\bar{\lambda}) \cup \im(\bar{\lambda}') = \bar{E}$, and write $\bar{D} = \bar{U}\times \bar{u}$ for an open rational box $\bar{U}\subseteq\RR^{\bar{E}}$ and $\bar{u}\in\QQ^{\bar{E}^c}$.

Put $\Pi = \Pi'\times\Pi'':\RR^{m+n}\to\RR^{\bar{m}+\bar{n}}$ and $\Phi = \Phi'\times\Phi'':\RR^{\bar{m}+\bar{n}}\to\RR^{m+n}$.
Define
\[
\begin{array}{rcll}
\tld{p}(\bar{x},\bar{y})
    & =
    & p\circ\Phi(\bar{x},\bar{y})
    & \quad\text{on $\RR^{\bar{m}+\bar{n}}$,}
\vspace*{3pt}\\
\bar{f}(\bar{x})
    & =
    & \Pi''\circ f\circ\Phi'(\bar{x})
    & \quad\text{on $\bar{C}$,}
\vspace*{3pt}\\
\bar{F}(\bar{x})
    & =
    & (\bar{x},\bar{f}(\bar{x}))
    & \quad\text{on $\bar{C}$,}
\vspace*{3pt}\\
\tld{P}(\bar{x})
    & =
    & \tld{p}\circ\bar{F}(\bar{x})
    & \quad\text{on $\bar{C}$,}
\vspace*{3pt}\\
\bar{\varphi}(x_\lambda)
    & =
    & \Pi'\circ\varphi(x_\lambda)
    & \quad\text{on $\Pi_\lambda(C) = \Pi_{\bar{\lambda}}(\bar{C})$,}
\vspace*{3pt}\\
\bar{\sigma}
    & =
    & \sigma\circ\nu
    & \quad\text{on $\{1,\ldots,\bar{n}\}$,}
\vspace*{3pt}\\
\bar{\alpha}
    & =
    & \alpha\circ\nu
    & \quad\text{on $\{1,\ldots,\bar{n}\}$,}
\vspace*{3pt}\\
\bar{\xi}
    & =
    & \xi\circ\nu
    & \quad\text{on $\{1,\ldots,\bar{n}\}$.}
\end{array}
\]

It follows from these definitions that $\bar{\varphi}:\Pi_{\bar{\lambda}}(\bar{C})\to\bar{C}$ is a section of the projection $\Pi_{\bar{\lambda}}$, that $\im(\bar{\varphi}) = \{\bar{x}\in\bar{C} : \tld{P}(\bar{x}) = 0\}$, that $\bar{\varphi}$ satisfies the $(\bar{\sigma},\bar{\alpha},\bar{\xi})$-distinctness condition, and that \eqref{eq:distCondDiagram} is a commutative diagram of $\C$-analytic isomorphisms of germs of $\C$-analytic manifolds.

We are almost done.  The only problem is that since $\im(\bar{\varphi})$ is a $d$-dimensional submanifold of $\RR^{\bar{m}}$, and the system $\tld{P}(\bar{x}) = 0$ has $|E|-d$ many equations, we have too many equations to define $\bar{\varphi}$ in a nonsingular manner, according to the implicit function theorem.  This will be remedied by suitably choosing $|\bar{E}|-d$ of the $|E|-d$ equations.

Write $b = \varphi(a)$.  Using Remark \ref{rmk:equivRelations}.3, we can write
\[
\PD{}{P}{x_{\lambda'}}(b)
=
\left(
\PD{}{p}{x_i}\circ F(b) + \PD{}{p}{y_{\gamma(i)}}\circ F(b) \PD{}{f_{\gamma(i)}}{x_i}(b)\right)_{i\in\im(\lambda')},
\]
which is a nonsingular $(|E|-d)\times(|E|-d)$ matrix written as a row of column vectors.  Abbreviating $\sim_{(\sigma,\alpha,\xi,\varphi)}$ as $\sim$, it follows that
\[
\PD{}{\tld{P}}{x_{\bar{\lambda}'}}(b)
=
\left(
\sum_{j \sim i} \PD{}{p}{x_j}\circ F(b) + \PD{}{p}{y_{\gamma(j)}}\circ F(b) \PD{}{f_{\gamma(j)}}{x_j}(b)\right)_{i\in\im(\bar{\lambda}')},
\]
is an $(|E|-d)\times(|\bar{E}|-d)$ matrix of rank $|\bar{E}|-d$, and hence it contains a nonsingular $(|\bar{E}|-d)\times(|\bar{E}|-d)$ submatrix whose rows are indexed by $\im(\delta)$ for some increasing map $\delta:\{1,\ldots,|\bar{E}|-d\}\to\{1,\ldots,|E|-d\}$.  Put
\begin{eqnarray*}
\bar{p}(\bar{x},\bar{y})
    & = &
    \Pi_\delta\circ\tld{p}(\bar{x},\bar{y}),
\\
\bar{P}(\bar{x})
    & = &
    \Pi_\delta\circ\tld{P}(\bar{x}),
\end{eqnarray*}
and note that $(\bar{p}(\bar{x},\bar{y}),\bar{\sigma},\bar{\alpha},\bar{\xi},\name(\bar{C})) \in \Delta^{d}_{\bar{m},\bar{n}}(\rho)$ is a name for $\bar{P}$, $\det\PD{}{\bar{P}}{x_{\bar{\lambda}'}}\circ\bar{\varphi}(a) \neq 0$, and the germ $\bar{\varphi}_a:\RR^{d}_{a}\to\RR^{\bar{m}-d}_{\bar{\varphi}(a)}$ is implicitly define by $\bar{P}\circ\bar{\varphi}(x_\lambda) = 0$ on $\bar{C}_{\bar{\varphi}(a)}$, which gives \eqref{eq:detBarPu}.
\end{proof}

\begin{proof}[Proof of Proposition \ref{prop:distCond}]
Pick $a\in\QQ^d\cap\Pi_\lambda(C)$.  By Lemma \ref{lemma:distCond}, $\Zar(\im(F\circ\varphi_a))$ and $\Zar(\im(\bar{F}\circ\bar{\varphi}_a))$ are isomorphic varieties, which together with Lemma \ref{lemma:tdBounds} gives
\[
\td_{\QQ} \QQ(\im(F\circ\varphi_a))
=
\td_{\QQ} \QQ(\bar{F}\circ\bar{\varphi}_a))
\leq
\bar{n} + d.
\]
The proposition now follows from the observation that $\varphi$ has the $(\sigma,\alpha,\xi)$-distinctness condition if and only if $\bar{n} = n$, by Remark \ref{rmk:equivRelations}.1.
\end{proof}

\begin{lemma}\label{lemma:dimReduction}
Consider the setup given in \eqref{eq:IFsetup}.
If $\varphi$ satisfies the $(\sigma,\alpha,\xi)$-distinctness condition, then there exists a dense open subset $V$ of $\Pi_\lambda(C)$ such that for every $a\in V$, $\varphi(a)$ also satisfies the $(\sigma,\alpha,\xi)$-distinctness condition.
\end{lemma}

\begin{proof}
Let $E_1,\ldots,E_k$ be the $\approx_{(\sigma,\alpha)}$-equivalence classes on $\{1,\ldots,n\}$.  Then the set
\[
V = \bigcap_{i=1}^{k} \bigcap_{j_1,j_2\in E_i \atop j_1\neq j_2}
\left\{x_\lambda\in \Pi_\lambda(C) : \Pi_{\xi(j_1)}\circ\varphi(x_\lambda) \neq \Pi_{\xi(j_2)}\circ\varphi(x_\lambda)\right\}
\]
is as desired.
\end{proof}

\begin{definition}\label{def:generic}
We say that $\S$ is {\bf generic} if for every $m\in\NN_+$ and $n\in\NN$, every name
\[
(p(x,y),\sigma,\alpha,\xi,\name(D))\in\Delta_{m,n}^{0}(\rho)
\]
for an $\S$-polynomial map $P = p\circ F:D\to\RR^{\dim(D)}$, and every nonsingular zero $a$ of $P$, if $a$ satisfies the $(\sigma,\alpha,\xi)$-distinctness condition, then
\[
\td_\QQ \QQ(F(a)) = n.
\]
\end{definition}

\begin{lemma}\label{lemma:genSection}
The following are equivalent.
\begin{enumerate}{\setlength{\itemsep}{3pt}
\item
The family $\S$ is generic.

\item
Consider the setup given in \eqref{eq:IFsetup}.  If $\varphi$ satisfies the $(\sigma,\alpha,\xi)$-distinctness condition, then
\[
\td_\QQ \QQ(\im(F\circ\varphi)) = n + d.
\]
}\end{enumerate}
\end{lemma}

\begin{proof}
Clearly 2 implies 1, since the definition of generic is statement 2 with $d = 0$.  Conversely, assume that $\S$ is generic, and consider the setup given in \eqref{eq:IFsetup}.  By Lemmas \ref{lemma:tdBounds} and \ref{lemma:dimReduction}, there exists a dense open subset $V$ of $\Pi_\lambda(C)$ such that for every $a\in\QQ^d\cap V$,
\[
\td_\QQ \QQ(\im(F\circ\varphi)) \geq \td_\QQ \QQ(F\circ\varphi(a)) + d,
\]
and $\varphi(a)$ satisfies the $(\sigma,\alpha,\xi)$-distinctness condition.  So fix $a\in \QQ^d\cap V$, and note that $\varphi(a)$ is a nonsingular zero in $D\times\RR^n$ of the map $x\mapsto (x_\lambda - a, P(x))$, which is in $\Delta_{n,m}^{0}(\S)$.  Thus $\td_\QQ \QQ(F\circ\varphi(a)) = n$, so $\td_\QQ \QQ(\im(F\circ\varphi)) = n + d$.
\end{proof}

\begin{remarks}\label{rmk:genTrans}
Suppose that $\S=\{S_\sigma\}_{\sigma\in\Sigma}$ is generic.
\begin{enumerate}{\setlength{\itemsep}{5pt}
\item
Each function $S_\sigma$ is transcendental:
\begin{quote}
This means that there does not exist a nonzero $q(x,y)\in\QQ[x,y]$ such that $q(x,S_\sigma(x)) = 0$.
\end{quote}

\begin{proof}
Fix $\sigma\in\Sigma$.  Write $\eta(\sigma) = m$ and $\rho(\sigma) = r$; we shall use variables $x = (x_1,\ldots,x_m)$, $y$ and $z$.  Define $P:[-r,r]\times\RR\to\RR$ by $P(x,y) = y - S_\sigma(x)$.  Then $P\in\Delta_{m+1,1}^{m}(\S)$, since $P(x,y) = p\circ F(x,y)$ with $F(x,y) = (x,y,S_\sigma(x))$ and $p(x,y,z) = y-z$.  Since $P(x,S_\sigma(x)) = 0$ and $\PD{}{P}{y} = 1$, Lemma \ref{lemma:genSection} gives
\begin{eqnarray*}
m+1
    & = &
    \td_{\QQ}\QQ(x,S_\sigma(x))
    \\
    & = &
    \td_\QQ \QQ(x) + \td_{\QQ(x)}\QQ(x,S_\sigma(x))
    \\
    & = &
    m + \td_{\QQ(x)}\QQ(x,S_\sigma(x)),
\end{eqnarray*}
so $\td_{\QQ(x)}\QQ(x,S_\sigma(x)) = 1$, which means that $S_\sigma$ is transcendental.
\end{proof}

\item
More generally, suppose that $\eta(\sigma) = m$ for all $\sigma\in\Sigma$.  Then $\S$ is differentially algebraically independent over $\QQ(x)$:
\begin{quote}
This means that for all lists of distinct pairs $(\sigma(1),\alpha(1)),\ldots,(\sigma(n),\alpha(n))$ in $\Sigma\times\NN^m$, there does not exist a nonzero $q(x,y)\in\QQ[x,y]$ such that $q\left(x, \PDn{\alpha(1)}{S_{\sigma(1)}}{x}(x), \ldots, \PDn{\alpha(n)}{S_{\sigma(n)}}{x}(x)\right) = 0$.
\end{quote}

\begin{proof}
Fix maps $\sigma:\{1,\ldots,n\}\to\Sigma$ and $\alpha:\{1,\ldots,n\}\to\NN^m$ such that
\[
(\sigma(1),\alpha(1)),\ldots,(\sigma(n),\alpha(n))
\]
are distinct pairs.  Write $x,y_2,\ldots,y_n$ for $m$-tuples of variables, write $z_1,\ldots,z_n$ for single variables, and write $y = (y_2,\ldots,y_n)$ and $z = (z_1,\ldots,z_n)$. Define $P:[-r,r]^n\times\RR^n \to \RR^{m(n-1)+n}$ by
\begin{eqnarray*}
P(x,y,z)
    & = &
    \left(y_2-x, \ldots, y_n-x, z_1 - \PDn{\alpha(1)}{S_{\sigma(1)}}{x}(x),\right.
    \\
    & &
    \left.
    \quad z_2 - \PDn{\alpha(2)}{S_{\sigma(2)}}{x}(y_2),
    \ldots,
    z_n - \PDn{\alpha(n)}{S_{\sigma(n)}}{x}(y_n)
    \right).
\end{eqnarray*}
Then $P\in\Delta_{mn+n,n}^{m}(\S)$.  Note that $P(x,\ldots,x,\PDn{\alpha(1)}{S_{\sigma(1)}}{x}(x), \ldots, \PDn{\alpha(n)}{S_{\sigma(n)}}{x}(x)) = 0$, and $\det\PD{}{P}{(y,z)} = 1$, so  Lemma \ref{lemma:genSection} gives
\[
m+n = \td_{\QQ} \QQ\left(x,\ldots,x,\PDn{\alpha(1)}{S_{\sigma(1)}}{x}(x), \ldots, \PDn{\alpha(n)}{S_{\sigma(n)}}{x}(x)\right),
\]
so
\[
n = \td_{\QQ(x)} \QQ\left(x,\PDn{\alpha(1)}{S_{\sigma(1)}}{x}(x), \ldots, \PDn{\alpha(n)}{S_{\sigma(n)}}{x}(x)\right).
\]
Thus $\PDn{\alpha(1)}{S_{\sigma(1)}}{x}, \ldots, \PDn{\alpha(n)}{S_{\sigma(n)}}{x}$ are algebraically independent over $\QQ(x)$.
\end{proof}
}\end{enumerate}
\end{remarks}

We are now ready to prove Theorem \ref{introThm:genDec}, restated below.

\begin{theorem}\label{thm:genDec}
If $\S$ is generic, then the theory of $\RR_{\S}$ is decidable if and only if the approximation oracle for $\S$ is decidable.
\end{theorem}

\begin{proof}
The theory of $\RR_{\S}$ decides the approximation oracle for $\S$.  So assume that $\S$ is generic and has a decidable approximation oracle.  We must prove that the theory of $\RR_{\S}$ is decidable.  To do this, it suffices by Lemma \ref{lemma:idealMemOracle} to show that the ideal membership problem for $\S$ is decidable.

Consider the setup given in \eqref{eq:IFsetup}.  Our goal is to decide the membership problem for the ideal $\II(\im(F\circ\varphi))$.  Fix $a\in\QQ^d\cap\Pi_\lambda(C)$.  Since $\II(\im(F\circ\varphi)) = \II(\im(F\circ\varphi_a))$, it suffices to study the germ $\im(F\circ\varphi_a)$.  If we knew the equivalence relation $\approx_{(\sigma,\alpha,\xi,\varphi)}$ on $\{1,\ldots,n\}$, then we could solve the membership problem for the ideal $\II(\im(F\circ\varphi_a))$ as follows:  Use Lemma \ref{lemma:distCond} to construct $\bar{\varphi}$ which satisfies the $(\sigma,\alpha,\xi)$-distinctness condition.  Then construct a set of generators for $\II(\bar{X})$, where $\bar{X}$ is the unique irreducible component of $\VV(\bar{p}(\bar{x},\bar{y}), \bar{x}_{\bar{E}^c} - \bar{u})$ containing $a$ (and thus containing $\im(\bar{F}\circ\bar{\varphi}_a)$).  Since $\S$ is generic, $\td_{\QQ}\QQ(\im(\bar{F}\circ\bar{\varphi}_a)) = \bar{n} + d$ by Lemma \ref{lemma:genSection}, so $\II(\im(\bar{F}\circ\bar{\varphi}_a)) = \II(X)$.  Since $\Pi:\im(F\circ\varphi_a) \to \im(\bar{F}\circ\bar{\varphi}_a)$ and $\Phi: \im(\bar{F}\circ\bar{\varphi}_a) \to \im(F\circ\varphi_a)$ are inverse germs of polynomial maps, we have that
\[
\II(\im(F\circ\varphi_a))
=
\{q\in\QQ[x,y] = q\circ\Phi \in \II(\im(\bar{F}\circ\bar{\varphi}_a))\}
=
\{q\in\QQ[x,y] = q\circ\Phi \in \II(X)\},
\]
and we are done since we have a set of generators for the ideal $\II(X)$.

Unfortunately, we do not know the equivalence relation $\approx_{(\sigma,\alpha,\xi,\varphi)}$ a priori, since it depends on knowing various equality relations among the components of the function $\varphi$.  But, there are two things we do know about  $\approx_{(\sigma,\alpha,\xi,\varphi)}$ from the onset.  First, we know that $\approx_{(\sigma,\alpha,\xi,\varphi)}$ refines the equivalence relation $\approx_{(\sigma,\alpha)}$.  Second, we know by Remark \ref{rmk:equivRelations}.5 that $\lambda(1),\ldots,\lambda(d)$ are in distinct $\sim_{(\sigma,\alpha,\xi,\varphi)}$-equivalence classes.  We therefore let $\E$ be the set of all equivalence relations $\approx$ on $\{1,\ldots,n\}$ which refine $\approx_{(\sigma,\alpha)}$ and are such that if we write $\sim$ for the equivalence relation on $\{1,\ldots,m\}$ induced by $\approx$, then $\lambda(1),\ldots,\lambda(d)$ are in distinct $\sim$-equivalence classes.  We will search $\E$ to find $\approx_{(\sigma,\alpha,\xi,\varphi)}$.

For each equivalence relation $\approx$ in $\E$, with induced equivalence relation $\sim$, let $\bar{m}$ and $\bar{n}$ be the number of equivalence classes of $\sim$ and $\approx$, respectively, and fix increasing maps
$\mu:\{1,\ldots,\bar{m}\}\to\{1,\ldots,m\}$ and $\nu:\{1,\ldots,\bar{n}\}\to\{1,\ldots,n\}$ whose images are sets of representatives for the equivalence relations $\sim$ and $\approx$, respectively, and such that
$\im(\lambda)\subseteq\im(\mu)$.  As we are describing for $\sim$ and $\approx$ what was already done for
$\sim_{(\sigma,\alpha,\xi,\varphi)}$ and $\approx_{(\sigma,\alpha,\xi,\varphi)}$ in the proof of Lemma \ref{lemma:distCond}, let us agree to use the following list of notation from the proof of the lemma without redefinition, but now apply it to $\sim$ and $\approx$ instead of $\sim_{(\sigma,\alpha,\xi,\varphi)}$ and $\approx_{(\sigma,\alpha,\xi,\varphi)}$:
\begin{itemize}
\item
$\Pi'(x) = \bar{x}$, $\Pi''(y) = \bar{y}$, and $\Pi(x,y) = (\bar{x},\bar{y})$,

\item
$\Phi'$, $\Phi''$, and $\Phi = \Phi'\times\Phi''$

\item
$\bar{\varphi} = \Pi'\circ\varphi:\Pi_\lambda(C)\to\bar{C}$, with $\bar{C} = \Pi'(C)$,

\item
$\bar{E} = \mu^{-1}(E)$ and $\bar{E}^c = \{1,\ldots,\bar{m}\}\setminus\bar{E}$,

\item
$\bar{D} = \Pi'(D)$, and write $\bar{D} = \bar{U}\times\bar{u}$ for an open rational box $\bar{U}\subseteq\RR^{\bar{E}}$ and $\bar{u}\in\QQ^{\bar{E}^c}$,

\item
$\bar{\lambda}:\{1,\ldots,d\}\to\{1,\ldots,\bar{m}\}$ and $\bar{\lambda}':\{1,\ldots,\dim(\bar{D})-d\}\to\{1,\ldots,\bar{m}\}$,

\item
$\bar{\sigma}$, $\bar{\alpha}$, and $\bar{\xi}$,

\item
$\bar{F}(\bar{x}) = (\bar{x},\bar{f}(\bar{x})) = \Pi\circ F\circ \Phi'(\bar{x})$,

\item
$\tld{p}(\bar{x},\bar{y})$ and $\tld{P}(\bar{x})$.
\end{itemize}
Note that with this notation, the equivalence relations $\approx$ and $\approx_{(\sigma,\alpha,\xi,\varphi)}$ are the same if and only if $\varphi = \Phi'\circ\bar{\varphi}$.

Note that $\Pi_{E^c}\circ\varphi(x_\lambda) = u$ and $\Pi_{\bar{E}^c}\circ\bar{\varphi}(x_\lambda) = \bar{u}$ for all $x_\lambda\in\Pi_\lambda(C)$.  (Also, it follows from the definition of $\sim$ that $|E^c| = |\bar{E}^c|$ and that $\bar{u}$ is obtained from $u$ by permuting its coordinates.)

The following two subroutines will be used to search for $\approx_{(\sigma,\alpha,\xi,\varphi)}$ among all the members of $\E$.\\

\noindent{\bf Subroutine 1.}
{\slshape Use $C^\infty$ approximation algorithms for $\varphi$ and $\Phi'\circ\bar{\varphi}$ to try to verify that $\varphi\neq\Phi'\circ\bar{\varphi}$.  This is done searching for a compact rational box $B\subseteq\Pi_\lambda(C)$ and disjoint, open, rational boxes $A, A' \subseteq \RR^{m+n}$ such that $\varphi(B) \subseteq A$ and $\Phi'\circ\bar{\varphi}(B) \subseteq A'$.}\\

Note that Subroutine 1 terminates if and only if $\approx$ does not equal $\approx_{(\sigma,\alpha,\xi,\varphi)}$.\\

\noindent{\bf Subroutine 2.}

\begin{description}
\item[Step 1]
{\slshape Use an approximation algorithm for $\Phi'\circ\bar{\varphi}(a)$ to try to verify that $\Phi'\circ\bar{\varphi}(a)\in C$.  (Note: Since $\Pi_{E^c}\circ\Phi'\circ\bar{\varphi}(a) = u$, this is equivalent to saying that $\Pi_E\circ\Phi'\circ\bar{\varphi}(a)$ is in $\Pi_E(C)$, which is an open rational box.)}
\end{description}
If $\Phi'\circ\bar{\varphi}(a)\notin C$, Step 1 does not terminate.  But, this can only occur when $\approx$ does not equal $\approx_{(\sigma,\alpha,\xi,\varphi)}$.

\begin{description}
\item[Step 2]
{\slshape Use an approximation algorithm for $\PD{}{\tld{P}}{\bar{x}}$ to search for an increasing map $\delta:\{1,\ldots,\dim(\bar{D}) - d\}\to\{1,\ldots,\dim(D)-d\}$ such that if we put $\bar{p} = \Pi_\delta\circ\tld{p}$ and $\bar{P} = \bar{p}\circ \bar{F}$, then $\det\PD{}{\bar{P}}{x_{\bar{\lambda}'}}\circ\bar{\varphi}(a) \neq 0$.}
\end{description}
If there is no such map $\delta$, Step 2 does not terminate.  But again, this can only occur when $\approx$ does not equal $\approx_{(\sigma,\alpha,\xi,\varphi)}$.

\begin{description}
\item[Step 3]
{\slshape Find sets of generators for the isolated primes of the ideal $\lb \bar{p}(\bar{x},\bar{y}), \bar{x}_{\bar{E}^c} - \bar{u}\rb$, as described in the discussion following Lemma \ref{lemma:idealMemOracle}.}
\end{description}
Call these primes $\mathfrak{p}_1,\ldots,\mathfrak{p}_k$.

\begin{description}
\item[Step 4]
{\slshape Search for all $j\in\{1,\ldots,k\}$ such that $\bar{F}\circ\bar{\varphi}(a)\notin\VV(\mathfrak{p}_j)$.  Continue searching until $(k-1)$-many such $j$'s have been found.  Let $i$ be the sole remaining member of $\{1,\ldots,k\}$, and let $X = \VV(\mathfrak{p}_i)$.}
\end{description}

If there is more than one $j$ such that $\bar{F}\circ\bar{\varphi}(a)\in\VV(\mathfrak{p}_j)$, Step 4 does not terminate.  But again, this can only occur when $\approx$ does not equal $\approx_{(\sigma,\alpha,\xi,\varphi)}$.

\begin{description}
\item[Step 5]
{\slshape Let
\[
I = \{q(x,y)\in\QQ[x,y] : q\circ\Phi(\bar{x},\bar{y})\in\II(X)\}.
\]
Check if $p_1(x,y),\ldots,p_{\dim(D)-d}(x,y)\in I$, where $p = (p_1,\ldots,p_{\dim(D)-d})$.  If this is the case, return the ideal $I$ and stop.  If this is not the case, just stop without returning anything.}
\end{description}
``Returning the ideal $I$'' means that we are returning the membership algorithm for $I$ which is induced from the membership algorithm for $\II(X)$ in the obvious way: to determine if a polynomial $q(x,y)\in\QQ[x,y]$ is in $I$, we simply test if $q\circ\Phi(\bar{x},\bar{y})\in\II(X)$, which can be done since we have a set of generators for $\II(X)$.

We claim that Step 5 returns the ideal $I$ if and only if $\approx$ equals $\approx_{(\sigma,\alpha,\xi,\varphi)}$, and in this case $I = \II(\im(F\circ\varphi_a))$.  To see this, first note that $q(x,y)\in\QQ[x,y]$ vanishes on $\Phi(X)$ if and only if $q\circ\Phi(\bar{x},\bar{y})$ vanishes on $X$, so
\begin{equation}\label{eq:ID}
I = \II(\Phi(X)).
\end{equation}

Now, assume that $\approx$ equals $\approx_{(\sigma,\alpha,\xi,\varphi)}$.  Then $\bar{\varphi}$ satisfies the $(\bar{\sigma},\bar{\alpha},\bar{\xi})$-distinctness condition and is implicitly defined on $\bar{D}$ by the equation $\bar{P}\circ\bar{\varphi}(x_\lambda) = 0$.  Because $\S$ is generic, we have $\II(\im(\bar{F}\circ\bar{\varphi}_a)) = \II(X)$, and therefore
\begin{equation}\label{eq:Zar}
\II(\im(\Phi\circ \bar{F}\circ\bar{\varphi}_a)) = \II(\Phi(X)).
\end{equation}
Also, because $\approx$ equals $\approx_{(\sigma,\alpha,\xi,\varphi)}$,
\begin{equation}\label{eq:Fvarphi}
\Phi\circ \bar{F}\circ\bar{\varphi} = F\circ\varphi.
\end{equation}
It follows from \eqref{eq:ID}, \eqref{eq:Zar} and \eqref{eq:Fvarphi} that $I = \II(\im(F\circ\varphi_a))$.  Clearly, $p_1,\ldots,p_{\dim(D)-d}\in\II(\im(F\circ\varphi_a))$.  Therefore Step 5 returns the ideal $\II(\im(F\circ\varphi_a))$.

Conversely, assume that $p_1,\ldots,p_{\dim(D)-d}\in I$.  We need to show that $\approx$ equals $\approx_{(\sigma,\alpha,\xi,\varphi)}$.  Put $b = \bar{F}\circ\bar{\varphi}(a)$.  Then $\im(\bar{F}\circ\bar{\varphi}_a) \subseteq X_b$, so
\begin{equation}\label{eq:1}
\im(\Phi \circ \bar{F}\circ \bar{\varphi}_a) \subseteq \Phi(X_b).
\end{equation}
Because $\approx$ refines the equivalence relation $\approx_{(\sigma,\alpha)}$, it follows from the definitions of the functions involved that $\Phi''\circ \bar{f} = f\circ\Phi'$.  Therefore
\begin{eqnarray}\label{eq:2}
\Phi \circ \bar{F}(\bar{x})
    & = &
    (\Phi'(\bar{x}),\Phi''\circ \bar{f}(\bar{x})),
    \\
    & = &
    (\Phi'(\bar{x}), f\circ\Phi'(\bar{x})),\nonumber\\
    & = &
    F\circ\Phi'(\bar{x}).\nonumber
\end{eqnarray}
Equations \eqref{eq:1} and \eqref{eq:2} show that
\begin{equation}\label{eq:3}
\im(F\circ\Phi'\circ\bar{\varphi}_a) \subseteq \Phi(X_b).
\end{equation}
Using \eqref{eq:ID} along with our assumption that $p_1,\ldots,p_{\dim(D)-d}\in I$ implies that $p$ vanishes on $\Phi(X)$.  Therefore applying $p$ to both sides of \eqref{eq:3} shows that $P\circ\Phi'\circ\bar{\varphi}_a = 0$.  Now, $\Phi'\circ\bar{\varphi}$ is a section of the projection $\Pi_\lambda$ and $\Phi'\circ\bar{\varphi}(a)\in C$.   But $\varphi:\Pi_\lambda(C)\to C$ is the unique section of $\Pi_\lambda:C\to \Pi_\lambda(C)$ implicitly defined by $P$ on $C$, so necessarily $\Phi'\circ\bar{\varphi} = \varphi$.  Therefore $\approx$ equals $\approx_{(\sigma,\alpha,\xi,\varphi)}$.

We are now ready to give the algorithm which solves the membership problem for the ideal $\II(\im(F\circ\varphi_a))$.\\

\noindent{\bf The Algorithm}.
{\slshape For each equivalence relation $\approx$ in $\E$, use time sharing to run Subroutines 1 and 2 on $\approx$ simultaneously.  One of the subroutines will stop first, since at least one of them must stop.  If Subroutine 1 stops first, then $\approx$ does not equal $\approx_{(\sigma,\alpha,\xi,\varphi)}$, so move on to the next equivalence relation in $\E$.  Likewise, if Subroutine 2 stops first but does not return an ideal $I$, then $\approx$ does not equal $\approx_{(\sigma,\alpha,\xi,\varphi)}$, so move on to the next equivalence relation in $\E$.  If Subroutine 2 stops first and returns the ideal $I$, then $\approx$ equals $\approx_{(\sigma,\alpha,\xi,\varphi)}$ and $I = \II(\im(F\circ\varphi_a))$, so return the ideal $I$ and terminate the program.\\
}

This algorithm eventually finds $\approx_{(\sigma,\alpha,\xi,\varphi)}$ and returns $\II(\im(F\circ\varphi_a))$ since $\E$ is finite.
\end{proof}

%% file: genDense.tex
\section{Constructing generic, computably $C^\infty$ families of functions}
\label{s:genDense}

In this section we prove Theorem \ref{introThm:genDense}.  We begin with a remark which verifies that
\begin{equation}\label{eq:topBase}
\{\Ball_{\C}(\S,\epsilon) : \text{$\S\in\Comp_{\C}(\rho)$ and $\epsilon:\Delta(\rho)\to\QQ_+$ is computable}\}
\end{equation}
is indeed a valid base for a topology on $\Comp_{\S}(\rho)$.

\begin{remark}\label{rmk:topBase}
Let $\S\in\Comp_{\C}(\rho)$, $\epsilon:\Delta(\rho)\to\QQ_+$ be computable, and $\T\in\Ball(\S,\epsilon)$.  Then there exists a computable map $\delta:\Delta(\rho)\to\QQ_+$ such that $\Ball(\T,\delta) \subseteq \Ball(\S,\epsilon)$.
\end{remark}

\begin{proof}
For each $(\sigma,\alpha)\in\Delta(\S)$, we have
\[
M = \max \left\{\left|
\PDn{\alpha}{S_\sigma}{x}(x) - \PDn{\alpha}{T_\sigma}{x}(x)
\right| : x\in[-\rho(\sigma),\rho(\sigma)]\right\}
< \epsilon(\sigma,\alpha),
\]
so by Lemma \ref{lemma:compEVT} we can effectively find a positive rational number $\delta(\sigma,\alpha) < \epsilon(\sigma,\alpha) - M$.  Note that $\Ball(\T,\delta) \subseteq\Ball(\S,\epsilon)$.
\end{proof}

The following lemma is a reformulation of a classical theorem of Severi \cite{Severi} (see also Lorentz \cite[Chapter 12]{Lorentz}).  It is fundamental to our proof of Theorem \ref{introThm:genDense}.

\begin{lemma}\label{lemma:interpolationPoly}
Let $n$ and $m$ be positive integers, and for each
$i\in\{1,\ldots,n\}$ let $I_i$ be a finite subset of $\NN^n$.
Consider an $n$-tuple of variables $x = (x_1,\ldots,x_n)$ and an
$mn$-tuple of variables $y = (y_1,\ldots,y_m)$, where $y_i =
(y_{i,1},\ldots,y_{i,n})$ for each $i\in\{1,\ldots,m\}$.  We can
effectively construct families of polynomials $\{N_{i,\alpha}(x,y)\}_{i\in\{1,\ldots,m\}, \alpha\in I_i} \subseteq \QQ[x,y]$ and $\{D_{i,\alpha}(y)\}_{i\in\{1,\ldots,m\}, \alpha\in I_i} \subseteq \QQ[x]$ such that if we put
\[
p_{i,\alpha}(x,y) = \frac{N_{i,\alpha}(x,y)}{D_{i,\alpha}(y)}
\quad\text{for each $i\in\{1,\ldots,m\}$ and $\alpha\in I_i$,}
\]
then the following hold for all distinct points $a_1,\ldots,a_m$ in $\RR^n$:
\begin{enumerate}{\setlength{\itemsep}{3pt}
\item
Write $a = (a_1,\ldots,a_m)$.  Then $D_{i,\alpha}(a)\neq 0$ for all $i\in\{1,\ldots,m\}$ and $\alpha\in I_i$.

\item
For all $i,j\in\{1,\ldots,m\}$, $\alpha\in I_i$, and $\beta\in I_j$,
\[
\PDn{\beta}{p_{i,\alpha}}{x}(a_j,a) = \left\{\begin{array}{ll}
1, & \text{if $(j,\beta) = (i,\alpha)$},\\
0, & \text{if $(j,\beta) \neq (i,\alpha)$}.
\end{array}\right.
\]
}\end{enumerate}
\end{lemma}

\begin{proof}
Throughout the proof, $a = (a_1,\ldots,a_m)$ denotes an arbitrary
$m$-tuple of distinct points $a_1,\ldots,a_m$ in $\RR^n$, and $\leq$
denotes the partial ordering of $\NN^n$ given by
\[
\alpha\leq\beta\quad\text{if and only if}\quad \alpha_1\leq\beta_1,\ldots,\alpha_n\leq\beta_n
\]
for all $\alpha = (\alpha_1,\ldots,\alpha_n)$ and $\beta = (\beta_1,\ldots,\beta_n)$ in $\NN^n$.  For each $j\in\{1,\ldots,m\}$, let
\[
h_j = \max\{|\beta| : \beta\in I_j\},
\]
and for each $(i,j)\in\{1,\ldots,m\}^2$ with $i\neq j$, let
\[
L_{i,j}(x,y) = (x - y_j)\cdot(y_j - y_i) = \sum_{k=1}^{n}(x_k - y_{j,k})(y_{j,k} - y_{i,k}).
\]
Thus the zero set of $L_{i,j}(x,a)$ is the $(n-1)$-dimensional
affine subspace of $\RR^n$ through $a_j$ normal to the vector $a_j -
a_i$.  In particular, $L_{i,j}(a_j,a) = 0$ and $L_{i,j}(a_i,a)\neq 0$.

For each $i\in\{1,\ldots,m\}$ and $\alpha\in I_i$, define
\[
q_{i,\alpha}(x,y) = (x - y_i)^\alpha
\left(
\prod_{j\in\{1,\ldots,m\}\setminus\{i\}}
L_{i,j}(x,y)^{h_j + 1}
\right).
\]
Note that
\begin{equation}\label{eq:iNotZero}
\PDn{\alpha}{q_{i,\alpha}}{x}(a_i,a) \neq 0,
\end{equation}
that  for every $\beta\in\NN^n$ for which $\beta_k < \alpha_k$ for some $k\in\{1,\ldots,n\}$,
\begin{equation}\label{eq:iZero}
\PDn{\beta}{q_{i,\alpha}}{x}(a_i,a) = 0,
\end{equation}
and that for all $j\in\{1,\ldots,m\}\setminus\{i\}$ and all $\beta\in I_j$,
\begin{equation}\label{eq:jZero}
\PDn{\beta}{q_{i,\alpha}}{x}(a_j,a) = 0.
\end{equation}
Fix $i\in\{1,\ldots,m\}$.  We construct the $p_{i,\alpha}$'s from the $q_{i,\alpha}$'s by induction.

Consider $\alpha\in I_i$, define $I_i(\alpha) = \{\beta\in I_i :
\beta\geq\alpha, \beta\neq\alpha\}$, and inductively assume that
$p_{i,\beta}$ has been defined for every $\beta\in I_i(\alpha)$,
with the understanding that the base case of the induction is when
$I_i(\alpha)$ is empty.  Define
\begin{equation}\label{eq:p}
p_{i,\alpha}(x,y) = \frac{q_{i,\alpha}(x,y) - \sum_{\beta\in I_i(\alpha)}
\PDn{\beta}{q_{i,\alpha}}{x}(y_i,y)
p_{i,\beta}(x,y)}
{\PDn{\alpha}{q_{i,\alpha}}{x}(y_i,y)}.
\end{equation}
It follows from \eqref{eq:iNotZero}-\eqref{eq:p} that $p_{i,\alpha}$ has the desired properties.
\end{proof}

Write $\id_n$ for the $n\times n$ identity matrix, and for a
matrix $A = (a_{i,j})_{i,j}$, write $\|A\| =
\max_{i,j}|a_{i,j}|$.

\begin{lemma}\label{lemma:pertubationPoly}
Let $a_1,\ldots,a_m$ be (not necessarily distinct) computable points
in $\RR^n$, and let $\alpha:\{1,\ldots,m\}\to\NN^n$ be such
that for all distinct $i,j\in\{1,\ldots,m\}$,
\begin{equation}\label{eq:distinct}
\text{if $\alpha(i) = \alpha(j)$, then $a_i\neq a_j$.}
\end{equation}
Then given any approximation algorithms for the points
$a_1,\ldots,a_m$ and any positive rational number $\epsilon$, we can
effectively find polynomials $p_1(x),\ldots,p_m(x)\in\QQ[x]$ such
that
\begin{equation}\label{eq:detNotZero}
\left\|\left(\PDn{\alpha(i)}{p_j}{x}(a_i)\right)_{(i,j)\in\{1,\ldots,m\}^2} - \id_m\right\| < \epsilon.
\end{equation}
\end{lemma}

\begin{proof}
For each $i\in\{1,\ldots,m\}$, use the approximation algorithm for
$a_i$ to construct a computable decreasing sequence of compact
rational boxes
\begin{equation}\label{eq:boxes}
B_{i,0} \supseteq B_{i,1} \supseteq B_{i,2} \supseteq \cdots
\end{equation}
such that $\bigcap_{k=0}^{\infty}B_{i,k} = \{a_i\}$.  For each
$k\in\NN$, let $c(k)$ be the number of connected components of $B(k)
:= \bigcup_{i=1}^{m}B_{i,k}$, and fix a computable map
$C:\{1,\ldots,m\}\times\NN\to\NN$ such that for each $k\in\NN$,
$\{C(1,k),\ldots,C(m,k)\} = \{1,\ldots,c(k)\}$ and
\begin{eqnarray*}
B(1,k) & := & \bigcup\{B_{i,k} : 1\leq i\leq m, C(i,k) = 1\},\\
\vdots\\
B(c(k),k) & := & \bigcup\{B_{i,k} : 1\leq i\leq m, C(i,k) = c(k)\},
\end{eqnarray*}
are the connected components of $B(k)$.  We assume that if $c(k+1) =
c(k)$, then $C(i,k) = C(i,k+1)$ for all $i\in\{1,\ldots,m\}$.  By
using the sequences \eqref{eq:boxes}, we can effectively verify that
\eqref{eq:distinct} holds.  So by throwing away initial segments of
the sequences \eqref{eq:boxes}, we may assume that for all distinct
$i,j\in\{1,\ldots,m\}$, if $\bar{\alpha}(i) = \bar{\alpha}(j)$, then
$a_i$ and $a_j$ are contained in different connected components of
$B(0)$, and hence $C(i,0)\neq C(j,0)$.

For each $k\in\NN$ and $i\in\{1,\ldots,c(k)\}$, let $I(i,k) =
\{\alpha(j) : C(j,k) = i\}$, and note that $|I(i,k)| = |\{j :
C(j,k) = i\}|$, so
\[
\sum_{i=1}^{c(k)}|I(i,k)| = \sum_{i=1}^{c(k)}|\{j : C(j,k) = i\}| = m.
\]
It follows from Lemma \ref{lemma:interpolationPoly} that for each
$i\in\{1,\ldots,c(k)\}$ and $\beta\in I(i,k)$, we can fix
$p^{(k)}_{i,\beta}(x,y) \in \QQ(y)[x]$ such that for all $b =
(b_1,\ldots,b_{c(k)})\in B(1,k)\times\cdots\times B(c(k),k)$, all
$j\in\{1,\ldots,c(k)\}$, and all $\gamma\in I(j,k)$,
\[
\PDn{\gamma}{p^{(k)}_{i,\beta}}{x}(b_j,b) = \left\{\begin{array}{ll}
1, & \text{if $(\gamma,j) = (\beta,i)$},\\
0, & \text{if $(\gamma,j) \neq (\beta,i)$}.
\end{array}\right.
\]

The connected components of $B(k)$ converge to the distinct points
of $\{a_1,\ldots,a_m\}$ as $k\to\infty$, so there exists $K\in\NN$
such that for all $k,l\geq K$, the numbers $c(k)$ and  $c(l)$ are
the same, the maps $i\mapsto C(i,k)$ and $i\mapsto C(i,l)$ are the
same, and the polynomials $p^{(k)}_{i,\beta}$ and $p^{(l)}_{i,\beta}$
are the same.  We simply refer to this common number as $c$, this
common map as $i\mapsto C(i)$, and these common polynomials as
$p_{i,\beta}$.  Let $\tld{a} = (\tld{a}_1,\ldots,\tld{a}_c)$ be the $c$-tuple of points in $\RR^n$ consisting of the $c$ distinct members of $\{a_1,\ldots,a_m\}$
ordered so that $\tld{a}_i\in B(i,k)$ for all $i\in\{1,\ldots,c\}$ and
$k\geq K$. Then for all $i,j\in\{1,\ldots,m\}$,
\[
\PDn{\alpha(i)}{p_{C(j),\alpha(j)}}{x}(a_i,\tld{a}) = \left\{\begin{array}{ll}
1, & \text{if $i = j$},\\
0, & \text{if $i \neq j$}.
\end{array}\right.
\]
Fix a computable map $\chi$ on the set $\{(i,k)\in\NN^2 : 1\leq
i\leq c(k)\}$ such that $\chi(i,k)\in B(i,k)\cap\QQ^n$ for all
$i,k$. Put $\chi(k) = (\chi(1,k),\ldots,\chi(c(k),k))$.  Then
$\lim_{k\to\infty}\chi(k) = \tld{a}$, so
\[
\lim_{k\to\infty}
\left(
\PDn{\alpha(i)}{p^{(k)}_{C(j,k),\alpha(j)}}{x}(a_i,\chi(k))
\right)_{\!\!\!(i,j)\in\{1,\ldots,m\}^2}
\hspace*{-10pt}
=
\left(
\PDn{\alpha(i)}{p_{C(k),\alpha(j)}}{x}(a_i,\tld{a})
\right)_{\!\!\!(i,j)\in\{1,\ldots,m\}^2}
\hspace*{-10pt}
= \id_m.
\]
To find our desired polynomials $p_1(x),\ldots,p_m(x)$, simply find $k\in\NN$ such that
\[
\left\|\left(
\PDn{\alpha(i)}{p^{(k)}_{C(j,k),\alpha(j)}}{x}(a_i,\chi(k))
\right)_{(i,j)\in\{1,\ldots,m\}^2}
- \id_m\right\| < \epsilon,
\]
and let $p_j(x) = p^{(k)}_{C(j,k),\alpha(j)}(x,\chi(k))$ for each $j\in\{1,\ldots,m\}$.
\end{proof}

\begin{definition}\label{def:distCondBoxPt}
Let $\sigma$, $\alpha$, and $\xi$ be maps on $\{1,\ldots,n\}$ as described in clause 2 of Definition \ref{def:SpolyName}.  A compact rational box $B\subseteq\RR^m$ {\bf satisfies the $(\sigma,\alpha,\xi)$-distinctness condition} if for all distinct $i,j\in\{1,\ldots,n\}$ such that $i\approx_{(\sigma,\alpha)} j$, we have $\Pi_{\xi(i)}(B) \cap \Pi_{\xi(j)}(B) = \emptyset$.
A number $\epsilon > 0$ {\bf witnesses that $a\in\RR^m$ satisfies the $(\sigma,\alpha,\xi)$-distinctness condition} if for all distinct $i,j\in\{1,\ldots,n\}$ such that $i\approx_{(\sigma,\alpha)} j$, we have
$\|\Pi_{\xi(i)}(a) - \Pi_{\xi(j)}(a)\| > \epsilon$.
\end{definition}

\begin{notation}
The {\bf width} of a bounded box $B = \prod_{i=1}^{n}B_i$ in $\RR^n$ is defined by
\[
\width(B) = \max_{i\in\{1,\ldots,n\}}\length(B_i).
\]
For any nonempty open interval $I$ and $\delta > 0$, define $I_\delta$ according to the type of the interval $I$, as follows:
\[
I_\delta =
\begin{cases}
[-\frac{1}{\delta},\frac{1}{\delta}],
    & \text{if $I = \RR$,}
\\
[a+\delta, \frac{1}{\delta}],
    & \text{if $I = (a,+\infty)$,}
\\
[-\frac{1}{\delta},b-\delta],
    & \text{if $I = (-\infty,b)$,}
\\
[a+\delta, b-\delta],
    & \text{if $I = (a,b)$.}
\end{cases}
\]
If $D\subseteq\RR^m$ is a rational box manifold given by
\[
D = \left(\prod_{i\in E}U_i\right)\times\{u\},
\]
where $E\subseteq\{1,\ldots,m\}$, $u\in\QQ^{E^c}$, and each $U_i$ is an open rational interval in $\RR^{\{i\}}$, then for each $\delta > 0$, define
\[
D_\delta = \left(\prod_{i\in E}(U_i)_\delta\right)\times\{u\}.
\]
\end{notation}

\begin{lemma}\label{lemma:findNonsingZeros}
Given any $\delta\in\QQ_+$ and any name $(p(x,y),\sigma,\alpha,\xi,\name(D))\in \Delta_{m,n}^{0}(\rho)$ for an $\S$-polynomial map $P = p\circ F:D\to\RR^{\dim(D)}$, we can effectively find a finite family $\B$ of disjoint, bounded,  rational box manifolds, and a corresponding family $\{\lambda'_B\}_{B\in\B}$ of increasing functions $\lambda'_B:\{1,\ldots,\dim(D)\}\to\{1,\ldots,m+n\}$, with the following properties:
\begin{enumerate}{\setlength{\itemsep}{3pt}
\item
For each $B\in\B$, we have $\width(B) < \delta$ and $B = B'\times B''$, where $B''$ is open in $\RR^n$, $B'\subseteq\RR^m$ is open in $D$, $B'\cap D_{2\delta} \neq \emptyset$, and $B'$ satisfies the $(\sigma,\alpha,\xi)$-distinctness condition.

\item
Let $E\subseteq\{1,\ldots,m\}$ be such that $D = U\times\{u\}$ with $U$ open in $\RR^E$.  Then for each $B\in\B$, the statement $\IF_\emptyset(P;B')$ holds, $\im(\lambda'_B)\subseteq E\cup\{m+1,\ldots,m+n\}$, and $\det\PD{}{p}{(x,y)_{\lambda'_B}} \neq 0$ on $B$.

\item
For each $a\in D_{2\delta}$, if $P(a) = 0$, $|\det\PD{}{P}{x_E}(x)| > \delta$ for all $x\in D$ with $\|x-a\| \leq \delta$, and $2\delta$ witnesses that $a$ satisfies the $(\sigma,\alpha,\xi)$-distinctness condition, then there exists $B = B'\times B'' \in \B$ such that $a\in B'$.
}\end{enumerate}
\end{lemma}

When applying this lemma in the proof of Theorem \ref{introThm:genDense}, we say that a point $a$ is {\bf realized} by a box $B = B'\times B''\in\B$ if $a$ the unique zero of $P$ in $B'$.  We say that a point is {\bf realized by $\B$} if it is realized by some box in $\B$.  Also, we will write $\lambda_B:\{1,\ldots,n\}\to\{1,\ldots,m+n\}$ for the unique increasing function such that $\im(\lambda_B)\cup\im(\lambda'_B) = E\cup\{m+1,\ldots,m+n\}$. We call $\lambda_B$ the map {\bf complementary to $\lambda'_B$ relative to $D$}.  (The phrase ``relative to $D$'' is used because $D$ determines the set $E$.)

\begin{proof}
Begin computably enumerating all rational box manifolds $A\subseteq\RR^m$ which are open in $D$, $\width(A) < \delta$, $A\cap D_{2\delta}\neq\emptyset$, and either $|\det \PD{}{P}{x_E}| < \delta$ on $\cl(A)$ or $|\det \PD{}{P}{x_E}| > \frac{\delta}{2}$ on $\{x\in D : \min_{a\in\cl(A)}\|x-a\|\leq \delta\}$.  Stop enumerating once these boxes cover the compact rational box $D_{2\delta}$, which will occur after finite time.  Discard all boxes $A$ for which $|\det \PD{}{P}{x_E}| < \delta$ on $\cl(A)$, since we do not need to find any nonsingular zeros in these boxes. Let $K$ be the intersection of $D_{2\delta}$ with the union of the closures of all the remaining boxes $A$, for which $|\det \PD{}{P}{x_E}| > \frac{\delta}{2}$ holds on $\{x\in D : \min_{a\in\cl(A)}\|x-a\|\leq \delta\}$.  Start computably enumerating all rational box manifolds $C'$ which are open in $D$, $\width(C') < \delta$, $C'\cap K \neq\emptyset$, and either $\IF_{\emptyset}(P;C')$ holds or $P(C') \subseteq\RR^{\dim(D)}\setminus\{0\}$.  Stop enumerating once these boxes cover the compact set $K$, which will occur after finite time since every zero of $P$ in $K$ is a nonsingular zero.  Keep all the enumerated boxes $C'$ for which $\IF_\emptyset(P;C')$ holds and which satisfy the $(\sigma,\alpha,\xi)$-distinctness condition; discard all other boxes $C'$.  This collection of boxes $C'$ satisfies all of the properties in 1-3 which concern only the boxes $B'$ in the statement of the lemma, and not the boxes $B''$.

Fix one such box $C'$, and let $a\in C'$ be the unique zero of $P$ in $C'$.  Applying the computation following Lemma \ref{lemma:idealMemOracle} with $d=0$ (specifically, \eqref{eq:rankp}) shows that $\left.\rank \PD{}{p}{(x_E,y)}\right|_{(x,y) = F(a)} = \dim(D)$.  Since $F(a) = (a,f(a))$ is a computable point, we may find a bounded, open, rational box $B''\subseteq\RR^n$ containing $f(a)$ and an increasing map $\lambda':\{1,\ldots,\dim(D)\}\to\{1,\ldots,m+n\}$ such that $\im(\lambda')\subseteq E\cup\{m+1,\ldots,m+n\}$ and $\det \PD{}{p}{(x,y)_\lambda} \neq 0$ on $\{a\}\times \cl(B'')$.  By computably enumerating all rational box manifolds $B'$ which are open in $C'$ and contain $a$, we may find $B'$ such that $\IF_\emptyset(P;B')$ holds and $f(B')\subseteq B''$.  Put $B = B'\times B''$ and $\lambda'_B = \lambda'$.  The collection of all boxes $B$ and maps $\lambda'_B$, constructed as such, satisfy the conclusion of the lemma.
\end{proof}

\begin{notation}
For any families $\S=\{S_\sigma\}_{\sigma\in\Sigma}$ and $\T=\{T_\sigma\}_{\sigma\in\Sigma}$ of complex-valued functions with $\dom(S_\sigma) = \dom(T_\sigma)$ for all $\sigma\in\Sigma$, define $\S + \T = \{S_\sigma + T_\sigma\}_{\sigma\in\Sigma}$ and $\S - \T = \{S_\sigma - T_\sigma\}_{\sigma\in\Sigma}$.
\end{notation}

We are now ready to prove Theorem \ref{introThm:genDense}, restated below.

\begin{theorem}\label{thm:genDense}
If $\C$ contains the IF-system of all computably analytic functions, then $\Comp_{\C}(\rho)\cap\Gen_{\C}(\rho)$ is dense in $\Comp_{\C}(\rho)$.  In fact, there is an algorithm which acts as follows:
\begin{quote}
Given a $C^\infty$ approximation algorithm for $\S\in\Comp_{\C}(\rho)$ and a computable map $\epsilon:\Delta(\rho)\to\QQ_+$, the algorithm returns a $C^\infty$ approximation algorithm for some $\T\in \Gen(\rho) \cap \Ball_{\C}(\S,\epsilon)$.
\end{quote}
\end{theorem}

\begin{proof}
Fix $\S = \{S_\sigma\}_{\sigma\in\Sigma}$ in $\Comp_{\C}(\rho)$ and a computable map $\epsilon:\Delta(\rho)\to\QQ_+$.  The goal is to construct a generic $\T = \{T_\sigma\}_{\sigma\in\Sigma}$ in $\Ball_{\C}(\S,\epsilon)$.  Since we can always replace $\epsilon$ with the map $(\sigma,\alpha) \mapsto \epsilon(\sigma,\alpha)/2$, it suffices to construct a generic $\T$ in $\Comp_{\C}(\rho)$ such that
\begin{equation}\label{eq:boundT}
\left|\PDn{\alpha}{T_\sigma}{x} - \PDn{\alpha}{S_\sigma}{x}\right| \leq \epsilon(\sigma,\alpha)
\quad\text{on $[-\rho(\sigma),\rho(\sigma)]$}
\end{equation}
for all $(\sigma,\alpha)\in\Delta(\rho)$.

Fix a computable enumeration of $\bigcup_{(m,n)\in\NN_+\times\NN}\Delta_{m,n}^{0}(\rho)$,
\begin{equation}\label{eq:namesEnum}
\left\{(p_i(x,y), \sigma_i, \alpha_i, \xi_i, \name(D_i))\right\}_{i\in\NN_+},
\end{equation}
and a doubly indexed computable enumeration of
$\bigcup_{n\in\NN_+}(\QQ[x_1,\ldots,x_n]\backslash\{0\})$,
\[
\left\{q_{n,i}(x_1,\ldots,x_n)\right\}_{(n,i)\in\NN_{+}^{2}},
\]
such that $\QQ[x_1,\ldots,x_n]\backslash\{0\} =
\{q_{n,i}(x_1,\ldots,x_n) : i\in\NN_+\}$ for each $n\in\NN_+$, where $\NN_+ = \{1,2,3,\ldots\}$.  For each $i\in\NN_+$, let $m(i)$ and $n(i)$ be such that $(p_i(x,y), \sigma_i, \alpha_i, \xi_i,\name(D_i)) \in \Delta_{m(i),n(i)}^{0}(\rho)$.  Thus when we write $p_i(x,y)$, it is understood that $x$ denotes an $m(i)$-tuple of
variables and $y$ denotes an $n(i)$-tuple of variables.

Also fix a computable map $R:\Sigma\to\bigcup_{n\in\NN}\QQ^{n}_{+}$ with arity map $\eta$.  (The choice of $R$ is rather irrelevant, so $R$ could be rather simple.  For example, one could define $R(\sigma) = (1,\ldots,1)\in\QQ_{+}^{\eta(\sigma)}$ for each $\sigma\in\Sigma$.)

We will construct
\begin{enumerate}{\setlength{\itemsep}{3pt}
\item
a computable sequence $\{\S^{(k)}\}_{k\in\NN}$ in $\Comp_{\C}(\rho)$,

\item
a computable sequence $\{(\B^{(k)}_{1},\ldots,\B^{(k)}_{k})\}_{k\in\NN_+}$, where each $\B^{(k)}_{i}$ is a finite family of disjoint, bounded, rational box manifolds which are open in $D_i\times\RR^{n(i)}$,

\item
a computable sequence $\{(\{\lambda'_B\}_{B\in\B^{(k)}_{1}}, \ldots, \{\lambda'_B\}_{B\in\B^{(k)}_{k}})\}_{k\in\NN_+}$, where for each $k$, $i$, and $B\in\B^{(k)}_{i}$, $\lambda'_B:\{1,\ldots,n(i))\} \to \{1,\ldots,m(i)+n(i)\}$ is an increasing map,
}\end{enumerate}
such that Property($k$) holds for all $k\in\NN$.  For each $k\in\NN$ and $i\in\NN_+$, we write $P^{(k)}_{i} = p_i\circ F^{(k)}_{i}:D_i\to\RR^{\dim(D_i)}$ for the $\S^{(k)}$-polynomial map named by $(p_i(x,y), \sigma_i, \alpha_i, \xi_i, \name(D_i))$.\\

\noindent{\bf Property(0).}
{\slshape
We have $\S^{(0)} = \S$.
}\\

\noindent{\bf Property(k) for $\mathbf{k > 0}$.}
{\slshape
We have that $\S^{(k)} \in \Ball(\S,\epsilon)$, that $\S^{(k)} - \S^{(k-1)}$
is a family of polynomial functions with rational coefficients, that $S^{(k)}_{\sigma} = S^{(k-1)}_{\sigma}$ for all $\sigma\in\Sigma \setminus \bigcup_{i=1}^{k}\im(\sigma_i)$, and that
\begin{equation}\label{eq:forceCauchy}
\dist(\S^{(k)}, \S^{(k-1)})
:=
\sup\left\{
|S^{(k)}_{\sigma}(x) - S^{(k-1)}_{\sigma}(x)| : \sigma\in\Sigma, x\in\CC(\rho(\sigma),R(\sigma))\right\}
\leq
\frac{1}{2^k}.
\end{equation}
For each $i\in\{1,\ldots,k\}$, the family $\B^{(k)}_{i}$ and its corresponding family of increasing maps $\{\lambda'_B\}_{B\in\B^{(k)}_{i}}$ satisfy the conclusion of Lemma \ref{lemma:findNonsingZeros} applied to the $\S^{(k)}$-polynomial map $P^{(k)}_{i}:D_i\to\RR^{\dim(D_i)}$ and the number $\delta = 2^{-k}$.  For every $i\in\{1,\ldots,k-1\}$ and every point $a^{(k-1)}$ realized by a box $B^{(k-1)}$ in $\B^{(k-1)}_{i}$, the statement $\IF_\emptyset(P^{(k)}_{i};B^{(k-1)})$ holds; if we let $a^{(k)}$ be the unique zero of $P^{(k)}_{i}$ in $B^{(k-1)}$, then $a^{(k)}$ is realized by a box $B^{(k)}$ in $\B^{(k)}_{i}$ such that $B^{(k)} \subseteq B^{(k-1)}$ and $\lambda'_{B^{(k)}} = \lambda'_{B^{(k-1)}}$.  We call $a^{(k)}$ the realized point of $\B^{(k)}_{i}$ {\bf associated} to $a^{(k-1)}$ and call $B^{(k)}$ the box in $\B^{(k)}_{i}$ {\bf associated} to $B^{(k-1)}$; we shall use the word ``associated'' transitively.  Finally, for each $i\in\{1,\ldots,k\}$ and $B\in\B^{(k)}_{i}$,
\begin{equation}\label{eq:forceGeneric}
q_{n(i),j}\circ\Pi_{\lambda_B}(x,y) \neq 0
\quad\text{for all $(x,y)\in B$ and $j\in\{1,\ldots,k\}$,}
\end{equation}
where $\lambda_B:\to\{1,\ldots,n(i)\}\to\{1,\ldots,m(i)+n(i)\}$ is complementary to $\lambda'_B$ relative to $D_i$.
}\\

Suppose we have constructed $\{\S^{(k)}\}_{k\in\NN}$, along with the corresponding $\B^{(k)}_{i}$'s and $\lambda'_B$'s, satisfying Property($k$) for each $k\in\NN$.  Note that for each $k\in\NN$, $\S^{(k)}-\S$ is a family of polynomial functions (all but finitely many of which are zero).  It follows from \eqref{eq:forceCauchy} that $\{\S^{(k)}-\S\}_{k\in\NN_+}$ is a computably Cauchy sequence of members of $\Comp_{\an}(\rho,R)$, and hence converges computably to a function in $\Comp_{\an}(\rho,R)$ by Lemma \ref{lemma:compCauchy}.  Define
\[
\T = \lim_{k\to\infty}\S^{(k)} = \S + \lim_{k\to\infty}(\S^{(k)} - \S),
\]
and note that $\T\in\Comp_{\C}(\rho)$ since $\C$ contains the IF-system of all computably analytic functions.  Lemma \ref{lemma:compConvDiff} (just the noneffective content) implies that $\lim_{k\to\infty}\PDn{\alpha}{S^{(k)}_{\sigma}}{x} = \PDn{\alpha}{T_\sigma}{x}$ on $[-\rho(\sigma),\rho(\sigma)]$ for all $(\sigma,\alpha)\in\Delta(\rho)$.  Therefore $\T$ satisfies \eqref{eq:boundT} since $\S^{(k)}\in\Ball(\S,\epsilon)$ for all $k\in\NN$.

We now show that $\T$ is generic.  To do this, we shall quote Lemma \ref{lemma:compConvDiff} and Proposition \ref{prop:contIFT}, but here we are only using the noneffective content of these statements since genericity has nothing to do with computability.  Fix a name $(p(x,y),\sigma,\alpha,\xi,\name(D)) \in \Delta_{m,n}^{0}(\rho)$ for a $\T$-polynomial map $P = p\circ F:D\to\RR^{\dim(D)}$ and a nonsingular zero $a$ of $P$ satisfying the $(\sigma,\alpha,\xi)$-distinctness condition.  By Remark \ref{rmk:IFsection}.3, there exists a bounded rational box manifold $A$ which contains $a$, is open in $D$, $\cl(A)\subseteq D$, and $\IF_\emptyset(P;A)$ holds.  For each $k\in\NN$, write $P^{(k)} = p\circ F^{(k)} : D\to\RR^{\dim(D)}$ for the $\S^{(k)}$-polynomial map named by $(p(x,y),\sigma,\alpha,\xi,\name(D))$.  Since $\lim_{k\to\infty} \partial^1(P^{(k)}) = \partial^1(P)$ on $\cl(A)$ by Lemma \ref{lemma:compConvDiff}, Proposition \ref{prop:contIFT} implies that there exists $K\in\NN$ such that $\IF_\emptyset(P^{(k)};A)$ holds for all $k\geq K$.  For each $k\in\ K$, write $a^{(k)}$ for the unique zero of $P^{(k)}$ in $A$. Proposition \ref{prop:contIFT} also implies that $\lim_{k\to\infty} a^{(k)} = a$.  We may fix $\delta > 0$ such that $a\in D_{2\delta}$, $|\det\PD{}{P}{x}(x)| > \delta$ for all $x\in D$ with $\|x-a\|\leq \delta$, and $2\delta$ witnesses that $a$ satisfies the $(\sigma,\alpha,\xi)$-distinctness condition.  Therefore by making $K$ larger, we may obtain that for all $k\geq K$, the point $a^{(k)}$ is realized by some box $B^{(k)} = B'^{(k)} \times B''^{(k)}$ in $\B^{(k)}_{i}$, where $K\geq i$ and $(p(x,y),\sigma,\alpha,\xi,\name(D))$ is the $i$th member of the sequence \eqref{eq:namesEnum}.  Again, for $K$ sufficiently large, we have
\[
A \supseteq B'^{(K)} \supseteq B'^{(K+1)} \supseteq B'^{(K+2)} \supseteq \cdots
\]
and
\[
\bigcap_{k\geq K}B^{(k)} = \{F(a)\},
\]
and also
\[
q_{n,i}\circ\Pi_{\lambda}(x,y) \neq 0
\quad\text{on $B^{(k)}$, for all $k\geq K$ and $i\in\{1,\ldots,k\}$,}
\]
where $\lambda':\{1,\ldots,\dim(D)\}\to\{1,\ldots,m+n\}$ is the common increasing map associated with each of the boxes $B^{(k)}$ (for all $k\geq K$), and $\lambda:\{1,\ldots,n\}\to\{1,\ldots,m+n\}$ is complementary to $\lambda'$ relative to $D$.  Therefore $q\circ\Pi_\lambda\circ F(a) \neq 0$ for all $q\in\QQ[x_1,\ldots,x_n]\setminus\{0\}$, which shows that $\td_{\QQ} \QQ(F(a)) \geq n$, and hence that
\[
\td_{\QQ}\QQ(F(a)) = n
\]
by Lemma \ref{lemma:tdBounds}.1.  This proves the theorem, assuming we can construct the sequence $\{\S^{(k)}\}_{k\in\NN}$ with $\S^{(k)}$ satisfying Property($k$) for all $k\in\NN$.\\

To construct this sequence, let $k\in\NN$, and inductively assume that for each $l\in\{0,\ldots,k\}$ we have constructed $\S^{(l)}$, along with $(\B^{(l)}_{1},\ldots,\B^{(l)}_{l})$ and the corresponding $\lambda'_B$'s, satisfying Property($l$).  To construct $\S^{(k+1)}$, begin by applying Lemma \ref{lemma:findNonsingZeros} to each of the functions $P^{(k)}_{1},\ldots,P^{(k)}_{k+1}$ with $\delta = 2^{-(k+1)}$.  Denote the resulting families of boxes by $\B^{(k,0)}_{1},\ldots,\B^{(k,0)}_{k+1}$.  By relying on our previous work, and by shrinking boxes when necessary by using the comments in the last paragraph of the Remarks \ref{rmk:IFsection}, we can assume that for each $i\in\{1,\ldots,k\}$, every point realized by $\B^{(k)}_{i}$ is realized by $\B^{(k,0)}_{i}$, say by boxes $B\in\B^{(k)}_{i}$ and $\tld{B}\in\B^{(k,0)}_{i}$ with $\tld{B}\subseteq B$ and $\lambda'_{\tld{B}} = \lambda'_B$.  Suppose there are exactly $d$ points realized by the families $\B^{(k,0)}_{1},\ldots,\B^{(k,0)}_{k+1}$, and call these points $a_1,\ldots,a_d$.  We will construct $\S^{(k+1)}$ by first constructing a sequence $\S^{(k,0)},\ldots,\S^{(k,d)}$ of members of $\Comp_{\C}(\rho)$, along with the associated $\B^{(k,0)}_{i}$'s,\ldots,$\B^{(k,d)}_{i}$'s and $\lambda'_B$'s, satisfying Property$(k,l)$ for all $l\in\{0,\ldots,d\}$.\\

\noindent{\bf Property(k,0).}
{\slshape
We have $\S^{(k,0)} = \S^{(k)}$, and $\B^{(k,0)}_{1},\ldots,\B^{(k,0)}_{k+1}$ and the $\lambda'_B$'s are defined above.
}\\

\noindent{\bf Property(k,l) for $\mathbf{0 < l \leq d}$.}
{\slshape
We have that $\S^{(k,l)} \in \Ball(\S,\epsilon)$, that $\S^{(k,l)} - \S^{(k,l-1)}$ is a sequence of polynomial functions with rational coefficients, that $\S^{(k,l)}_{\sigma} = \S^{(k,l-1)}_{\sigma}$ for all $\sigma\in\Sigma \setminus \bigcup_{i=1}^{k+1} \im(\sigma_i)$, and that
\[
\dist(\S^{(k,l)},\S^{(k,l-1)}) \leq \frac{1}{d 2^{k+1}}.
\]
For each $i\in\{1,\ldots,k+1\}$, the families $\B^{(k,l)}_{i}$ and $\B^{(k,l-1)}_{i}$ realize the same number of points, with each box $\tld{B}$ in $\B^{(k,l)}_{i}$ being a subset of a unique box $B$ in $\B^{(k,l-1)}_{i}$, and $\lambda'_{\tld{B}} = \lambda'_B$.  Thus we have a notion of association, and for each $i\in\{1,\ldots,d\}$ we shall write $a_{i}^{(l)}$ for the point realized by $\B^{(k,l)}$ associated to $a_i$.  Finally, the following hold for each $i\in\{1,\ldots,k+1\}$ and $B\in\B^{(k,l)}_{i}$:
\begin{enumerate}{\setlength{\itemsep}{3pt}
\item
If $i\in\{1,\ldots,k\}$ and the point realized by $B$ is associated to a point realized by $\B^{(k)}_{i}$, then
\[
\text{$q_{n(i),j}\circ\Pi_{\lambda_B}(x,y) \neq 0$
for all $j\in\{1,\ldots,k\}$ and $(x,y)\in B$.}
\]

\item
If the point realized by $B$ is among the points $\{a^{(l)}_{1},\ldots,a^{(l)}_{l}\}$, then
\[
\text{$q_{n(i),j}\circ\Pi_{\lambda_B}(x,y) \neq 0$
for all $j\in\{1,\ldots,k+1\}$ and $(x,y)\in B$.}
\]
}\end{enumerate}
}

Once we have constructed $\S^{(k,0)},\ldots,\S^{(k,d)}$, along with the corresponding families of boxes and names of coordinate projections, satisfying Property$(k,l)$ for all $l\in\{0,\ldots,d\}$, then we define $\S^{(k+1)} = \S^{(k,d)}$ and $\B^{(k+1)}_{i} = \B^{(k,d)}_{i}$ for all $i\in\{1,\ldots,k+1\}$.  Observe that we have achieved Property$(k+1)$.  This finishes the proof, up the the construction of $\S^{(k,0)},\ldots,\S^{(k,d)}$.
\\

To construct $\S^{(k,0)},\ldots,\S^{(k,d)}$, we let $l\in\{0,\ldots,d-1\}$, and inductively assume that we have constructed $\S^{(k,0)},\ldots,\S^{(k,l)}$, along with their corresponding families of boxes and names of coordinate projections, satisfying Property$(k,i)$ for all $i\in\{0,\ldots,l\}$.  To complete the proof we must construct $\S^{(k,l+1)}$, and we do this by concentrating on the point $a^{(l)}_{l+1}$.  The point $a^{(l)}_{l+1}$ is realized by a box $B$ in $\B^{(k,l)}_{i}$ for some $i\in\{1,\ldots,k+1\}$.  Now, because we are
concentrating on the point $a^{(l)}_{l+1}$, the box $B$, and the $\S^{(k,l)}$-polynomial map  $P^{(k,l)}_{i}:D_i\to\RR^{\dim(D_i)}$
named by $(p_i(x,y),\sigma_i,\alpha_i,\xi_i,\name(D_i))$, to reduce
clutter in notation, and to free up the index $i$ for other use, we shall
use the following notational abbreviation.\\

\noindent\emph{Notational Abbreviation $(\star)$}.  We write $a$ for
the point $a^{(l)}_{l+1}$, write $m$ and $n$ for $m(i)$ and $n(i)$, write $P =
p\circ F:D\to\RR^{\dim(D)}$ and $F(x) = (x,f(x))$ for $P^{(k,l)}_{i} =
p_i\circ F^{(k,l)}_{i}:D_i\to\RR^{\dim(D_i)}$ and $F^{(k,l)}_{i}(x)
= (x,f^{(k,l)}_{i}(x))$, and write $(p(x,y),\sigma,\alpha,\xi,\name(D))$ for
$(p_i(x,y),\sigma_i,\alpha_i,\xi_i,\name(D_i))$.  We still write $B$ for the box realizing $a$, and write $B = B'\times B''$ with $B'\subseteq D$ and $B''\subseteq\RR^n$.  We also let $E\subseteq\{1,\ldots,m\}$ be such that $D = U\times\{u\}$ with $U$ open in $\RR^E$ and $u\in\QQ^{E^c}$.\\

Suppose the image of the map $\sigma$ consists of exactly $r$ members of $\Sigma$, which we call $\sigma_1,\ldots,\sigma_r$.  For each $i\in\{1,\ldots,r\}$, let $\Gamma(i)$ be the $\approx_\sigma$-equivalence class given by
\[
\Gamma(i) = \{j\in\{1,\ldots,n\} : \sigma(j) = \sigma_i\}.
\]
Put $s(i) = |\Gamma(i)|$ for each $i$, and fix a bijection
\[
\Gamma:\{(i,j)\in\NN^2 : 1\leq i\leq r, 1\leq j\leq s(i)\} \to \{1,\ldots,n\}
\]
such that $\Gamma(i) = \{\Gamma(i,1),\ldots,\Gamma(i,s(i))\}$ for each $i\in\{1,\ldots,r\}$.  The box $B$ satisfies the $(\sigma,\alpha,\xi)$-distinctness condition, so the point $a$ does too.  Therefore for each $i\in\{1,\ldots,r\}$ and distinct $j_1,j_2\in\{1,\ldots,s(i)\}$ such that $\alpha(j_1) = \alpha(j_2)$,
\[
\Pi_{\xi\circ\Gamma(i,j_1)}(a) \neq \Pi_{\xi\circ\Gamma(i,j_2)}(a).
\]
It therefore follows from Lemma \ref{lemma:pertubationPoly} that for each $i\in\{1,\ldots,r\}$, we can effectively find polynomials $\Phi_{i,1},\ldots,\Phi_{i,s(i)} \in \QQ[z_i]$, where $z_i$ is an $\eta(\sigma_i)$-tuple of variables, such that the matrix
\begin{equation}\label{eq:nonsingDet}
\left(
\PDn{\alpha\circ\Gamma(i,j_1)}{\Phi_{i,j_2}}{z_i}
(\Pi_{\xi\circ\Gamma(i,j_1)}(a))
\right)_{(j_1,j_2)\in\{1,\ldots,s(i)\}^2}
\end{equation}
is invertible.  For each $i\in\{1,\ldots,r\}$, let $w_i = (w_{i,1},\ldots,w_{i,s(i)})$ and
\[
\Phi_i(z_i,w_i) = \sum_{j=1}^{s(i)} w_{i,j}\Phi_{i,j}(z_i).
\]
Put $w = (w_1,\ldots,w_r)$ and
\[
\Phi(x,w) = \left(\Phi_i(\Pi_{\xi\circ\Gamma(i,j)}(x), w_i)\right)_{
1\leq i\leq r, 1\leq j\leq s(i)},
\]
where $\Phi_i(\Pi_{\xi\circ\Gamma(i,j)}(x), w_i)$ is the $\Gamma(i,j)$-th component of $\Phi(x,w)$.

The functions $\Phi$ allows us to define a parameterized form of $\S^{(k,l)}$.  Let $\S^{(k,l)}_{\Phi} = \{\S^{(k,l)}_{\Phi,\varsigma}\}_{\varsigma\in\Sigma}$, where $\S^{(k,l)}_{\Phi,\varsigma} : [-\rho(\varsigma),\rho(\varsigma)] \times \RR^n\to\RR$ is defined by
\[
\S^{(k,l)}_{\Phi,\varsigma}(y,w)
=
\begin{cases}
S^{(k,l)}_{\sigma_i}(y) + \Phi_i(y,w_i),
    & \text{if $\varsigma = \sigma_i$ for some $i\in\{1,\ldots,r\}$,}
\\
S^{(k,l)}_{\varsigma}(y),
    & \text{if $\varsigma\in\Sigma\setminus\{\sigma_1,\ldots,\sigma_r\}$,}
\end{cases}
\]
where $y$ denotes an $\eta(\varsigma)$-tuple of variables.  For each $w\in\RR^n$, write $\S^{(k,l)}_{\Phi}(w) = \{\S^{(k,l)}_{\Phi,\varsigma}(w)\}_{\varsigma\in\Sigma}$, where $\S^{(k,l)}_{\Phi,\varsigma}(w):[-\rho(\varsigma),\rho(\varsigma)]\to\RR $ is the map $y\mapsto \S^{(k,l)}_{\Phi,\varsigma}(y,w)$.
\\

For the moment, let us depart from Notational Abbreviation $(\star)$.  Consider any name $(p(x,y),\sigma,\alpha,\xi,\name(D)) \in \Delta^{0}_{m,n}(\rho)$ (not just our particular name of interest from $(\star)$, but any name).  In the natural way, $(p(x,y),\sigma,\alpha,\xi,\name(D))$ names both an $\S^{(k,l)}$-polynomial map $P = p\circ F:D\to\RR^{\dim(D)}$, with $F(x) = (x,f(x))$, and also an $\S^{(k,l)}_{\Phi}$-polynomial map $P_\Phi = p\circ F_\Phi:D\times\RR^n\to\RR^{\dim(D)}$, with $F_\Phi(x,w) = (x,f_\Phi(x,w))$, where
\[
f_\Phi(x,w) =
\left(
\PDn{\alpha(1)}{S^{(k,l)}_{\Phi,\sigma(1)}}{x}(\Pi_{\xi(1)}(x),w), \ldots,
\PDn{\alpha(n)}{S^{(k,l)}_{\Phi,\sigma(n)}}{x}(\Pi_{\xi(n)}(x),w)
\right).
\]
Because $\PDn{\beta}{\Phi}{x}(x,0) = 0$ for all $\beta\in\NN^n$, we have $f_\Phi(x,0) = f(x)$, $F_\Phi(x,0) = F(x)$, and $P_\Phi(x,0) = P(x)$.  Thus $\S^{(k,l)}_{\Phi}(0) = \S^{(k,l)}$.  Note that for any $b\in\QQ^n$, $\S^{(k,l)}_{\Phi}(b)$ is a member of $\Comp_{\C}(\rho)$.  Also note that there are only finitely many $\sigma\in\Sigma$ such that $\S^{(k,l)}_{\Phi,\sigma}(w)$ differs from $\S^{(k,l)}_{\sigma}$, any difference is a rational polynomial, and polynomials have only finitely many nonzero derivatives.  It follows that the map $b\mapsto \S^{(k,l)}_{\Phi}(b)$ is a continuous map from $\QQ^n$ into $\Comp_{\C}(\rho)$, and in fact, computably so:  given any $b\in\QQ^n$ and computable map $\delta:\Delta(\rho)\to\QQ_+$, we can effectively find an open rational box $W\subseteq\RR^n$ containing $b$ such that $\S^{(k,l)}_{\Phi}(W\cap\QQ^n) \subseteq \Ball(\S^{(k,l)}_{\Phi}(b),\delta)$.  It follows that we can effectively find an open rational box $W\subseteq\RR^n$ containing the origin such that the following hold for all $b\in W\cap\QQ^n$:
\begin{enumerate}{\setlength{\itemsep}{3pt}
\item
The family $\S^{(k,l)}_{\Phi}(b)$ has Property$(k,l)$, for same families of boxes and coordinate projections as $\S^{(k,l)}$.  (Proposition \ref{prop:contIFT} is being used here, in full strength.)

\item
$\displaystyle \dist(S^{(k,l)}_{\Phi}(b), \S^{(k,l)}) \leq \frac{1}{d 2^{k+1}}$.

\item
$\S^{(k,l)}_{\Phi}(b) \in \Ball(\S,\epsilon)$.\\
}\end{enumerate}

We now return to Notational Abbreviation $(\star)$.  Because $P_\Phi(a,0) = P(a) = 0$ and $\det\PD{}{P_\Phi}{x_E}(a,0) = \det\PD{}{P}{x_E}(a) \neq 0$, Remark \ref{rmk:IFsection}.4 shows that by shrinking $W$, in an effective manner, we can find a bounded rational box manifold $V$ which is open in $D\times\RR^n$, contains $(b,0)$, $W = \{w : (x,w)\in V\}$, and $\IF_\Lambda(P_\Phi;V)$ holds, where $\Pi_\Lambda:\RR^{m+n}\to\RR^n$ is the coordinate projection $\Pi_\Lambda(x,w) = w$.  Let $\varphi:W\to V$ be the section of the projection $\Pi_\Lambda$ implicitly defined by $P_\Phi\circ\varphi(x,w) = 0$ on $W$.  Now, the map $\varphi:W\to\varphi(W)$ is a $\C$-analytic isomorphism of $n$ dimensional manifolds, and
\[
\PD{}{F_\Phi}{(x,w)}(a,0)
=
\left(\begin{matrix}
\id             & 0 \\
\PD{}{f}{x}(a)  & A
\end{matrix}\right),
\]
where up to permuting the rows of $A$, the matrix $A$ can be written in the block diagonal form
\[
A = \left(\begin{array}{cccc}
A_1    & 0      & \cdots & 0      \\
0      & A_2    & \ddots & \vdots \\
\vdots & \ddots & \ddots & 0      \\
0      & \cdots & 0      & A_r
\end{array}\right),
\]
where for each $i\in\{1,\dots,r\}$, $A_i$ is the matrix \eqref{eq:nonsingDet}.  Therefore $\PD{}{F_\Phi}{(x,w)}(a,0)$ is invertible.  So by Remark \ref{rmk:IFsection}.4 again, we may shrink $W$ in an effective manner so that $F_\Phi\circ\varphi:W\to D\times\RR^n$ is a $\C$-analytic isomorphism onto its image.  Its image is contained in $\{(x,y)\in B : p(x,y) = 0\}$, which is an $n$-dimensional manifold, so $F_\Phi\circ\varphi(W)$ is an open submanifold of $\{(x,y)\in B : p(x,y) = 0\}$.  The projection $\Pi_{\lambda_B}:\RR^{m+n}\to\RR^n$ is a $\C$-analytic isomorphism from $\{(x,y)\in B : p(x,y) = 0\}$ onto an open subset of $\RR^n$, so the map $\Pi_{\lambda_B}\circ F_\Phi\circ\varphi :W\to\RR^n$ is a $\C$-analytic isomorphism onto its image, which is open in $\RR^n$.  Therefore the set
\[
\bigcap_{i=0}^{k+1}\{w\in W : q_{n,i}\circ\Pi_{\lambda_B} \circ F_\Phi \circ \varphi(w) \neq 0\}
\]
is dense and c.e.\ open in $W$, with the computable content of this statement following from the fact that $q_{n,i}\circ\Pi_{\lambda_B} \circ F_\Phi \circ \varphi$ is computably continuous.  So we can effectively find $b\in \QQ^n \cap W$ such that
\[
q_{n,i}\circ\Pi_{\lambda_B} \circ F_\Phi \circ \varphi(b) \neq 0
\quad\text{for all $i\in\{1,\ldots,k+1\}$,}
\]
and thereby can find a sufficiently small compact rational box $\tld{B}\subseteq B$ which is a neighborhood of $F_\Phi\circ\varphi(b)$ in $D\times\RR^n$ such that
\[
q_{n,i}\circ\Pi_{\lambda_B}(x,y) \neq 0
\quad\text{for all $(x,y)\in \tld{B}$ and $i\in\{1,\ldots,k+1\}$.}
\]
Therefore we satisfy Property$(k,l+1)$ if we set $\S^{(k,l+1)} = \S^{(k,l)}_{\Phi}(b)$, replace $B$ by $\tld{B}$ to form the family $\B^{(k,l+1)}_{i}$ from $\B^{(k,l)}_{i}$ (where $(p(x,y),\sigma,\alpha,\xi,\name(D))$ is the $i$th member of the sequence \eqref{eq:namesEnum}), set $\lambda'_{\tld{B}} = \lambda'_B$, and keep everything else the same.  This completes the proof.
\end{proof}

%% file: mainDec2.bbl
\providecommand{\bysame}{\leavevmode\hbox to3em{\hrulefill}\thinspace}
\providecommand{\MR}{\relax\ifhmode\unskip\space\fi MR }
\providecommand{\MRhref}[2]{%
  \href{http://www.ams.org/mathscinet-getitem?mr=#1}{#2}
}
\providecommand{\href}[2]{#2}
\begin{thebibliography}{1}

\bibitem{BW}
T.~Becker and V.~Weispfennig, \emph{Gr{\"{o}}bner {B}ases: a computational
  approach to commutative algebra}, Springer-Verlag, 1993.

\bibitem{CLO}
D.~Cox, J.~Little, and D.~O'Shea, \emph{Ideals, {V}arieties, and {A}lgorithms :
  an introduction to computational algebraic geometry and commutative algebra},
  Springer-Verlag, 2005.

\bibitem{JonesServi}
G.~Jones and T.~Servi, \emph{On the decidability of the real field with a
  generic power function}, 2010, preprint.

\bibitem{Lorentz}
R.~Lorentz, \emph{Multivariate birkhoff interpolation}, Lecture Notes in
  Mathematics, Springer-Verlag, 1992.

\bibitem{MW}
A.~Macintyre and A.~J. Wilkie, \emph{On the decidability of the real
  exponential field}, Kreiseliana: About and around Georg Kreisel, A. K.
  Peters, 1996, pp.~441--467.

\bibitem{DJMcharDec}
D.~J. Miller, \emph{Characterizing decidability in a quasianalytic setting},
  preprint, 2010.

\bibitem{Severi}
F.~Severi and E.~L{\"{o}}ffler, \emph{Vorlesungen {\"{u}}ber algebraische
  geometrie}, Teubner, Berlin, 1921.

\bibitem{Tarski}
A.~Tarski, \emph{A decision method for elementary algebra and geometry}, Tech.
  report, RAND Corporation, Berkeley and Los Angeles, 1951, second edition
  revised.

\end{thebibliography}
